\documentclass[10pt]{article}%
\usepackage{makeidx}
\usepackage{amsfonts}
\usepackage{amsmath}
\usepackage{amssymb}
\usepackage{graphicx}
\usepackage{bbold}
\usepackage{palatino}
\usepackage{extarrows}
\usepackage[colorlinks]{hyperref}%
\providecommand{\U}[1]{\protect\rule{.1in}{.1in}}
\providecommand{\U}[1]{\protect\rule{.1in}{.1in}}
\providecommand{\U}[1]{\protect\rule{.1in}{.1in}}
\hypersetup{colorlinks=blue, linkcolor=cyan, citecolor=green,
filecolor=black, urlcolor=blue }
\newtheorem{theorem}{Theorem}

\newtheorem{corollary}[theorem]{Corollary}

\newtheorem{definition}[theorem]{Definition}
\newtheorem{example}[theorem]{Example}

\newtheorem{lemma}[theorem]{Lemma}
\newtheorem{notation}[theorem]{Notation}

\newtheorem{proposition}[theorem]{Proposition}
\newtheorem{remark}[theorem]{Remark}

\newenvironment{proof}[1][Proof]{\noindent\textbf{#1.} }{\ \rule{0.5em}{0.5em}}

\renewcommand{\thefootnote}{\fnsymbol{footnote}}
\begin{document}

\title{$L^{p}$--Variational Solution of Backward Stochastic Differential Equation
driven by subdifferential operators on a deterministic interval time}
\author{Aurel R\u{a}\c{s}canu$\bigskip$\\{\small \textquotedblleft Octav Mayer\textquotedblright\ Institute of
Mathematics of the Romanian Academy,}\\{\small Carol I Blvd., no. 8, 700506, Ia\c{s}i, Romania}}
\maketitle

\begin{abstract}
Our aim is to study the existence and uniqueness of the $L^{p}$--variational
solution, with $p>1,$ of the following multivalued backward stochastic
differential equation with $p$--integrable data:%
\[
\left\{
\begin{array}
[c]{r}%
-dY_{t}+\partial_{y}\Psi\left(  t,Y_{t}\right)  dQ_{t}\ni H\left(
t,Y_{t},Z_{t}\right)  dQ_{t}-Z_{t}dB_{t},\;t\in\left[  0,T\right]
,\smallskip\\
\multicolumn{1}{l}{Y_{T}=\eta,}%
\end{array}
\,\right.
\]
where $Q$ is a progresivelly measurable increasing continuous stochastic
process and $\partial_{y}\Psi$ is the subdifferential of the convex lower
semicontinuous function $y\mapsto\Psi\left(  t,y\right)  .$

In the framework of \cite{ma-ra/15} (the case $p\geq2$), the strong solution
found it there is the unique variational solution, via the uniqueness property
proved in the present article.

\end{abstract}


AMS Classification subjects: 60H10, 60F25, 47J20, 49J40.\medskip

Keywords: Backward stochastic differential equations; Subdifferential
ope\-rators; Stochastic variational inequalities; $p$--integrable data

\footnotetext{{\scriptsize E--mail address: \texttt{aurel.rascanu@uaic.ro}}} \renewcommand{\thefootnote}{\arabic{footnote}}

\section{Introduction}

The study of the standard backward stochastic differential equations (BSDEs)
was initiated by E. Pardoux and S. Peng in \cite{pa-pe/90}. The authors have
proved the existence and the uniqueness of the solution for the BSDE on fixed
time interval, under the assumption of Lipschitz continuity of the generator
$F$ with respect to $y$ and $z$ and square integrability of $\eta$ and
$F\left(  t,0,0\right)  $. The case of BSDEs on random time interval have been
treated by R.W.R. Darling and E. Pardoux in \cite{da-pa/97}, where it is
obtained, as application, the existence of a continuous viscosity solution to
the elliptic partial differential equations (PDEs) with Dirichlet boundary
conditions. The more general case of reflected BSDEs was considered for the
first time by N. El Karoui et al. in \cite{ka-ka-pa/97}.

In the present paper, we prove the existence and uniqueness of a new type of
solution, called $L^{p}$--variational solution, in the case $p>1,$ of the
genera\-li\-zed backward stochastic variational inequality (BSVI for short)
with $p$--integrable data:%
\begin{equation}
\left\{
\begin{array}
[c]{r}%
\displaystyle Y_{t}+{\int\nolimits_{t}^{T}}dK_{s}=\eta+{\int\nolimits_{t}^{T}%
}\left[  F\left(  s,Y_{s},Z_{s}\right)  ds+G\left(  s,Y_{s}\right)
dA_{s}\right]  -{\int\nolimits_{t}^{T}}Z_{s}dB_{s}\,,\quad t\in\left[
0,T\right]  ,\medskip\\
\multicolumn{1}{l}{dK_{t}\in\partial\varphi\left(  Y_{t}\right)
dt+\partial\psi\left(  Y_{t}\right)  dA_{t}\,,\quad\text{on }\left[
0,T\right]  ,}%
\end{array}
\right.  \label{GBSVI 1}%
\end{equation}
where $\partial\varphi$ and $\partial\psi$ are the subdifferentials of two
convex lower semicontinuous functions $\varphi$ and $\psi$ and $\left\{
A_{t}:t\geq0\right\}  $ is a progressively measurable increasing continuous
stochastic process.

The existence is obtained using the Moreau--Yosida regularization of $\varphi$
and $\psi$ and the mollifier approximations of the generators $F$ and $G.$ The
proof of the existence and uniqueness in the case of a random time interval
and in the case $p=1$ are, for the moment, in work and there will be the
subjects of a future article.$\medskip$

In fact, we will define and prove the existence and uniqueness of the $L^{p}$
variational solution for an equivalent form of (\ref{GBSVI 1}):%
\begin{equation}
\left\{
\begin{array}
[c]{r}%
\displaystyle Y_{t}+{\int\nolimits_{t}^{T}}dK_{s}=\eta+{\int\nolimits_{t}^{T}%
}H\left(  s,Y_{s},Z_{s}\right)  dQ_{s}-{\int\nolimits_{t}^{T}}Z_{s}%
dB_{s}\,,\;t\in\left[  0,T\right]  \medskip\\
\multicolumn{1}{l}{dK_{t}\in\partial_{y}\Psi\left(  t,Y_{t}\right)
dQ_{t}\,,\;\text{on }\left[  0,T\right]  ,}%
\end{array}
\,\right.  \label{GBSVI 2}%
\end{equation}
with $Q,$ $H$ and $\Psi$ adequately defined.

The second condition in (\ref{GBSVI 1}) says, among others, that the first
component $Y$ of the solution is forced to stay in the set $\mathrm{Dom}%
\left(  \partial\varphi\right)  \cap\mathrm{Dom}\left(  \partial\psi\right)
.$ The role of $K$ is to act in the evolution of the process $Y$ and also to
keep $Y$ in these domains.

We mention that the presence of the process $A$ is justified by the possible
applications of equation (\ref{GBSVI 1}) in proving probabilistic proofs for
the existence of a solution of PDEs with Neumann boundary conditions on a
domain from $\mathbb{R}^{m}.$ The stochastic approach of the existence problem
for multivalued parabolic PDEs, was considered by L. Maticiuc and A.
R\u{a}\c{s}canu in \cite{ma-ra/10} and \cite{ma-ra/16}. We emphasize that if
the obstacles are fixed, the reflected BSDEs becomes a particular case of the
BSVI of type (\ref{GBSVI 1}), by taking $\varphi$ as convex indicator of the
interval defined by obstacles. In this case the solution of the BSVI belongs
to the domain of the multivalued operator $\partial\varphi$ and it is
reflected at the boundary of this domain.$\medskip$

The standard work on BSVI in the finite dimensional case is that of E. Pardoux
and A. R\u{a}\c{s}canu \cite{pa-ra/98}, where it is proved the existence and
uniqueness of the solution $\left(  Y,Z,K\right)  $ for the BSVI
(\ref{GBSVI 1}) with $A\equiv0$, under the following assumptions on $F$:
continuity with respect to $y$, monotonicity with respect to $y$ (in the sense
that $\langle y^{\prime}-y,F(t,y^{\prime},z)-F(t,y,z)\rangle\leq
\alpha|y^{\prime}-y|^{2}$), lipschitzianity with respect to $z$ and a
sublinear growth for $F\left(  t,y,0\right)  $. Moreover, it was shown that,
unlike the forward case, the process $K$ is absolute continuous with respect
to $dt$. In \cite{pa-ra/99} the same authors extend these results to the
Hilbert spaces framework.

We mention that assumptions of Lipschitz continuity of the generator $F$ with
respect to $y$ and $z$ and the square integrability of the final condition and
$F\left(  t,0,0\right)  $ (as in articles El Karoui et al. \cite{ka-ka-pa/97}
and E. Pardoux and S. Peng \cite{pa-pe/90}) are sometimes too strong for
applications (see, e.g., D. Duffie and L. Epstein \cite{du-ep/92} and El
Karoui et al. \cite{ka-pe-qu/97} for the applications in mathematical finance
and P. Briand et al. \cite{br-ca/00} and A. Rozkosz and L. S\l omi\'{n}ski
\cite{ro-sl/12b} for the applications to PDEs). A possibility is to weaken the
integrability conditions imposed on $\eta$ and $F$ or to weaken the assumption
which concerns the Lipschitz continuity of the generators. In P. Briand and R.
Carmona \cite{br-ca/00} or E. Pardoux \cite{pa/99} it is considered the case
where the generators are Lipschitz continuous with respect to $z$, continuous
with respect to $y$ and satisfy a monotonicity condition and a growth
condition of the type $\left\vert F\left(  t,y,z\right)  \right\vert
\leq\left\vert F\left(  t,0,z\right)  \right\vert +\phi\left(  \left\vert
y\right\vert \right)  ,$ where $\phi$ is a polynomial or even an arbitrary
positive increasing continuous function.

We recall that the previous assumption was used in \cite{pa/99} in order to
prove the existence of a solution in $L^{2}$. This result was generalized by
P. Briand et al. in \cite{br-de-hu-pa/03}, where it is proved the existence
and uniqueness of $L^{p}$ solutions, with $p\in\lbrack1,2]$, for BSDEs
considered with a random terminal time $T$: in the case $p\in(1,2],$ if
$\eta\in L^{p}$, $\displaystyle\int_{0}^{T}\left\vert F\left(  s,0,0\right)
\right\vert ds\in L^{p}$, for any $r>0$, $\displaystyle\int_{0}^{T}%
\sup_{\left\vert y\right\vert \leq r}\left\vert F\left(  s,y,0\right)
-F\left(  s,0,0\right)  \right\vert ds\in L^{1}$ and $F$ is Lipschitz
continuous with respect to $z$, continuous with respect to $y$ and satisfy a
monotonicity condition, then there exists a unique $L^{p}$ solution; in the
case $p=1$ similar result is proved if $T$ is a fixed deterministic terminal
time and under additional assumptions.

We also note that the study of the reflected BSDEs was the subject, e.g., of
the papers: J.P. Lepeltier et al. \cite{le-ma-xu/05} (in the case of the
general growth condition with respect to $y$ and for $p=2$), S. Hamad\`{e}ne
and A. Popier \cite{ha-po/12} (in the case of Lipschitz continuity with
respect to $y$ the and for $p\in\left(  1,2\right)  $). Studies made, roughly
speaking, under the assumptions of \cite{br-de-hu-pa/03} are, e.g.: A. Aman
\cite{am/09} (in the case of a generalized reflected BSDE and for $p\in\left(
1,2\right)  $), A. Rozkosz and L. S\l omi\'{n}ski \cite{ro-sl/12a} (for
$p\in\left[  1,2\right]  $) and T. Klimsiak \cite{kl/13} (in the case of BSDE
with two irregular reflecting barriers and for $p\in\left[  1,2\right]
$).\medskip

Our paper generalizes the existence and uniqueness results from
\cite{pa-ra/98} by considering the $L^{p}$ solutions in the case $p\in\left(
1,2\right)  ,$ the Lebesgue--Stieltsjes integral terms, and by assuming a
weaker boundedness condition for the generator $F$ (instead of the sublinear
growth):%
\begin{equation}
\mathbb{E}\Big(\int_{0}^{T}F_{\rho}^{\#}(s)ds\Big)^{p}<\infty,\quad\text{where
}F_{\rho}^{\#}\left(  t\right)  :=\sup\limits_{\left\vert y\right\vert
\leq\rho}\left\vert F(t,y,0)\right\vert . \label{local bound 2}%
\end{equation}
We remark that article \cite{ma-ra/15} concerns the same type of backward
equation as in our study (and under the similar assumptions), but considered
in the infinite dimensional framework and in the case $p\geq2.$ In addition,
it is worth pointing out that in the case $p\geq2,$ if we are in the framework
of \cite{ma-ra/15}, our variational solution is a strong one since we have
proved the uniqueness property of the variational solution.$\medskip$

In this paper we use the following notation: $(\Omega,\mathcal{F},\mathbb{P})$
is a complete probability space, the set $N_{\mathbb{P}}:=\{A\in
\mathcal{F}:\mathbb{P}\left(  A\right)  =0\}$, $\left\{  \mathcal{F}%
_{t}\right\}  _{t\geq0}$ is a right continuous and complete filtration
generated by a standard $k$--dimensional Brownian motion $\left(
B_{t}\right)  _{t\geq0}\,.$

$S_{m}^{p}\left[  0,T\right]  $ is the space of (equivalent classes of)
continuous progressively measurable stochastic processes (p.m.s.p. for short)
$X:\Omega\times\left[  0,T\right]  \rightarrow\mathbb{R}^{m}$ such that
$\mathbb{E}\sup_{t\in\left[  0,T\right]  }\left\vert X_{t}\right\vert
^{p}<+\infty,$ if $p>0.$

$\Lambda_{m}^{p}\left(  0,T\right)  $ is the space of p.m.s.p. $X:\Omega
\times\left(  0,T\right)  \rightarrow\mathbb{R}^{m}$ such that such that
$\int_{0}^{T}\left\vert X_{t}\right\vert ^{2}dt<+\infty$, $\mathbb{P}$--a.s.
if $p=0$ and $\mathbb{E}\left(  \int_{0}^{T}\left\vert X_{t}\right\vert
^{2}dt\right)  ^{p/2}<+\infty$, if $p>0.$\medskip

The article is organized as follows: next section is dedicated to the
presentation of the assumptions needed in our study. In the third section we
present firstly a intuitive introduction for the notion of $L^{p}%
$--variational solution and the we prove the uniqueness property. The fourth
section is devoted to the proof of the existence of our type of solution. The
Appendix contains, following \cite{pa-ra/14}, some results useful throughout
the paper.

\section{Assumptions and definitions\label{assumptions}}

In the beginning of this subsection we introduce the assumptions regarding
equation (\ref{GBSVI 1}).

Let $T>0.$ We consider throughout this paper that $p>1.$

\begin{itemize}
\item[\textrm{(A}$_{1}$\textrm{)}] The random variable $\eta:\Omega
\rightarrow\mathbb{R}^{m}$\ is $\mathcal{F}_{T}$--measurable such that
$\mathbb{E}\left\vert \eta\right\vert ^{p}<\infty$\ and $\left(  \xi
,\zeta\right)  \in S_{m}^{p}\left[  0,T\right]  \times\Lambda_{m\times k}%
^{p}\left(  0,T\right)  $\ is the unique pair associated to $\eta$\ given by
the martingale representation formula (see \cite[Corollary 2.44]{pa-ra/14})%
\begin{equation}
\xi_{t}=\eta-{\int_{t}^{T}}\zeta_{s}\,dB_{s}\,,\quad t\in\left[  0,T\right]
,\;\mathbb{P}\text{--a.s..} \label{mart repr th 1}%
\end{equation}

\item[\textrm{(A}$_{2}$\textrm{)}] The process $\left\{  A_{t}:t\geq0\right\}
$\ is a\ increasing and continuous p.m.s.p. such that $A_{0}=0$ and%
\[
\mathbb{E}\left(  e^{\alpha A_{t}}\right)  <\infty,\quad\text{for any }%
\alpha,t>0;
\]

\item[\textrm{(A}$_{3}$\textrm{)}] $\varphi,\psi:\mathbb{R}^{m}\rightarrow
\left[  0,+\infty\right]  $\ are proper convex lower semicontinuous (l.s.c.
for short) functions, $\partial\varphi$ and $\partial\psi$\ denote their
subdifferentials and we suppose that $0\in\partial\varphi\left(  0\right)
\cap\partial\psi\left(  0\right)  $\ (or equivalently $0=\varphi\left(
0\right)  \leq\varphi\left(  y\right)  $\ and $0=\psi\left(  0\right)
\leq\psi\left(  y\right)  $\ for all $y\in\mathbb{R}^{m}$).

\item[\textrm{(A}$_{4}$\textrm{)}] The functions $F:\Omega\times\mathbb{R}%
_{+}\times\mathbb{R}^{m}\times\mathbb{R}^{m\times k}\rightarrow\mathbb{R}^{m}%
$\ and $G:\Omega\times\mathbb{R}_{+}\times\mathbb{R}^{m}\rightarrow
\mathbb{R}^{m}$\ are such that $F\left(  \cdot,\cdot,y,z\right)  $, $G\left(
\cdot,\cdot,y\right)  $\ are p.m.s.p., for all $\left(  y,z\right)
\in\mathbb{R}^{m}\times\mathbb{R}^{m\times k},$ $F\left(  \omega,t,\cdot
,\cdot\right)  $, $G\left(  \omega,t,\cdot\right)  $ are continuous functions,
$d\mathbb{P}\otimes dt$-a.e. and, $\mathbb{P}$--a.s.,%
\begin{equation}
\int_{0}^{T}F_{\rho}^{\#}\left(  s\right)  ds+\int_{0}^{T}G_{\rho}^{\#}\left(
s\right)  dA_{s}<\infty,\quad\text{for all }\rho\geq0, \label{F, G assumpt 1}%
\end{equation}
where%
\begin{equation}
F_{\rho}^{\#}\left(  \omega,s\right)  :=\sup\nolimits_{\left\vert y\right\vert
\leq\rho}\left\vert F\left(  \omega,s,y,0\right)  \right\vert ,\quad G_{\rho
}^{\#}\left(  \omega,s\right)  :=\sup\nolimits_{\left\vert y\right\vert
\leq\rho}\left\vert G\left(  \omega,s,y\right)  \right\vert \,;
\label{def F sharp}%
\end{equation}

\item[\textrm{(A}$_{5}$\textrm{)}] Let%
\begin{equation}
n_{p}:=1\wedge\left(  p-1\right)  \quad\text{and}\quad\lambda\in\left(
0,1\right)  . \label{defnp}%
\end{equation}
Assume there exist three p.m.s.p. $\mu,\nu:\Omega\times\mathbb{R}%
_{+}\rightarrow\mathbb{R},$ $\ell:\Omega\times\mathbb{R}_{+}\rightarrow
\mathbb{R}_{+}\,,$ such that%
\begin{equation}
\mathbb{E}\exp\left(  p\int_{0}^{T}\left(  \left\vert \mu_{s}\right\vert
+\frac{1}{2n_{p}\lambda}\,\ell_{s}^{2}\right)  ds+p\int_{0}^{T}\left\vert
\nu_{s}\right\vert dA_{s}\right)  <\infty\label{ip-mnl}%
\end{equation}
and for all $t\in\left[  0,T\right]  ,$ $y,y^{\prime}\in\mathbb{R}^{m}$,
$z,z^{\prime}\in\mathbb{R}^{m\times k},$ $\mathbb{P}$--a.s.%
\begin{equation}%
\begin{array}
[c]{l}%
\left\langle y^{\prime}-y,F(t,y^{\prime},z)-F(t,y,z)\right\rangle \leq\mu
_{t}\,\left\vert y^{\prime}-y\right\vert ^{2},\medskip\\
\left\langle y^{\prime}-y,G(t,y^{\prime})-G(t,y)\right\rangle \leq\nu
_{t}\,\left\vert y^{\prime}-y\right\vert ^{2},\medskip\\
\left\vert F(t,y,z^{\prime})-F(t,y,z)\right\vert \leq\ell_{t}\,\left\vert
z^{\prime}-z\right\vert .
\end{array}
\label{F, G assumpt 2}%
\end{equation}

\end{itemize}

We define%
\[
Q_{t}\left(  \omega\right)  =t+A_{t}\left(  \omega\right)  ,
\]
and let $\left\{  \alpha_{t}:t\geq0\right\}  $\ be the real positive p.m.s.p.
such that $\alpha\in\left[  0,1\right]  $\ and $dt=\alpha_{t}dQ_{t}$ and
$dA_{t}=\left(  1-\alpha_{t}\right)  dQ_{t}.$

Let us introduce the functions%
\begin{equation}%
\begin{array}
[c]{l}%
\displaystyle H\left(  t,y,z\right)  :=\alpha_{t}F\left(  t,y,z\right)
+\left(  1-\alpha_{t}\right)  G\left(  t,y\right)  ,\medskip\\
\displaystyle\Psi\left(  \omega,t,y\right)  :=\alpha_{t}\left(  \omega\right)
\varphi\left(  y\right)  +\left(  1-\alpha_{t}\left(  \omega\right)  \right)
\psi\left(  y\right)  .
\end{array}
\label{def Phi}%
\end{equation}
Obviously, from (\ref{F, G assumpt 2}) we see that%
\begin{equation}%
\begin{array}
[c]{l}%
\left\langle y^{\prime}-y,H(t,y^{\prime},z)-H(t,y,z)\right\rangle \leq\left[
\mu_{t}\alpha_{t}+\nu_{t}\left(  1-\alpha_{t}\right)  \right]  \left\vert
y^{\prime}-y\right\vert ^{2},\medskip\\
\left\vert H(t,y,z^{\prime})-H(t,y,z)\right\vert \leq\alpha_{t}\ell
_{t}\,\left\vert z^{\prime}-z\right\vert .
\end{array}
\label{F, G assumpt 3}%
\end{equation}
Here and subsequently, $\lambda\in\left(  0,1\right)  $%
\begin{equation}
V_{t}\xlongequal{\hspace{-3pt}\textrm{def}\hspace{-3pt}}%
{\displaystyle\int_{0}^{t}}
\left(  \mu_{r}+\dfrac{1}{2n_{p}\lambda}\ell_{r}^{2}\right)  dr+%
{\displaystyle\int_{0}^{t}}
\nu_{r}dA_{r}\,. \label{defV_1}%
\end{equation}
and%
\[
V_{t}^{\left(  +\right)  }\xlongequal{\hspace{-3pt}\textrm{def}\hspace{-3pt}}%
{\displaystyle\int_{0}^{t}}
\left(  \mu_{r}+\dfrac{1}{2n_{p}\lambda}\ell_{r}^{2}\right)  ^{+}dr+%
{\displaystyle\int_{0}^{t}}
\nu_{r}^{+}\,dA_{r}\,.
\]
By the assumption (\ref{ip-mnl}) we clearly have%
\begin{equation}
\mathbb{E}~e^{pV_{T}}\leq\mathbb{E}~\sup_{t\in\left[  0,T\right]  }e^{pV_{t}%
}\leq\mathbb{E}e^{pV_{T}^{\left(  +\right)  }}<\infty. \label{exp-VT}%
\end{equation}

\begin{definition}
The notation $dK_{t}\in\partial_{y}\Psi\left(  t,Y_{t}\right)  dQ_{t}$ means
that $K$ is $\mathbb{R}^{m}$--valued locally bounded variation stochastic
process, $Q$ is a real increasing stochastic process, $Y$ is $\mathbb{R}^{m}%
$-valued continuous stochastic process such that $\int_{0}^{T}\Psi\left(
t,Y_{t}\right)  dQ_{t}<\infty$, a.s. and, $\mathbb{P}$--a.s., for any $0\leq
t\leq s\leq T,$%
\[%
{\displaystyle\int_{t}^{s}}
\left\langle y\left(  r\right)  -Y_{r},dK_{r}\right\rangle +%
{\displaystyle\int_{t}^{s}}
\Psi\left(  r,Y_{r}\right)  dQ_{r}\leq%
{\displaystyle\int_{t}^{s}}
\Psi\left(  r,y\left(  r\right)  \right)  dQ_{r},\quad\text{for any }y\in
C\left(  \mathbb{R}_{+};\mathbb{R}^{m}\right)  .
\]

\end{definition}

\begin{remark}
The condition $0\in\partial\varphi\left(  0\right)  \cap\partial\psi\left(
0\right)  $ does not restrict the generality of the problem, since from
$Dom\left(  \partial\varphi\right)  \cap Dom\left(  \partial\psi\right)
\neq\emptyset$ it follows that there exists $u_{0}\in Dom\left(
\partial\varphi\right)  \cap Dom\left(  \partial\psi\right)  $ and $\hat
{u}_{01}\in\partial\varphi\left(  u_{0}\right)  $, $\hat{u}_{02}\in
\partial\psi\left(  u_{0}\right)  $. In this case equation (\ref{GBSVI 1}) is
equivalent to%
\[
\left\{
\begin{array}
[c]{r}%
\hat{Y}_{t}+\displaystyle{\int_{t}^{T}}d\hat{K}_{s}=\eta+{\int_{t}^{T}%
}\big[\hat{F}(s,\hat{Y}_{s},\hat{Z}_{s})ds+\hat{G}(s,\hat{Y}_{s}%
)dA_{s}\big]-{\int_{t}^{T}}\hat{Z}_{s}dB_{s},\text{\ a.s.,}\smallskip\\
\multicolumn{1}{l}{d\hat{K}_{t}\in\partial\hat{\varphi}(\hat{Y}_{t}%
)dt+\partial\hat{\psi}(\hat{Y}_{t})dA_{t},\;\text{for all }t\in\left[
0,T\right]  ,}%
\end{array}
\,\right.
\]
where%
\[
\hat{Y}_{t}:=Y_{t}-u_{0}\,,\quad\hat{Z}_{t}:=Z_{t}\,,\quad\hat{\eta}%
:=\eta-u_{0}%
\]
and%
\[%
\begin{array}
[c]{l}%
\hat{F}\left(  s,y,z\right)  =F\left(  t,y+u_{0},z\right)  -\hat{u}%
_{01}\,,\quad\hat{G}\left(  s,y,z\right)  =G\left(  t,y+u_{0}\right)  -\hat
{u}_{02}\,,\medskip\\
\hat{\varphi}\left(  y\right)  =\varphi\left(  y+u_{0}\right)  -\left\langle
\hat{u}_{01},y\right\rangle \,,\quad\hat{\psi}\left(  y\right)  =\psi\left(
y+u_{0}\right)  -\left\langle \hat{u}_{02},y\right\rangle \,,\medskip\\
\partial\hat{\varphi}\left(  y\right)  =\partial\varphi\left(  y+u_{0}\right)
-\hat{u}_{01}\,,\quad\partial\hat{\psi}\left(  y\right)  =\partial\psi\left(
y+u_{0}\right)  -\hat{u}_{02}\medskip\\
\text{and}\medskip\\
d\hat{K}_{t}=dK_{t}-\hat{u}_{01}dt-\hat{u}_{02}dA_{t}\,.
\end{array}
\]

\end{remark}

Let $\varepsilon>0$ and the Moreau--Yosida regularization of $\varphi:$%
\begin{equation}
\varphi_{\varepsilon}\left(  y\right)  :=\inf\big\{\frac{1}{2\varepsilon
}\left\vert y-v\right\vert ^{2}+\varphi\left(  v\right)  :v\in\mathbb{R}%
^{m}\big\}, \label{fi-MY}%
\end{equation}
which is a $C^{1}$--convex function.

The gradient $\nabla\varphi_{\varepsilon}(x)=\partial\varphi_{\varepsilon
}\left(  x\right)  \in\partial\varphi\left(  J_{\varepsilon}\left(  x\right)
\right)  ,$ where $J_{\varepsilon}\left(  x\right)  :=x-\varepsilon
\nabla\varphi_{\varepsilon}(x)$ and satisfies%
\begin{equation}%
\begin{array}
[c]{ll}%
\left(  a\right)  & \left\vert J_{\varepsilon}\left(  x\right)
-J_{\varepsilon}\left(  y\right)  \right\vert \leq\left\vert x-y\right\vert
,\medskip\\
\left(  b\right)  & \left\vert \nabla\varphi_{\varepsilon}\left(  x\right)
-\nabla\varphi_{\varepsilon}\left(  y\right)  \right\vert \leq\dfrac
{1}{\varepsilon}\left\vert x-y\right\vert ,\medskip\\
\left(  c\right)  & \varphi_{\varepsilon}\left(  y\right)  =\dfrac{\left\vert
y-J_{\varepsilon}\left(  y\right)  \right\vert ^{2}}{2\varepsilon}%
+\varphi\left(  J_{\varepsilon}\left(  y\right)  \right)
\end{array}
\label{fi-lip}%
\end{equation}
and%
\begin{align}
-\left\langle u-v,\nabla\varphi_{\varepsilon}\left(  u\right)  -\nabla
\varphi_{\delta}\left(  v\right)  \right\rangle  &  \leq(\varepsilon
+\delta)\left\langle \nabla\varphi_{\varepsilon}(u),\nabla\varphi_{\delta
}(v)\right\rangle \label{fi-Cauchy}\\
&  \leq\dfrac{\varepsilon+\delta}{2}\Big[|\nabla\varphi_{\varepsilon}%
(u)|^{2}+|\nabla\varphi_{\delta}(v)|^{2}\Big]\nonumber
\end{align}
(for other useful inequalities see, e.g., \cite[inequalities $\left(
2.8\right)  $]{ma-ra/15}). Since $0\in\partial\varphi\left(  0\right)  $ we
deduce that%
\begin{equation}%
\begin{array}
[c]{l}%
0=\varphi\left(  0\right)  \leq\varphi\left(  J_{\varepsilon}\left(  u\right)
\right)  \leq\varphi_{\varepsilon}\left(  u\right)  \leq\varphi\left(
u\right)  ,\quad\text{for any }u\in\mathbb{R}^{m},\medskip\\
J_{\varepsilon}\left(  0\right)  =0,\quad\nabla\varphi_{\varepsilon
}(0)=0,\quad\text{and }\varphi_{\varepsilon}\left(  0\right)  =0.
\end{array}
\label{minimum point}%
\end{equation}
Also it holds that $\varphi\left(  J_{\varepsilon}u\right)  \leq
\varphi_{\varepsilon}\left(  u\right)  $, for any $u.$

We introduce the compatibility conditions between $\varphi,\psi$ and $F,G$.

\begin{itemize}
\item[\textrm{(A}$_{6}$\textrm{)}] For all $\varepsilon>0$, $t\in\left[
0,T\right]  $, $y\in\mathbb{R}^{m}$, $z\in\mathbb{R}^{m\times k}$%
\begin{equation}%
\begin{array}
[c]{rl}%
\left(  i\right)  & \left\langle \nabla\varphi_{\varepsilon}\left(  y\right)
,\nabla\psi_{\varepsilon}\left(  y\right)  \right\rangle \geq0,\medskip\\
\left(  ii\right)  & \left\langle \nabla\varphi_{\varepsilon}\left(  y\right)
,G\left(  t,y\right)  \right\rangle \leq\left\vert \nabla\psi_{\varepsilon
}\left(  y\right)  \right\vert \left\vert G\left(  t,y\right)  \right\vert
,\quad\mathbb{P}\text{--a.s.,}\medskip\\
\left(  iii\right)  & \left\langle \nabla\psi_{\varepsilon}\left(  y\right)
,F\left(  t,y,z\right)  \right\rangle \leq\left\vert \nabla\varphi
_{\varepsilon}\left(  y\right)  \right\vert \left\vert F\left(  t,y,z\right)
\right\vert ,\quad\mathbb{P}\text{--a.s..}%
\end{array}
\label{compAssumpt}%
\end{equation}

\end{itemize}

\begin{example}
$\quad$

\begin{enumerate}
\item[$\left(  a\right)  $] If $\varphi=\psi$ then the compatibility
assumptions (\ref{compAssumpt}) are clearly satisfied.

\item[$\left(  b\right)  $] Let $m=1$. Since $\nabla\varphi_{\varepsilon}$ and
$\nabla\psi_{\varepsilon}$ are increasing monotone functions on $\mathbb{R}$,
we see that, if $y\cdot G\left(  t,y\right)  \leq0$ and $y\cdot F\left(
t,y,z\right)  \leq0$, for all $t,y,z,$ then the compatibility assumptions
(\ref{compAssumpt}) are satisfied.

\item[$\left(  c\right)  $] Let $m=1$. If $\varphi,\psi:\mathbb{R}%
\rightarrow(-\infty,+\infty]$ are the convexity indicator functions
$\varphi\left(  y\right)  =\left\{
\begin{array}
[c]{rl}%
0, & \text{if\ }y\in\left[  a,b\right]  ,\smallskip\\
+\infty, & \text{if\ }y\notin\left[  a,b\right]  ,
\end{array}
\right.  $ and $\psi\left(  y\right)  =\left\{
\begin{array}
[c]{rl}%
0, & \text{if\ }y\in\left[  c,d\right]  ,\smallskip\\
+\infty, & \text{if\ }y\notin\left[  c,d\right]  ,
\end{array}
\right.  $ where $-\infty\leq a\leq b\leq+\infty$ and $-\infty\leq c\leq
d\leq+\infty$ are such that $0\in\left[  a,b\right]  \cap\left[  c,d\right]  $
(see (A$_{6}$)), then $\nabla\varphi_{\varepsilon}\left(  y\right)  =\dfrac
{1}{\varepsilon}[\left(  y-b\right)  ^{+}-\left(  a-y\right)  ^{+}]$, and
$\nabla\psi_{\varepsilon}\left(  y\right)  =\dfrac{1}{\varepsilon}[\left(
y-d\right)  ^{+}-\left(  c-y\right)  ^{+}].$

Assumption (A$_{7}-i$) is clearly fulfilled; the remaining compatibility
assumptions are satisfies if, for example, $G\left(  t,y\right)  \geq0$, for
$y\leq a$,$\quad G\left(  t,y\right)  \leq0$, for $y\geq b$, and,
respectively, $F\left(  t,y,z\right)  \geq0$, for $y\leq c$,$\quad F\left(
t,y,z\right)  \leq0$, for $y\geq d.$
\end{enumerate}
\end{example}

\section{$L^{p}$--variational solutions}

\subsection{Intuitive introduction}

At the beginning of this section, until further notice, we will consider
$p\geq0$ and $n_{p}=\left(  p-1\right)  ^{+}\wedge1$ and let be fixed an
arbitrary $\lambda\in\left(  0,1\right)  .$ We recall definition
(\ref{defV_1}) of $V$ and we extend the definition of $V$ to the case
$p\in\left[  0,1\right]  \ $by considering%
\[
V_{t}\xlongequal{\hspace{-3pt}\textrm{def}\hspace{-3pt}}%
{\displaystyle\int_{0}^{t}}
\mu_{r}dr+%
{\displaystyle\int_{0}^{t}}
\nu_{r}dA_{r}\,,\quad\text{if }p\in\left[  0,1\right]
\]
(in the case $p\in\left[  0,1\right]  $ we will consider $\ell=0$ (i.e. $H$ is
independent of $z$) and $\ell_{r}^{2}/n_{p}=0).$

By the assumption (\ref{ip-mnl}) we clearly have%
\[
\mathbb{E}\left(  \sup\nolimits_{r\in\left[  0,T\right]  }{e^{pV_{r}}}\right)
<\infty.
\]
Let us define the space $S_{m}^{p}\left(  \gamma,N,R;V\right)  $, $p\geq0$, of
the continuous progressively measurable stochastic process (p.m.s.p.) $M$ such
that for all $p>0$%
\[
\mathbb{E}\left(  \sup\nolimits_{r\in\left[  0,T\right]  }{e^{pV_{r}}%
}\left\vert M_{r}\right\vert ^{p}\right)  <\infty.\medskip
\]
and having the form
\begin{align*}
M_{t}  &  =\gamma-%
{\displaystyle\int_{0}^{t}}
N_{r}dQ_{r}+%
{\displaystyle\int_{0}^{t}}
R_{r}dB_{r},\;\;\text{or equivalent}\\
M_{t}  &  =M_{T}+\int_{t}^{T}N_{r}dQ_{r}-\int_{t}^{T}R_{r}dB_{r}%
~,\;\;M_{0}=\gamma
\end{align*}
where $\gamma\in\mathbb{R}^{m}$and $N:\Omega\times\mathbb{R}_{+}%
\rightarrow\mathbb{R}^{m},$ $R:\Omega\times\mathbb{R}_{+}\rightarrow
\mathbb{R}^{m\times k}$ are p.m.s.p. such that%
\[
\mathbb{E}\left(
{\displaystyle\int_{0}^{T}}
{e^{V_{r}}}\left\vert N_{r}\right\vert dr\right)  ^{p}+\mathbb{E}\left(
{\displaystyle\int_{0}^{T}}
{e^{2V_{r}}}\left\vert R_{r}\right\vert ^{2}dr\right)  ^{p/2}<\infty
,\quad\text{if }p>0
\]
and
\[%
{\displaystyle\int_{0}^{T}}
{e^{V_{r}}}N_{r}dr+%
{\displaystyle\int_{0}^{T}}
{e^{2V_{r}}}\left\vert R_{r}\right\vert ^{2}dr<\infty,\;\mathbb{P}%
-a.s.,\quad\text{if }p=0.
\]
For a intuitive introduction let $p\geq1$ and $\left(  Y,Z,U\right)  $ be a
strong a solution of (\ref{GBSVI 1}) or (\ref{GBSVI 2}) that is $Y,Z,$ and $U$
are p.m.s.p., $Y$ has continuous trajectories,
\[%
{\displaystyle\int_{0}^{T}}
{e^{2V_{r}}}\left\vert Z_{r}\right\vert ^{2}dr+%
{\displaystyle\int_{0}^{T}}
{e^{2V_{r}}}\left\vert U_{r}\right\vert ^{2}dr<\infty,\quad\mathbb{P}-a.s.;
\]
the following equation is satisfied%
\[
\left\{
\begin{array}
[c]{r}%
\displaystyle Y_{t}+{\int_{t}^{T}}dK_{r}=Y_{T}+{\int_{t}^{T}}H\left(
r,Y_{r},Z_{r}\right)  dQ_{r}-{\int_{t}^{T}}Z_{r}dB_{r},\;\text{a.s., for all
}t\in\left[  0,T\right]  ,\medskip\\
\multicolumn{1}{l}{\displaystyle dK_{r}=U_{r}dQ_{r}\in\partial_{y}\Psi\left(
r,Y_{r}\right)  dQ_{r}\,.}%
\end{array}
\,\right.
\]
For $\delta\in(0,1]$ we define%
\begin{equation}
\delta_{q}:=\delta\,\mathbf{1}_{[1,2)}\left(  q\right)  =\left\{
\begin{array}
[c]{ll}%
\delta, & \text{if }1\leq q<2,\\
0, & \text{if }q\geq2.
\end{array}
\right.  \label{defDelta}%
\end{equation}
Let $q\in\left[  1,2\right]  $ and $M\in S_{m}^{0}\left(  \gamma,N,R;V\right)
.$ By It\^{o}'s formula for $\left(  \Gamma_{t}\right)  ^{q},$ where%
\[
\Gamma_{t}:=\left(  \left\vert M_{t}-Y_{t}\right\vert ^{2}+\delta_{q}\right)
^{1/2}%
\]
(see (\ref{ito3}) from Proposition \ref{p1-ito}) with%
\[
M_{t}=M_{T}+%
{\displaystyle\int_{t}^{T}}
N_{r}dQ_{r}-%
{\displaystyle\int_{t}^{T}}
R_{r}dB_{r},
\]
we deduce by inequality (\ref{ito4}) from Remark \ref{r1-ito} that for all
$0\leq t\leq s\leq T$ and for all $\delta\in(0,1],$%
\begin{equation}%
\begin{array}
[c]{l}%
\left(  \Gamma_{t}\right)  ^{q}+\dfrac{q}{2}\,n_{q}\,%
{\displaystyle\int_{t}^{s}}
{\left(  \Gamma_{r}\right)  ^{q-2}}{\Large \,}\left\vert R_{r}-Z_{r}%
\right\vert ^{2}dr+q%
{\displaystyle\int_{t}^{s}}
{\left(  \Gamma_{r}\right)  ^{q-2}}\left\langle M_{r}-Y_{r},U_{r}%
dQ_{r}\right\rangle \medskip\\
\leq\left(  \Gamma_{s}\right)  ^{q}+q%
{\displaystyle\int_{t}^{s}}
{\left(  \Gamma_{r}\right)  ^{q-2}}\langle M_{r}-Y_{r},N_{r}-H\left(
r,Y_{r},Z_{r}\right)  \rangle dQ_{r}\medskip\\
\quad-q%
{\displaystyle\int_{t}^{s}}
{\left(  \Gamma_{r}\right)  ^{q-2}}\,\langle M_{r}-Y_{r},\left(  R_{r}%
-Z_{r}\right)  dB_{r}\rangle,
\end{array}
\label{def1a}%
\end{equation}
where $U_{t}dQ_{t}\in\partial_{y}\Psi\left(  t,Y_{t}\right)  dQ_{t}$ and
$n_{q}:=\left(  q-1\right)  \wedge1=q-1.$

Using the subdifferential inequality%
\[
\left\langle M_{r}-Y_{r},U_{t}dQ_{t}\right\rangle +{\Psi}\left(
r,Y_{r}\right)  dQ_{r}\leq{\Psi}\left(  r,M_{r}\right)  dQ_{r}%
\]
we get, from (\ref{def1a}),%
\begin{equation}%
\begin{array}
[c]{l}%
\left(  \Gamma_{t}\right)  ^{q}+\dfrac{q}{2}n_{q}%
{\displaystyle\int_{t}^{s}}
{\left(  \Gamma_{r}\right)  ^{q-2}}{\Large \,}\left\vert R_{r}-Z_{r}%
\right\vert ^{2}dr+{q%
{\displaystyle\int_{t}^{s}}
\left(  \Gamma_{r}\right)  ^{q-2}\Psi}\left(  r,Y_{r}\right)  dQ_{r}\medskip\\
=\left(  \Gamma_{s}\right)  ^{q}+{q%
{\displaystyle\int_{t}^{s}}
\left(  \Gamma_{r}\right)  ^{q-2}\Psi}\left(  r,M_{r}\right)  dQ_{r}+q%
{\displaystyle\int_{t}^{s}}
{\left(  \Gamma_{r}\right)  ^{q-2}}\langle M_{r}-Y_{r},N_{r}-H\left(
r,Y_{r},Z_{r}\right)  \rangle dQ_{r}\medskip\\
\quad-q%
{\displaystyle\int_{t}^{s}}
{\left(  \Gamma_{r}\right)  ^{q-2}}\,\langle M_{r}-Y_{r},\left(  R_{r}%
-Z_{r}\right)  dB_{r}\rangle.
\end{array}
\label{def1b}%
\end{equation}

\subsection{Definition and preliminary estimates}

Following the approach for the forward stochastic variational inequalities
from article \cite{ra/81}, we propose, stating from (\ref{def1b}), the next
variational formulation for a solution of the multivalued BSDE (\ref{GBSVI 2}%
).\medskip

\begin{definition}
\label{definition_weak solution} Let $V$ be given by definition (\ref{defV_1}%
). We say that $\left(  Y_{t},Z_{t}\right)  _{t\in\left[  0,T\right]  }$ is a
$L^{p}-$variational solution of (\ref{GBSVI 2}) if:

\begin{itemize}
\item $Y:\Omega\times\left[  0,T\right]  \rightarrow\mathbb{R}^{m}$ and
$Z:\Omega\times\left[  0,T\right]  \rightarrow\mathbb{R}^{m\times k}$ are two
p.m.s.p., $Y$ has continuous trajectories%
\begin{equation}
\mathbb{E}\left(  \sup\nolimits_{r\in\left[  0,T\right]  }{e^{pV_{r}}%
}\left\vert Y_{r}\right\vert ^{p}\right)  <\infty. \label{def0-1}%
\end{equation}
and%
\begin{equation}
\mathbb{E}\left(
{\displaystyle\int_{0}^{T}}
{e^{2V_{r}}}\left\vert Z_{r}\right\vert ^{2}dr\right)  ^{p/2}+\mathbb{E}%
\left(
{\displaystyle\int_{0}^{T}}
e^{2V_{r}}{\Psi}\left(  r,Y_{r}\right)  dQ_{r}\right)  ^{p/2}<\infty;
\label{def0-2}%
\end{equation}

\item if $\Gamma_{t}:=\left(  \left\vert M_{t}-Y_{t}\right\vert ^{2}%
+\delta_{q}\right)  ^{1/2},$ where $\delta_{q}$ is defined by (\ref{defDelta}%
), it holds%
\begin{equation}%
\begin{array}
[c]{l}%
\displaystyle\left(  \Gamma_{t}\right)  ^{q}+\dfrac{q\left(  q-1\right)  }{2}%
{\displaystyle\int_{t}^{s}}
{\left(  \Gamma_{r}\right)  ^{q-2}}{\Large \,}\left\vert R_{r}-Z_{r}%
\right\vert ^{2}dr+{q%
{\displaystyle\int_{t}^{s}}
\left(  \Gamma_{r}\right)  ^{q-2}\Psi}\left(  r,Y_{r}\right)  dQ_{r}\medskip\\
\displaystyle\leq\left(  \Gamma_{s}\right)  ^{q}+{q%
{\displaystyle\int_{t}^{s}}
\left(  \Gamma_{r}\right)  ^{q-2}\Psi}\left(  r,M_{r}\right)  dQ_{r}+q%
{\displaystyle\int_{t}^{s}}
{\left(  \Gamma_{r}\right)  ^{q-2}}\langle M_{r}-Y_{r},N_{r}-H\left(
r,Y_{r},Z_{r}\right)  \rangle dQ_{r}\medskip\\
\displaystyle\quad-q%
{\displaystyle\int_{t}^{s}}
{\left(  \Gamma_{r}\right)  ^{q-2}}\,\langle M_{r}-Y_{r},\left(  R_{r}%
-Z_{r}\right)  dB_{r}\rangle,
\end{array}
\label{def1}%
\end{equation}
for any $q\in\{2,p\wedge2\},$ $\delta\in(0,1],$ $0\leq t\leq s\leq T,$ and
$M\in S_{m}^{0}\left(  \gamma,N,R;V\right)  .$
\end{itemize}
\end{definition}

\begin{remark}
For $q=2$ inequality (\ref{def1}) becomes%
\begin{equation}%
\begin{array}
[c]{l}%
\left\vert M_{t}-Y_{t}\right\vert ^{2}+{%
{\displaystyle\int_{t}^{s}}
}\left\vert R_{r}-Z_{r}\right\vert ^{2}dr+{2%
{\displaystyle\int_{t}^{s}}
\Psi}\left(  r,Y_{r}\right)  dQ_{r}\medskip\\
\leq\left\vert M_{s}-Y_{s}\right\vert ^{2}{+2%
{\displaystyle\int_{t}^{s}}
\Psi}\left(  r,M_{r}\right)  dQ_{r}\medskip\\
\quad+{2%
{\displaystyle\int_{t}^{s}}
}\langle M_{r}-Y_{r},N_{r}-H\left(  r,Y_{r},Z_{r}\right)  \rangle dQ_{r}-2%
{\displaystyle\int_{t}^{s}}
\,\langle M_{r}-Y_{r},\left(  R_{r}-Z_{r}\right)  dB_{r}\rangle,\;\mathbb{P}%
\text{-a.s..}%
\end{array}
\label{def1-a}%
\end{equation}
which was in \cite{ma-ra/15} the definition of the variational solution in the
case $p\geq2.$
\end{remark}

\begin{remark}
Let
\[%
\begin{array}
[c]{l}%
\displaystyle\Lambda_{t}=\dfrac{q}{2}\,\left(  q-1\right)  \int\nolimits_{0}%
^{t}{\left(  \Gamma_{r}\right)  ^{q-2}}{\Large \,}\left\vert R_{r}%
-Z_{r}\right\vert ^{2}dr+{q\int_{0}^{t}\left(  \Gamma_{r}\right)  ^{q-2}\Psi
}\left(  r,Y_{r}\right)  dQ_{r}\medskip\\
\displaystyle\quad\quad-{q\int_{0}^{t}\left(  \Gamma_{r}\right)  ^{q-2}\Psi
}\left(  r,M_{r}\right)  dQ_{r}-q\int_{0}^{t}{\left(  \Gamma_{r}\right)
^{q-2}}\langle M_{r}-Y_{r},N_{r}-H\left(  r,Y_{r},Z_{r}\right)  \rangle
dQ_{r}\medskip\\
\displaystyle\quad\quad+q\int_{0}^{t}{\left(  \Gamma_{r}\right)  ^{q-2}%
}\,\langle M_{r}-Y_{r},\left(  R_{r}-Z_{r}\right)  dB_{r}\rangle,
\end{array}
\]
Since from (\ref{def1}) it follows that%
\[
t\longmapsto\left(  \Gamma_{t}\right)  ^{q}-\Lambda_{t}%
\]
is a nondecreasing stochastic process then $t\longmapsto\Gamma_{t}^{q}=\left[
\left(  \Gamma_{t}\right)  ^{q}-\Lambda_{t}\right]  +\Lambda_{t}$ is a
semimartingale and consequently for all $0\leq t\leq s\leq T$%
\begin{align*}
e^{qV_{s}}\left(  \Gamma_{s}\right)  ^{q}-e^{qV_{t}}\left(  \Gamma_{t}\right)
^{q}  &  =%
{\displaystyle\int_{t}^{s}}
d\left[  e^{qV_{r}}\left(  \Gamma_{r}\right)  ^{q}\right] \\
&  =q%
{\displaystyle\int_{t}^{s}}
e^{qV_{r}}\left(  \Gamma_{r}\right)  ^{q}dV_{r}+%
{\displaystyle\int_{t}^{s}}
e^{qV_{r}}d\left[  \left(  \Gamma_{r}\right)  ^{q}-\Lambda_{r}\right]  +%
{\displaystyle\int_{t}^{s}}
e^{qV_{r}}d\Lambda_{r}\\
&  \geq q%
{\displaystyle\int_{t}^{s}}
e^{qV_{r}}\left(  \Gamma_{r}\right)  ^{q}dV_{r}+%
{\displaystyle\int_{t}^{s}}
e^{qV_{r}}d\Lambda_{r}%
\end{align*}
which yields%
\begin{equation}%
\begin{array}
[c]{l}%
e^{qV_{t}}\left(  \Gamma_{t}\right)  ^{q}+q%
{\displaystyle\int_{t}^{s}}
e^{qV_{r}}\left(  \Gamma_{r}\right)  ^{q}dV_{r}+\dfrac{q}{2}\left(
q-1\right)  {%
{\displaystyle\int_{t}^{s}}
}e^{qV_{r}}\left(  \Gamma_{r}\right)  ^{q-2}\left\vert R_{r}-Z_{r}\right\vert
^{2}dr\medskip\\
\quad+{q%
{\displaystyle\int_{t}^{s}}
e^{qV_{r}}\left(  \Gamma_{r}\right)  ^{q-2}\Psi}\left(  r,Y_{r}\right)
dQ_{r}\medskip\\
\leq e^{qV_{s}}\left(  \Gamma_{s}\right)  ^{q}+{q%
{\displaystyle\int_{t}^{s}}
e^{qV_{r}}\left(  \Gamma_{r}\right)  ^{q-2}\Psi}\left(  r,M_{r}\right)
dQ_{r}+{q%
{\displaystyle\int_{t}^{s}}
e^{qV_{r}}}\left(  \Gamma_{r}\right)  ^{q-2}\langle M_{r}-Y_{r},N_{r}-H\left(
r,Y_{r},Z_{r}\right)  \rangle dQ_{r}\medskip\\
\quad-q{%
{\displaystyle\int_{t}^{s}}
e^{qV_{r}}}\left(  \Gamma_{r}\right)  ^{q-2}\langle M_{r}-Y_{r},\left(
R_{r}-Z_{r}\right)  dB_{r}\rangle.
\end{array}
\label{def1-b}%
\end{equation}
for any $q\in\{2,p\wedge2\},$ $\delta\in(0,1],$ $0\leq t\leq s\leq T,$ and
$M\in S_{m}^{0}\left(  \gamma,N,R;V\right)  .$

Following the previous calculus, we see that, in fact, inequality
(\ref{def1-b}) holds true for any arbitrary continuous bounded variation
p.m.s.p. $\left\{  V_{t}:t\in\left[  0,T\right]  \right\}  .$
\end{remark}

\begin{remark}
\label{r-mart part}Let $\left\{  V_{t}:t\in\left[  0,T\right]  \right\}  $ be
defined by (\ref{defV_1}) and $M\in S_{m}^{q}\left(  \gamma,N,R;V\right)  ,$
$q\in\left\{  2,p\wedge2\right\}  .$ Since by assumption (\ref{ip-mnl}) we
have%
\[
\mathbb{E}\left(  \delta_{q}\sup\nolimits_{r\in\left[  0,T\right]  }%
{e^{qV_{r}}}\right)  <\infty,
\]
we deduce
\[%
\begin{array}
[c]{l}%
\displaystyle\mathbb{E}\left[
{\displaystyle\int_{0}^{T}}
{e^{2qV_{r}}}\left(  \Gamma_{r}\right)  ^{2q-4}\,\left\vert M_{r}%
-Y_{r}\right\vert ^{2}\left\vert R_{r}-Z_{r}\right\vert ^{2}dr\right]
^{1/2}\medskip\\
\displaystyle\leq\mathbb{E}\left[
{\displaystyle\int_{0}^{T}}
{e^{2qV_{r}}}\left(  \Gamma_{r}\right)  ^{2q-2}\,\left\vert R_{r}%
-Z_{r}\right\vert ^{2}dr\right]  ^{1/2}\medskip\\
\displaystyle\leq\mathbb{E}\left[
{\displaystyle\int_{0}^{T}}
{e^{2qV_{r}}}\left(  \Gamma_{r}\right)  ^{2q-2}\,\left\vert R_{r}%
-Z_{r}\right\vert ^{2}dr\right]  ^{1/2}\medskip\\
\displaystyle\leq\mathbb{E}\left[  \sup\limits_{r\in\left[  0,T\right]
}{e^{\left(  q-1\right)  V_{r}}}\left(  \left\vert M_{r}-Y_{r}\right\vert
^{2}+\delta_{q}\right)  ^{\left(  q-1\right)  /2}\left(
{\displaystyle\int_{0}^{T}}
{e^{2V_{r}}}\left\vert R_{r}-Z_{r}\right\vert ^{2}dr\right)  ^{1/2}\right]
\medskip\\
\displaystyle\leq\left[  \mathbb{E}\left(  \sup\limits_{r\in\left[
0,T\right]  }{e^{qV_{r}}}\left(  \left\vert M_{r}-Y_{r}\right\vert ^{2}%
+\delta_{q}\right)  ^{q/2}\right)  \right]  ^{\left(  q-1\right)  /q}\,\left[
\mathbb{E}\left(
{\displaystyle\int_{0}^{T}}
{e^{2V_{r}}}\left\vert R_{r}-Z_{r}\right\vert ^{2}dr\right)  ^{q/2}\right]
^{1/q}%
\end{array}
\]
In the case $q=p\wedge2$, we infer that
\[
\mathbb{E}\left[
{\displaystyle\int_{0}^{T}}
{e^{2qV_{r}}}\left(  \Gamma_{r}\right)  ^{2q-4}\,\left\vert M_{r}%
-Y_{r}\right\vert ^{2}\left\vert R_{r}-Z_{r}\right\vert ^{2}dr\right]
^{1/2}<\infty
\]
and the stochastic integral $J_{t}=%
{\displaystyle\int_{0}^{t}}
{e^{qV_{r}}}\left(  \Gamma_{r}\right)  ^{q-2}\langle M_{r}-Y_{r},\left(
R_{r}-Z_{r}\right)  dB_{r}\rangle$ is a continuous martingale; therefore for
all stopping times $0\leq\sigma\leq\theta\leq T:$%
\[
\mathbb{E}^{\mathcal{F}_{\sigma}}%
{\displaystyle\int_{\sigma}^{\theta}}
e^{qV_{r}}\left(  \Gamma_{r}\right)  ^{q-2}\langle M_{r}-Y_{r},\left(
R_{r}-Z_{r}\right)  dB_{r}\rangle=0.
\]
We also have
\begin{align*}
&  \mathbb{E}%
{\displaystyle\int_{0}^{T}}
e^{qV_{r}}\left(  \Gamma_{r}\right)  ^{q-2}\left\vert \langle M_{r}%
-Y_{r},N_{r}-H\left(  r,Y_{r},Z_{r}\right)  \rangle\right\vert dQ_{r}\\
&  \leq\mathbb{E}%
{\displaystyle\int_{0}^{T}}
e^{qV_{r}}\left(  \Gamma_{r}\right)  ^{q-1}\left[  \left\vert N_{r}\right\vert
+\left\vert H\left(  r,Y_{r},Z_{r}\right)  \right\vert \right]  dQ_{r}\\
&  \leq\mathbb{E}\left[  \sup\limits_{r\in\left[  0,T\right]  }{e^{\left(
q-1\right)  V_{r}}}\left(  \left\vert M_{r}-Y_{r}\right\vert ^{2}+\delta
_{q}\right)  ^{\left(  q-1\right)  /2}\left(
{\displaystyle\int_{0}^{T}}
{e^{V_{r}}}\left[  \left\vert N_{r}\right\vert +\left\vert H\left(
r,Y_{r},Z_{r}\right)  \right\vert \right]  dQ_{r}\right)  \right] \\
&  \leq\left[  \mathbb{E}\left(  \sup\limits_{r\in\left[  0,T\right]
}{e^{qV_{r}}}\left(  \left\vert M_{r}-Y_{r}\right\vert ^{2}+\delta_{q}\right)
^{q/2}\right)  \right]  ^{\left(  q-1\right)  /q}\,\left[  \mathbb{E}\left(
{\displaystyle\int_{0}^{T}}
{e^{V_{r}}}\left[  \left\vert N_{r}\right\vert +\left\vert H\left(
r,Y_{r},Z_{r}\right)  \right\vert \right]  dQ_{r}dr\right)  ^{q}\right]
^{1/q}.
\end{align*}
Hence if $\left(  Y_{t},Z_{t}\right)  _{t\in\left[  0,T\right]  }$ is an
$L^{p}-$variational solution of (\ref{GBSVI 2}) then, if%
\[
\mathbb{E}\left(
{\displaystyle\int_{0}^{T}}
{e^{V_{r}}}\left\vert H\left(  r,Y_{r},Z_{r}\right)  \right\vert
dQ_{r}\right)  ^{p\wedge2}<\infty,
\]
the following inequality is satisfied $\mathbb{P}-a.s.\;$%
\begin{equation}%
\begin{array}
[c]{l}%
\displaystyle e^{qV_{\sigma}}\Gamma_{\sigma}^{q}+q\mathbb{E}^{\mathcal{F}%
_{\sigma}}%
{\displaystyle\int_{\sigma}^{\theta}}
e^{qV_{r}}\Gamma_{r}^{q}dV_{r}+\dfrac{q}{2}n_{q}\mathbb{E}^{\mathcal{F}%
_{\sigma}}%
{\displaystyle\int_{\sigma}^{\theta}}
e^{qV_{r}}\Gamma_{r}^{q-2}\left\vert R_{r}-Z_{r}\right\vert ^{2}dr\medskip\\
\displaystyle\quad+{q\mathbb{E}^{\mathcal{F}_{\sigma}}%
{\displaystyle\int_{\sigma}^{\theta}}
e^{pV_{r}}\Gamma_{r}^{q-2}\Psi}\left(  r,Y_{r}\right)  dQ_{r}\medskip\\
\displaystyle\leq\mathbb{E}^{\mathcal{F}_{\sigma}}e^{qV_{\theta}}%
\Gamma_{\theta}^{q}+{q\mathbb{E}^{\mathcal{F}_{\sigma}}%
{\displaystyle\int_{\sigma}^{\theta}}
e^{qV_{r}}\Gamma_{r}^{q-2}\Psi}\left(  r,M_{r}\right)  dQ_{r}\\
\displaystyle\quad+{q\mathbb{E}^{\mathcal{F}_{\sigma}}%
{\displaystyle\int_{\sigma}^{\theta}}
e^{qV_{r}}\Gamma_{r}^{q-2}}\langle M_{r}-Y_{r},N_{r}-H\left(  r,Y_{r}%
,Z_{r}\right)  \rangle dQ_{r},
\end{array}
\label{def1d}%
\end{equation}
for $q=p\wedge2$ and all $M\in S_{m}^{q}\left(  \gamma,N,R;V\right)  $ and for
all stopping times $0\leq\sigma\leq\theta\leq T.$
\end{remark}

\begin{remark}
\label{s-w}It is obviously that a strong solution $\left(  Y,Z\right)  \in
S_{m}^{0}\times\Lambda_{m\times k}^{0}$ for (\ref{GBSVI 2}) such that
(\ref{def0-1}), (\ref{def0-2}) and (\ref{def2}) are satisfied is also an
$L^{p}-$variational solution (see the intuitive introduction for inequality
(\ref{def1b})).$\medskip$

Conversely, if $\left(  Y,Z\right)  $ is an $L^{p}-$variational solution of
the BSDE (\ref{GBSVI 2}) with $\varphi=\psi=0,$ $V$ is a nondecreasing
stochastic process and
\[
\mathbb{E}\left(
{\displaystyle\int_{0}^{T}}
{e^{V_{r}}}\left\vert H\left(  r,Y_{r},Z_{r}\right)  \right\vert
dQ_{r}\right)  ^{q}<\infty,
\]
then $\left(  Y,Z\right)  $ is a strong solution of BSDE (\ref{GBSVI 2}).

Indeed, by \cite[Corollary 2.45]{pa-ra/14} there exists a unique pair $\left(
M,R\right)  \in S_{m}^{q}\left[  0,T\right]  \times\Lambda_{m\times k}%
^{q}\left(  0,T\right)  $ such that
\[
M_{t}=Y_{T}+%
{\displaystyle\int_{t}^{T}}
H\left(  r,Y_{r},Z_{r}\right)  dQ_{r}-%
{\displaystyle\int_{t}^{T}}
R_{r}dB_{r}%
\]
and%
\[
\mathbb{E}\sup\limits_{t\in\left[  0,T\right]  }\left\vert e^{V_{t}}%
M_{t}\right\vert ^{q}+\mathbb{E}\left(
{\displaystyle\int_{0}^{T}}
e^{2V_{r}}\left\vert R_{r}\right\vert ^{2}dr\right)  ^{q/2}<\infty.
\]
With this $M$ the inequality (\ref{def1-b}) becomes, $\mathbb{P}-a.s.\;$%
\[%
\begin{array}
[c]{r}%
e^{qV_{t}}\left(  \Gamma_{t}\right)  ^{q}+q%
{\displaystyle\int_{t}^{s}}
e^{qV_{r}}\left(  \Gamma_{r}\right)  ^{q}dV_{r}+\dfrac{q}{2}\left(
q-1\right)  {%
{\displaystyle\int_{t}^{s}}
}e^{qV_{r}}\left(  \Gamma_{r}\right)  ^{q-2}\left\vert R_{r}-Z_{r}\right\vert
^{2}dr\medskip\\
\leq e^{qV_{s}}\left(  \Gamma_{s}\right)  ^{q}-q{%
{\displaystyle\int_{t}^{s}}
e^{qV_{r}}}\left(  \Gamma_{r}\right)  ^{q-2}\langle M_{r}-Y_{r},\left(
R_{r}-Z_{r}\right)  dB_{r}\rangle.
\end{array}
\]
for any $q\in\{2,p\wedge2\},$ $\delta\in(0,1],$ $0\leq t\leq s\leq T.$

By Remark \ref{r-mart part} for $q=p\wedge2$ the stochastic integral is a
martingale and therefore since $0<\delta\leq1,$ we obtain $\mathbb{P}-a.s.\;$
\begin{equation}
e^{qV_{t}}\left(  \Gamma_{t}\right)  ^{q}+\dfrac{q}{2}\left(  q-1\right)
\mathbb{E}^{\mathcal{F}_{t}}{%
{\displaystyle\int_{t}^{T}}
}e^{qV_{r}}\frac{\left\vert R_{r}-Z_{r}\right\vert ^{2}}{\left(  \left\vert
M_{r}-Y_{r}\right\vert ^{2}+1\right)  ^{\left(  2-q\right)  /2}}dr\leq\left(
\delta_{q}\right)  ^{q}\mathbb{E}^{\mathcal{F}_{t}}e^{qV_{T}},\quad\text{for
all }0\leq t\leq T. \label{w-to-s}%
\end{equation}
Passing to limit as $\delta\rightarrow0_{+},$ by Fatou's Lemma we obtain,
$\mathbb{P}-a.s.\;$%
\[
e^{qV_{t}}\left\vert M_{t}-Y_{t}\right\vert ^{q}+\dfrac{q}{2}\left(
q-1\right)  \mathbb{E}^{\mathcal{F}_{t}}{%
{\displaystyle\int_{t}^{T}}
}e^{qV_{r}}\frac{\left\vert R_{r}-Z_{r}\right\vert ^{2}}{\left(  \left\vert
M_{r}-Y_{r}\right\vert ^{2}+1\right)  ^{\left(  2-q\right)  /2}}%
dr=0,\quad\text{for all }t\in\left[  0,T\right]  .
\]
that clearly yields $\left(  M,R\right)  =\left(  Y,Z\right)  $ in $S_{m}%
^{0}\left[  0,T\right]  \times\Lambda_{m\times k}^{q}\left(  0,T\right)  .$
Consequently%
\[
Y_{t}=Y_{T}+%
{\displaystyle\int_{t}^{T}}
H\left(  r,Y_{r},Z_{r}\right)  dQ_{r}-%
{\displaystyle\int_{t}^{T}}
Z_{r}dB_{r}~.
\]

\end{remark}

\begin{proposition}
\label{p-estim}Let $M\in S_{m}^{0}\left(  \gamma,N,R;V\right)  $. Let
$Y:\Omega\times\left[  0,T\right]  \rightarrow\mathbb{R}^{m}$ and
$Z:\Omega\times\left[  0,T\right]  \rightarrow\mathbb{R}^{m\times k}$ be two
p.m.s.p. with $Y$ having continuous trajectories and%
\[%
\begin{array}
[c]{rl}%
\left(  i\right)  &
{\displaystyle\int_{0}^{T}}
{e^{2V_{r}}}\left\vert R_{r}-Z_{r}\right\vert ^{2}dr+%
{\displaystyle\int_{0}^{T}}
e^{2V_{r}}{\Psi}\left(  r,Y_{r}\right)  dQ_{r}<\infty,\;\mathbb{P}%
-a.s.,\medskip\\
\left(  ii\right)  & {\Psi}\left(  r,M_{r}\right)  \leq\mathbf{1}_{q\geq
2}{\Psi}\left(  r,M_{r}\right)  \medskip\\
\left(  iii\right)  & \left\langle M_{r}-Y_{r},N_{r}\right\rangle dQ_{r}%
\leq\left\vert M_{r}-Y_{r}\right\vert dL_{r}%
\end{array}
\]
with $L$ an increasing and continuous p.m.s.p. $L_{0}=0$. \newline%
\noindent\textbf{I.} If inequality (\ref{def1}) holds for $q=2,$ then for all
$k>0$ and for any stopping times $0\leq\sigma\leq\theta<T$%
\begin{equation}%
\begin{array}
[c]{l}%
\mathbb{E}^{\mathcal{F}_{\sigma}}\left(
{\displaystyle\int_{\sigma}^{\theta}}
{e^{2V_{r}}}\left\vert R_{r}-Z_{r}\right\vert ^{2}dr\right)  ^{k/2}%
+\mathbb{E}^{\mathcal{F}_{\sigma}}\left(
{\displaystyle\int_{\sigma}^{\theta}}
e^{2V_{r}}{\Psi}\left(  r,Y_{r}\right)  dQ_{r}\right)  ^{k/2}\medskip\\
\leq C_{k,\lambda}~\bigg[\mathbb{E}^{\mathcal{F}_{\sigma}}\sup\nolimits_{r\in
\left[  \sigma,\theta\right]  }e^{kV_{r}}\left\vert M_{r}-Y_{r}\right\vert
^{k}+\mathbb{E}^{\mathcal{F}_{\sigma}}\left(
{\displaystyle\int_{\sigma}^{\theta}}
e^{V_{r}}{\Psi}\left(  r,M_{r}\right)  dQ_{r}\right)  ^{k/2}\medskip\\
\quad+\mathbb{E}^{\mathcal{F}_{\sigma}}\left(
{\displaystyle\int_{\sigma}^{\theta}}
e^{V_{r}}\left\vert M_{r}-Y_{r}\right\vert \left[  dL_{r}+\left\vert H\left(
r,M_{r},R_{r}\right)  \right\vert dQ_{r}\right]  \right)  ^{k/2}%
\bigg]\medskip\\
\leq2C_{k,\lambda}\bigg[\mathbb{E}^{\mathcal{F}_{\sigma}}\sup\nolimits_{r\in
\left[  \sigma,\theta\right]  }e^{kV_{r}}\left\vert M_{r}-Y_{r}\right\vert
^{k}+\mathbb{E}^{\mathcal{F}_{\sigma}}\left(
{\displaystyle\int_{\sigma}^{\theta}}
e^{V_{r}}{\Psi}\left(  r,M_{r}\right)  dQ_{r}\right)  ^{k/2}\medskip\\
\quad+\mathbb{E}^{\mathcal{F}_{\sigma}}\left(
{\displaystyle\int_{\sigma}^{\theta}}
e^{V_{r}}\left[  dL_{r}+\left\vert H\left(  r,M_{r},R_{r}\right)  \right\vert
dQ_{r}\right]  \right)  ^{k}\bigg],\quad\mathbb{P}\text{--a.s..}%
\end{array}
\label{def-11}%
\end{equation}
In particular for $\gamma=0$, $N=0$, $R=0$, $L=0$, $M=0,$ ${\Psi}\left(
r,M\right)  ={\Psi}\left(  r,0\right)  =0$ it follows%
\begin{equation}%
\begin{array}
[c]{l}%
\mathbb{E}^{\mathcal{F}_{\sigma}}\left(
{\displaystyle\int_{\sigma}^{\theta}}
{e^{2V_{r}}}\left\vert Z_{r}\right\vert ^{2}dr\right)  ^{k/2}+\mathbb{E}%
^{\mathcal{F}_{\sigma}}\left(
{\displaystyle\int_{\sigma}^{\theta}}
e^{2V_{r}}{\Psi}\left(  r,Y_{r}\right)  dQ_{r}\right)  ^{k/2}\medskip\\
\leq C_{k,\lambda}\left[  \mathbb{E}^{\mathcal{F}_{\sigma}}\sup\nolimits_{r\in
\left[  \sigma,\theta\right]  }e^{kV_{r}}\left\vert Y_{r}\right\vert
^{k}+\mathbb{E}^{\mathcal{F}_{\sigma}}\left(
{\displaystyle\int_{\sigma}^{\theta}}
e^{V_{r}}\left\vert Y_{r}\right\vert \left\vert H\left(  r,0,0\right)
\right\vert dQ_{r}\right)  ^{k/2}\right]  \medskip\\
\leq2C_{k,\lambda}\left[  \mathbb{E}^{\mathcal{F}_{\sigma}}\sup\nolimits_{r\in
\left[  \sigma,\theta\right]  }e^{kV_{r}}\left\vert Y_{r}\right\vert
^{k}+\mathbb{E}^{\mathcal{F}_{\sigma}}\mathbb{~}\left(
{\displaystyle\int_{\sigma}^{\theta}}
e^{V_{r}}\left\vert H\left(  r,0,0\right)  \right\vert dQ_{r}\right)
^{k}\right]  ,\quad\mathbb{P}-a.s..
\end{array}
\label{def-11aa}%
\end{equation}
\noindent\textbf{II.} If inequality (\ref{def1}) holds and for some fixed
stopping times $0\leq\sigma\leq\theta<T,$ $1<q\leq k$
\begin{equation}
\mathbb{E}\left(  \sup\limits_{r\in\left[  \sigma,\theta\right]  }{e^{kV_{r}}%
}\left\vert M_{r}-Y_{r}\right\vert ^{k}\right)  <\infty, \label{def-11a}%
\end{equation}
then%
\begin{equation}%
\begin{array}
[c]{l}%
\mathbb{E}^{\mathcal{F}_{\sigma}}\sup\limits_{r\in\left[  \sigma
,\theta\right]  }e^{kV_{r}}\left\vert M_{r}-Y_{r}\right\vert ^{k}\\
\leq C_{\lambda,q,k}~~\bigg[\mathbb{E}^{\mathcal{F}_{\sigma}}e^{kV_{\theta}%
}\left\vert M_{\theta}-Y_{\theta}\right\vert ^{k}+\mathbb{E}^{\mathcal{F}%
_{\sigma}}\mathbb{~}\left(
{\displaystyle\int_{\sigma}^{\theta}}
e^{V_{r}}\left\vert M_{r}-Y_{r}\right\vert ^{q-2}\mathbf{1}_{q\geq2}{\Psi
}\left(  r,M_{r}\right)  dQ_{r}\right)  ^{k/q}\medskip\\
\quad+~\mathbb{E}^{\mathcal{F}_{\sigma}}\mathbb{~}\left(
{\displaystyle\int_{\sigma}^{\theta}}
{e^{qV_{r}}}\left\vert M_{r}-Y_{r}\right\vert ^{q-1}\left[  dL_{r}+\left\vert
H\left(  r,M_{r},R_{r}\right)  \right\vert dQ_{r}\right]  \right)
^{k/q}\bigg],\quad\mathbb{P}-a.s.\;
\end{array}
\label{def-11b}%
\end{equation}
and%
\begin{equation}%
\begin{array}
[c]{l}%
\mathbb{E}^{\mathcal{F}_{\sigma}}\Big(\sup\limits_{r\in\left[  \sigma
,\theta\right]  }e^{kV_{r}}\left\vert M_{r}-Y_{r}\right\vert ^{k}%
\Big)+\mathbb{E}^{\mathcal{F}_{\sigma}}~\left(
{\displaystyle\int_{\sigma}^{\theta}}
{e^{qV_{r}}\left\vert M_{r}-Y_{r}\right\vert ^{q-2}}\left\vert R_{r}%
-Z_{r}\right\vert ^{2}dr\right)  ^{k/q}\medskip\\
\quad+\mathbb{E}^{\mathcal{F}_{\sigma}}\left(
{\displaystyle\int_{\sigma}^{\theta}}
{e^{qV_{r}}\left\vert M_{r}-Y_{r}\right\vert ^{q-2}\Psi}\left(  r,Y_{r}%
\right)  dQ_{r}\right)  ^{k/q}\medskip\\
\leq C_{\lambda,q,k}~\bigg[\mathbb{E}^{\mathcal{F}_{\sigma}}~e^{kV_{\theta}%
}\left\vert M_{\theta}-Y_{\theta}\right\vert ^{k}+~\mathbb{E}^{\mathcal{F}%
_{\sigma}}~\left(
{\displaystyle\int_{\sigma}^{\theta}}
e^{V_{r}}\mathbf{1}_{q\geq2}{\Psi}\left(  r,M_{r}\right)  dQ_{r}\right)
^{k/2}\medskip\\
\quad+\mathbb{E}^{\mathcal{F}_{\sigma}}\Big(%
{\displaystyle\int_{\sigma}^{\theta}}
e^{V_{r}}\left[  dL_{r}+\left\vert H\left(  r,M_{r},R_{r}\right)  \right\vert
dQ_{r}\right]  \Big)^{q}\bigg],\;\quad\mathbb{P}-a.s.~.
\end{array}
\label{def-11c}%
\end{equation}
In particular for $\gamma=0$, $N=0$, $R=0$, $L=0$, $M=0,$ ${\Psi}\left(
r,M\right)  ={\Psi}\left(  r,0\right)  =0$ it follows $\mathbb{P}$--a.s.%
\begin{equation}
\mathbb{E}^{\mathcal{F}_{\sigma}}\sup\limits_{r\in\left[  \sigma
,\theta\right]  }e^{kV_{r}}\left\vert Y_{r}\right\vert ^{k}\leq C_{\lambda
,q,k}~~\mathbb{E}^{\mathcal{F}_{\sigma}}\left[  e^{kV_{\theta}}\left\vert
Y_{\theta}\right\vert ^{k}+\left(
{\displaystyle\int_{\sigma}^{\theta}}
{e^{qV_{r}}}\left\vert Y_{r}\right\vert ^{q-1}\left\vert H\left(
r,0,0\right)  \right\vert dQ_{r}\right)  ^{k/q}\right]  \label{def-11cc}%
\end{equation}
and%
\begin{equation}%
\begin{array}
[c]{l}%
\displaystyle\mathbb{E}^{\mathcal{F}_{\sigma}}\Big(\sup\limits_{r\in\left[
\sigma,\theta\right]  }e^{kV_{r}}\left\vert Y_{r}\right\vert ^{k}%
\Big)+\mathbb{E}^{\mathcal{F}_{\sigma}}\left(
{\displaystyle\int_{\sigma}^{\theta}}
{e^{qV_{r}}\left\vert Y_{r}\right\vert ^{q-2}}\left\vert Z_{r}\right\vert
^{2}dr\right)  ^{k/q}\medskip\\
\displaystyle\quad+\mathbb{E}^{\mathcal{F}_{\sigma}}\left(
{\displaystyle\int_{\sigma}^{\theta}}
{e^{qV_{r}}\left\vert Y_{r}\right\vert ^{q-2}\Psi}\left(  r,Y_{r}\right)
dQ_{r}\right)  ^{k/q}\medskip\\
\displaystyle\leq C_{\lambda,q}~\mathbb{E}^{\mathcal{F}_{\sigma}}\left[
e^{kV_{\theta}}\left\vert Y_{\theta}\right\vert ^{k}+\Big(%
{\displaystyle\int_{\sigma}^{\theta}}
e^{V_{r}}\left\vert H\left(  r,0,0\right)  \right\vert dQ_{r}\Big)^{k}\right]
.
\end{array}
\label{def-11ccc}%
\end{equation}

\end{proposition}

\begin{proof}
Using the monotonicity of $H:$%
\[%
\begin{array}
[c]{l}%
\displaystyle\left\langle M_{r}-Y_{r},-H\left(  r,Y_{r},Z_{r}\right)
dQ_{r}\right\rangle \medskip\\
\displaystyle=\left\langle M_{r}-Y_{r},-H\left(  r,M_{r},R_{r}\right)
dQ_{r}\right\rangle +\left\langle M_{r}-Y_{r},H\left(  r,M_{r},R_{r}\right)
-H\left(  r,Y_{r},Z_{r}\right)  dQ_{r}\right\rangle \medskip\\
\displaystyle\leq|M_{r}-Y_{r}|\left\vert H\left(  r,M_{r},R_{r}\right)
\right\vert dQ_{r}+|M_{r}-Y_{r}|^{2}dV_{r}+\dfrac{n_{p}\lambda}{2}\,\left\vert
R_{r}-Z_{r}\right\vert ^{2}ds
\end{array}
\]
we obtain from (\ref{def1}):
\[%
\begin{array}
[c]{l}%
\displaystyle\left(  \Gamma_{t}\right)  ^{q}+\dfrac{q}{2}\left(
q-1-n_{p}\lambda\right)
{\displaystyle\int_{t}^{s}}
\left(  \Gamma_{r}\right)  ^{q-2}{\Large \,}\left\vert R_{r}-Z_{r}\right\vert
^{2}dr+q\delta_{q}%
{\displaystyle\int_{t}^{s}}
\left(  \Gamma_{r}\right)  ^{q-2}dV_{r}\medskip\\
\displaystyle\quad+{q%
{\displaystyle\int_{t}^{s}}
\left(  \Gamma_{r}\right)  ^{q-2}\Psi}\left(  r,Y_{r}\right)  dQ_{r}\medskip\\
\displaystyle\leq\left(  \Gamma_{s}\right)  ^{q}+q%
{\displaystyle\int_{t}^{s}}
\left(  \Gamma_{r}\right)  ^{q}dV_{r}+{q%
{\displaystyle\int_{t}^{s}}
\left(  \Gamma_{r}\right)  ^{q-2}\mathbf{1}_{q\geq2}{\Psi}\left(
r,M_{r}\right)  }dQ_{r}\medskip\\
\displaystyle\quad+q%
{\displaystyle\int_{t}^{s}}
\left(  \Gamma_{r}\right)  ^{q-2}\left\vert M_{r}-Y_{r}\right\vert \left[
dL_{r}+\left\vert H\left(  r,M_{r},R_{r}\right)  \right\vert dQ_{r}\right]  -q%
{\displaystyle\int_{t}^{s}}
\left(  \Gamma_{r}\right)  ^{q-2}\,\langle M_{r}-Y_{r},\left(  R_{r}%
-Z_{r}\right)  dB_{r}\rangle.
\end{array}
\]
for all $0\leq t\leq s\leq T.$

Since $p>1$ and $q\in\left\{  2,p\wedge2\right\}  $, then $n_{p}=\left(
p-1\right)  \wedge1\leq q-1~$and%
\[
\left(  q-1\right)  \left(  1-\lambda\right)  \leq\left(  q-1-n_{p}%
\lambda\right)  .
\]
Applying a Gronwall's type stochastic inequality (see Lemma 12 from the
Appendix of \cite{ma-ra/07}) we conclude that%
\begin{equation}%
\begin{array}
[c]{l}%
\displaystyle e^{qV_{t}}\left(  \Gamma_{t}\right)  ^{q}+\dfrac{q}{2}\left(
q-1\right)  \left(  1-\lambda\right)
{\displaystyle\int_{t}^{s}}
e^{qV_{r}}\left(  \Gamma_{r}\right)  ^{q-2}{\Large \,}\left\vert R_{r}%
-Z_{r}\right\vert ^{2}dr+q~\delta_{q}%
{\displaystyle\int_{t}^{s}}
e^{qV_{r}}\left(  \Gamma_{r}\right)  ^{q-2}dV_{r}\medskip\\
\displaystyle\quad+{q%
{\displaystyle\int_{t}^{s}}
e^{qV_{r}}\left(  \Gamma_{r}\right)  ^{q-2}\Psi}\left(  r,Y_{r}\right)
dQ_{r}\medskip\\
\displaystyle\leq e^{qV_{s}}\left(  \Gamma_{s}\right)  ^{q}+{q%
{\displaystyle\int_{t}^{s}}
\left(  \Gamma_{r}\right)  ^{q-2}\mathbf{1}_{q\geq2}{\Psi}\left(
r,M_{r}\right)  }dQ_{r}\medskip\\
\displaystyle\quad+q%
{\displaystyle\int_{t}^{s}}
e^{qV_{r}}\left(  \Gamma_{r}\right)  ^{q-2}\left\vert M_{r}-Y_{r}\right\vert
\left[  dL_{r}+\left\vert H\left(  r,M_{r},R_{r}\right)  \right\vert
dQ_{r}\right]  \medskip\\
\displaystyle\quad-q%
{\displaystyle\int_{t}^{s}}
e^{qV_{r}}\left(  \Gamma_{r}\right)  ^{q-2}\,\langle M_{r}-Y_{r},\left(
R_{r}-Z_{r}\right)  dB_{r}\rangle.
\end{array}
\label{def1-e}%
\end{equation}
\noindent\textbf{I.} Writing (\ref{def1-e}) for $q=2$ we get%
\[%
\begin{array}
[c]{l}%
\displaystyle e^{2V_{t}}\left\vert M_{t}-Y_{t}\right\vert ^{2}+\left(
1-\lambda\right)  {%
{\displaystyle\int_{t}^{s}}
}e^{2V_{r}}\left\vert R_{r}-Z_{r}\right\vert ^{2}dr+{2%
{\displaystyle\int_{t}^{s}}
}e^{2V_{r}}{\Psi}\left(  r,Y_{r}\right)  dQ_{r}\medskip\\
\displaystyle\leq e^{2V_{s}}\left\vert M_{s}-Y_{s}\right\vert ^{2}+{2%
{\displaystyle\int_{t}^{s}}
{\Psi}\left(  r,M_{r}\right)  }dQ_{r}+{2%
{\displaystyle\int_{t}^{s}}
}e^{2V_{r}}\left\vert M_{r}-Y_{r}\right\vert \left[  dL_{r}+\left\vert
H\left(  r,M_{r},R_{r}\right)  \right\vert dQ_{r}\right]  \medskip\\
\displaystyle\quad-2%
{\displaystyle\int_{t}^{s}}
\,e^{2V_{r}}\langle M_{r}-Y_{r},\left(  R_{r}-Z_{r}\right)  dB_{r}%
\rangle,\;\mathbb{P}\text{--a.s..}%
\end{array}
\]
for all $0\leq t\leq s\leq T,$ that yields (\ref{def-11}) by Proposition
\ref{an-prop-dz} from Appendix.

\noindent\textbf{II.} Using Fatou's Lemma, Lebesgue dominated convergence
theorem and the continuity in probability of the stochastic integral we
clearly deduce from (\ref{def1-e}), as $\delta\rightarrow0_{+}$, that:%
\begin{equation}%
\begin{array}
[c]{l}%
\displaystyle e^{qV_{t}}\left\vert M_{t}-Y_{t}\right\vert ^{q}+\dfrac{q}%
{2}\left(  q-1\right)  \left(  1-\lambda\right)  {\int\nolimits_{t}^{s}%
}e^{qV_{r}}\left\vert M_{r}-Y_{r}\right\vert ^{q-2}{\Large \,}\left\vert
R_{r}-Z_{r}\right\vert ^{2}dr\medskip\\
\displaystyle\quad+{q\int\nolimits_{t}^{s}e^{qV_{r}}\left\vert M_{r}%
-Y_{r}\right\vert ^{q-2}\Psi}\left(  r,Y_{r}\right)  dQ_{r}\medskip\\
\displaystyle\leq e^{qV_{s}}\left\vert M_{s}-Y_{s}\right\vert ^{q}%
+{q\int\nolimits_{t}^{s}\left\vert M_{r}-Y_{r}\right\vert ^{q-2}%
{\Large \,}\mathbf{1}_{q\geq2}{\Psi}\left(  r,M_{r}\right)  }dQ_{r}\medskip\\
\displaystyle\quad+q{\int\nolimits_{t}^{s}}e^{qV_{r}}\left\vert M_{r}%
-Y_{r}\right\vert ^{q-1}\left[  dL_{r}+\left\vert H\left(  r,M_{r}%
,R_{r}\right)  \right\vert dQ_{r}\right]  \medskip\\
\displaystyle\quad-q{\int\nolimits_{t}^{s}}e^{qV_{r}}\,\langle\left\vert
M_{r}-Y_{r}\right\vert ^{q-2}\left(  M_{r}-Y_{r}\right)  ,\left(  R_{r}%
-Z_{r}\right)  dB_{r}\rangle.
\end{array}
\label{def1-f}%
\end{equation}

Using Proposition \ref{an-prop-ydz} from inequality (\ref{def1-f}) we get
(\ref{def-11b}) and (\ref{def-11c}).\hfill
\end{proof}

\subsection{Uniqueness and Continuity}

\begin{theorem}
\label{uniq}Let $p>1$, $q=p\wedge2$ and the assumptions $(\mathrm{A}%
_{1}-\mathrm{A}_{5})$ be satisfied. Then the backward stochastic variational
inequality (\ref{GBSVI 2}) has at most solution $\left(  Y,Z\right)  $ in the
sense of Definition \ref{definition_weak solution}.

Moreover, if $(\hat{Y},\hat{Z})$ and $(\tilde{Y},\tilde{Z})$ are two $L^{p}%
-$variational solutions of (\ref{GBSVI 2}) corresponding to $(\hat{\eta}%
,\hat{H})$ and $(\tilde{\eta},\tilde{H})$ respectively, where $\hat{H}$ and
$\tilde{H}$ have the same coefficients $\mu,\nu,\ell$ (there are constants
functions), then for any stopping times $\sigma,\theta,$ such that
$0\leq\sigma\leq\theta\leq T,$ it holds, $\mathbb{P}$--a.s.,%
\begin{equation}%
\begin{array}
[c]{l}%
^{qV_{\sigma}}|\hat{Y}_{\sigma}-\tilde{Y}_{\sigma}|^{q}+c_{q,\lambda}%
\times{\mathbb{E}}^{\mathcal{F}_{\sigma}}{%
{\displaystyle\int_{\sigma}^{\theta}}
}e^{qV_{r}}\dfrac{|\hat{Z}_{r}-\tilde{Z}_{r}|^{2}}{\big(|\hat{Y}_{r}-\tilde
{Y}_{r}|+1\big)^{2-q}}dr\medskip\\
\leq{\mathbb{E}}^{\mathcal{F}_{\sigma}}~e^{qV_{\theta}}|\hat{Y}_{\theta
}-\tilde{Y}_{\theta}|^{q}+q\,{\mathbb{E}}^{\mathcal{F}_{\sigma}}{%
{\displaystyle\int_{\sigma}^{\theta}}
}e^{qV_{r}}|\hat{Y}_{r}-\tilde{Y}_{r}|^{q-1}\big|\hat{H}(r,\hat{Y}_{r},\hat
{Z}_{r})-\tilde{H}(r,\hat{Y}_{r},\hat{Z}_{r})\big|dQ_{r}\medskip\\
\leq\mathbb{E}^{\mathcal{F}_{\sigma}}~e^{qV_{\theta}}|\hat{Y}_{\theta}%
-\tilde{Y}_{\theta}|^{q}+M_{\sigma}\,\left[  \mathbb{E}^{\mathcal{F}_{\sigma}%
}\left(  {%
{\displaystyle\int_{\sigma}^{\theta}}
}e^{V_{r}}\big|\hat{H}(r,\hat{Y}_{r},\hat{Z}_{r})-\tilde{H}(r,\hat{Y}_{r}%
,\hat{Z}_{r})\big|dQ_{r}\right)  ^{q}\right]  ^{1/q},
\end{array}
\label{cont-1}%
\end{equation}
where%
\begin{align*}
M_{\sigma}  &  =C_{q,\lambda}\,\bigg[{\mathbb{E}}^{\mathcal{F}_{\sigma}%
}\left(  e^{qV_{T}}|\hat{\eta}|^{q}+~\Big(%
{\displaystyle\int_{0}^{T}}
e^{V_{r}}\big|\hat{H}\left(  r,0,0\right)  \big|dQ_{r}\Big)^{q}\right) \\
&  \quad+{\mathbb{E}}^{\mathcal{F}_{\sigma}}\left(  e^{qV_{T}}|\tilde{\eta
}|^{q}+\Big(%
{\displaystyle\int_{0}^{T}}
e^{V_{r}}\big|\tilde{H}\left(  r,0,0\right)  \big|dQ_{r}\Big)^{q}\right)
\bigg]^{\left(  q-1\right)  /q}%
\end{align*}
and $c_{q,\lambda},$ $C_{q,\lambda}$ are positive constants depending only $q$
and $\lambda.\medskip$

Moreover, for all $0<\alpha<1$%
\begin{equation}%
\begin{array}
[c]{l}%
\mathbb{E}\sup\limits_{t\in\left[  0,T\right]  }e^{\alpha qV_{t}}|\hat{Y}%
_{t}-\tilde{Y}_{t}|^{\alpha q}+\left(  \mathbb{E}%
{\displaystyle\int_{0}^{T}}
\dfrac{1}{\left(  e^{V_{r}}|\hat{Y}_{r}-\tilde{Y}_{r}|+1\right)  ^{2-q}%
}\,e^{2V_{r}}|\hat{Z}_{r}-\tilde{Z}_{r}|^{2}dr\right)  ^{\alpha}\medskip\\
\quad\leq C_{\alpha,q,\lambda}\left[  \mathbb{E}e^{qV_{T}}|\hat{\eta}%
-\tilde{\eta}|^{q}+K\,\left(  {\mathbb{E}}\left(
{\displaystyle\int_{0}^{T}}
e^{V_{r}}\big|\hat{H}(r,\hat{Y}_{r},\hat{Z}_{r})-\tilde{H}(r,\hat{Y}_{r}%
,\hat{Z}_{r})\big|dQ_{r}\right)  ^{q}\right)  ^{1/q}\right]  ^{\alpha}%
\end{array}
\label{cont-2}%
\end{equation}
where%
\begin{align*}
K  &  =\bigg[{\mathbb{E}}\left(  e^{qV_{T}}|\hat{\eta}|^{q}+~\Big(%
{\displaystyle\int_{0}^{T}}
e^{V_{r}}\big|\hat{H}\left(  r,0,0\right)  \big|dQ_{r}\Big)^{q}\right) \\
&  \quad+{\mathbb{E}}\left(  e^{qV_{T}}|\tilde{\eta}|^{q}+\Big(%
{\displaystyle\int_{0}^{T}}
e^{V_{r}}\big|\tilde{H}\left(  r,0,0\right)  \big|dQ_{r}\Big)^{q}\right)
\bigg]^{\left(  q-1\right)  /q}.
\end{align*}
and $C_{\alpha,q,\lambda}$ is a positive constant depending only on $\left(
\alpha,q,\lambda\right)  .$
\end{theorem}

\begin{proof}
Let $(\hat{Y},\hat{Z})$ and $(\tilde{Y},\tilde{Z})$ be two $L^{p}-$variational
solutions of (\ref{GBSVI 2}) corresponding to $(\hat{\eta},\hat{H})$ and
$(\tilde{\eta},\tilde{H})$ respectively, where $\hat{H}$ and $\tilde{H}$ have
the same constants functions $\mu,\nu,\ell$ as monotonicity and Lipschitz
coefficients. Let $M\in S_{m}^{q}\left(  \gamma,N,R;V\right)  $ and
\[
\hat{\Gamma}_{t}=\left(  |M_{t}-\hat{Y}_{t}|^{2}+\delta_{q}\right)
^{1/2}\quad\text{and}\quad\tilde{\Gamma}_{t}=\left(  |M_{t}-\tilde{Y}_{t}%
|^{2}+\delta_{q}\right)  ^{1/2}%
\]
Let the stopping times $0\leq\sigma\leq\theta\leq T$ such that
\[
\mathbb{E}\left(
{\displaystyle\int_{\sigma}^{\theta}}
{e^{V_{r}}}\left[  |\hat{H}(r,\hat{Y}_{r},\hat{Z}_{r})|+|\tilde{H}(r,\tilde
{Y}_{r},\tilde{Z}_{r})|\right]  dQ_{r}\right)  ^{q}<\infty
\]
From (\ref{def1d}) we deduce that $\mathbb{P}-a.s.$%
\begin{equation}%
\begin{array}
[c]{l}%
\left[  e^{qV_{\sigma}}(\hat{\Gamma}_{\sigma})^{q}+e^{qV_{\sigma}}%
(\tilde{\Gamma}_{\sigma})^{q}\right]  +q\,\mathbb{E}^{\mathcal{F}_{\sigma}}%
{\displaystyle\int_{\sigma}^{\theta}}
e^{qV_{r}}\left[  (\hat{\Gamma}_{r})^{q}+(\tilde{\Gamma}_{r})^{q}\right]
dV_{r}\\
\quad+\dfrac{q\left(  q-1\right)  }{2}\,\mathbb{E}^{\mathcal{F}_{\sigma}}{%
{\displaystyle\int_{\sigma}^{\theta}}
}e^{qV_{r}}\left[  (\hat{\Gamma}_{r})^{q-2}\left\vert R_{r}-\hat{Z}%
_{r}\right\vert ^{2}+(\tilde{\Gamma}_{r})^{q-2}\left\vert R_{r}-\tilde{Z}%
_{r}\right\vert ^{2}\right]  dr\medskip\\
\quad+{q\mathbb{E}^{\mathcal{F}_{\sigma}}%
{\displaystyle\int_{\sigma}^{\theta}}
}e^{qV_{r}}\left[  (\hat{\Gamma}_{r})^{q-2}{\Psi}\left(  r,\hat{Y}_{r}\right)
+(\tilde{\Gamma}_{r})^{q-2}{\Psi}\left(  r,\tilde{Y}_{r}\right)  \right]
dQ_{r}\medskip\\
\leq\mathbb{E}^{\mathcal{F}_{\sigma}}\left[  e^{qV_{\theta}}(\hat{\Gamma
}_{\theta})^{q}+e^{qV_{\theta}}(\tilde{\Gamma}_{\theta})^{q}\right]
+{q\mathbb{E}^{\mathcal{F}_{\sigma}}~%
{\displaystyle\int_{\sigma}^{\theta}}
}e^{qV_{r}}\left[  (\hat{\Gamma}_{r})^{q-2}+(\tilde{\Gamma}_{r})^{q-2}\right]
{\Psi}\left(  r,M_{r}\right)  dQ_{r}\medskip\\
\quad+{q~\mathbb{E}^{\mathcal{F}_{\sigma}}%
{\displaystyle\int_{\sigma}^{\theta}}
}e^{qV_{r}}(\hat{\Gamma}_{r})^{q-2}\langle M_{r}-\hat{Y}_{r},N_{r}-\hat
{H}(r,\hat{Y}_{r},\hat{Z}_{r})\rangle dQ_{r}\\
\quad+{q~\mathbb{E}^{\mathcal{F}_{\sigma}}%
{\displaystyle\int_{\sigma}^{\theta}}
}e^{qV_{r}}(\tilde{\Gamma}_{r})^{q-2}\langle M_{r}-\tilde{Y}_{r},N_{r}%
-\tilde{H}(r,\tilde{Y}_{r},\tilde{Z}_{r})\rangle dQ_{r}\,.
\end{array}
\label{uniq1}%
\end{equation}
Let $Y_{r}=\dfrac{1}{2}\left(  \hat{Y}_{r}+\tilde{Y}_{r}\right)  .$ We have
for all $\beta>0$%
\begin{align*}
|M-\hat{Y}|^{2}  &  \leq\frac{1+\beta}{\beta}\,|M-Y|^{2}+\frac{1+\beta}%
{4}\,|\hat{Y}-\tilde{Y}|^{2},\quad\text{and}\\
|M-\tilde{Y}|^{2}  &  \leq\frac{1+\beta}{\beta}\,|M-Y|^{2}+\frac{1+\beta}%
{4}\,|\hat{Y}-\tilde{Y}|^{2}%
\end{align*}
and taking in account that $1\leq q\leq2$ then%
\begin{align*}
&  (\hat{\Gamma})^{q-2}|R-\hat{Z}|^{2}+(\tilde{\Gamma})^{q-2}|R-\tilde{Z}%
|^{2}\\
&  =\left(  |M-\hat{Y}|^{2}+\delta_{q}\right)  ^{\left(  q-2\right)
/2}{\Large \,}|R-\hat{Z}|^{2}+\left(  |M-\tilde{Y}|^{2}+\delta_{q}\right)
^{\left(  q-2\right)  /2}{\Large \,}|R-\tilde{Z}|^{2}\\
&  \geq\left[  \frac{1+\beta}{\beta}\,|M-Y|^{2}+\frac{1+\beta}{4}\,|\hat
{Y}-\tilde{Y}|^{2}+\delta_{q}\right]  ^{\left(  q-2\right)  /2}\left[
|R-\hat{Z}|^{2}+|R-\tilde{Z}|^{2}\right] \\
&  \geq\frac{1}{2}\left[  \frac{1+\beta}{\beta}\,|M-Y|^{2}+\frac{1+\beta}%
{4}\,|\hat{Y}-\tilde{Y}|^{2}+\delta_{q}\right]  ^{\left(  q-2\right)  /2}%
|\hat{Z}-\tilde{Z}|^{2}.
\end{align*}
Hence%
\begin{equation}
(\hat{\Gamma})^{q-2}|R-\hat{Z}|^{2}+(\tilde{\Gamma})^{q-2}|R-\tilde{Z}%
|^{2}\geq\frac{1}{2}\left[  \frac{1+\beta}{\beta}\,|M-Y|^{2}+\frac{1+\beta}%
{4}\,|\hat{Y}-\tilde{Y}|^{2}+\delta_{q}\right]  ^{\left(  q-2\right)  /2}%
|\hat{Z}-\tilde{Z}|^{2}. \label{uniq1b}%
\end{equation}
Let $0<\varepsilon\leq1$ and%
\[
M_{t}^{\varepsilon}=\mathbb{E}^{\mathcal{F}_{t}}%
{\displaystyle\int_{t\vee\varepsilon}^{\infty}}
\dfrac{1}{Q_{\varepsilon}}\,e^{-\frac{Q_{r}-Q_{t\vee\varepsilon}%
}{Q_{\varepsilon}}}\,Y_{r}\,dQ_{r}\,,\quad t\in\left[  0,T\right]  .
\]
with $Y_{r}=\frac{\hat{Y}_{r}+\tilde{Y}_{r}}{2}\,.$

Then by Proposition \ref{p1-approx} $\left(  M^{\varepsilon},R^{\varepsilon
}\right)  \in S_{m}^{p}\times\Lambda_{m\times k}^{p}$ is the unique solution
of the BSDE:
\[
\left\{
\begin{array}
[c]{l}%
M_{t}^{\varepsilon}=M_{T}^{\varepsilon}+{%
{\displaystyle\int_{t}^{T}}
}1_{[\varepsilon,\infty)}\left(  r\right)  \dfrac{1}{Q_{\varepsilon}}\left(
Y_{r}-M_{r}^{\varepsilon}\right)  dQ_{r}-{%
{\displaystyle\int_{t}^{T}}
}R_{r}^{\varepsilon}dB_{r}\,,\quad t\in\left[  0,T\right]  ,\medskip\\
\lim\nolimits_{T\rightarrow\infty}\mathbb{E}~\left\vert M_{T}^{\varepsilon
}-\xi_{T}\right\vert ^{p}=0.
\end{array}
\right.
\]
and%
\[%
\begin{array}
[c]{ll}%
\left(  a\right)   & \left\vert M_{t}^{\varepsilon}\right\vert \leq
\mathbb{E}^{\mathcal{F}_{t}}~\sup\limits_{r\geq0}\left\vert Y_{r}\right\vert
,\quad\text{a.s., for all }t\in\left[  0,T\right]  ,\medskip\\
\left(  b\right)   & \lim\limits_{\varepsilon\rightarrow0}M_{t}^{\varepsilon
}=Y_{t}~,\;\;\mathbb{P}-a.s.,\;\text{for all }t\in\left[  0,T\right]
.\medskip\\
\left(  c\right)   & \lim\limits_{\varepsilon\rightarrow0}\mathbb{E}%
\sup\limits_{t\in\left[  0,T\right]  }\left\vert M_{t}^{\varepsilon}%
-Y_{t}\right\vert ^{p}=0.
\end{array}
\]
We replace $M$ by $M^{\varepsilon}$ in (\ref{uniq1}) and remark that%
\[%
\begin{array}
[c]{l}%
\displaystyle\langle M_{r}^{\varepsilon}-\hat{Y}_{r},N_{r}^{\varepsilon
}\rangle=\langle M_{r}^{\varepsilon}-\hat{Y}_{r},\frac{1}{Q_{\varepsilon}%
}\left(  Y_{r}-M_{r}^{\varepsilon}\right)  \rangle=\frac{1}{2Q_{\varepsilon}%
}\,\langle M_{r}^{\varepsilon}-\hat{Y}_{r},(\hat{Y}_{r}-M_{r}^{\varepsilon
})+(\tilde{Y}_{r}-M_{r}^{\varepsilon})\rangle\medskip\\
\displaystyle\leq\frac{1}{2Q_{\varepsilon}}\,\left[  -|M_{r}^{\varepsilon
}-\hat{Y}_{r}|^{2}+|M_{r}^{\varepsilon}-\hat{Y}_{r}||M_{r}^{\varepsilon
}-\tilde{Y}_{r}|\right]  =-\frac{1}{2Q_{\varepsilon}}\,\left[  |M_{r}%
^{\varepsilon}-\hat{Y}_{r}|-|M_{r}^{\varepsilon}-\tilde{Y}_{r}|\right]
|M_{r}^{\varepsilon}-\hat{Y}_{r}|
\end{array}
\]
and similar%
\[
\langle M_{r}^{\varepsilon}-\tilde{Y}_{r},N_{r}^{\varepsilon}\rangle\leq
\frac{1}{2Q_{\varepsilon}}\,\left[  |M_{r}^{\varepsilon}-\hat{Y}_{r}%
|-|M_{r}^{\varepsilon}-\tilde{Y}_{r}|\right]  |M_{r}^{\varepsilon}-\tilde
{Y}_{r}|.
\]
Hence%
\begin{equation}%
\begin{array}
[c]{l}%
\left[  (\hat{\Gamma}_{r}^{\varepsilon})^{q-2}\langle M_{r}^{\varepsilon}%
-\hat{Y}_{r},N_{r}^{\varepsilon}\rangle+(\tilde{\Gamma}_{r}^{\varepsilon
})^{q-2}\langle M_{r}^{\varepsilon}-\tilde{Y}_{r},N_{r}^{\varepsilon}%
\rangle\right]  \medskip\\
=\left(  |M_{r}^{\varepsilon}-\hat{Y}_{r}|^{2}+\delta_{q}\right)  ^{\left(
q-2\right)  /2}\langle M_{r}^{\varepsilon}-\hat{Y}_{r},N_{r}^{\varepsilon
}\rangle+\left(  |M_{r}^{\varepsilon}-\tilde{Y}_{r}|^{2}+\delta_{q}\right)
^{\left(  q-2\right)  /2}\langle M_{r}^{\varepsilon}-\tilde{Y}_{r}%
,N_{r}^{\varepsilon}\rangle\medskip\\
\leq\dfrac{-1}{2Q_{\varepsilon}}\,\left[  \left(  |M_{r}^{\varepsilon}-\hat
{Y}_{r}|^{2}+\delta_{q}\right)  ^{\left(  q-2\right)  /2}|M_{r}^{\varepsilon
}-\hat{Y}_{r}|-\left(  |M_{r}^{\varepsilon}-\tilde{Y}_{r}|^{2}+\delta
_{q}\right)  ^{\left(  q-2\right)  /2}|M_{r}^{\varepsilon}-\tilde{Y}%
_{r}|\right]  \medskip\\
\quad\cdot\left[  |M_{r}^{\varepsilon}-\hat{Y}_{r}|-|M_{r}^{\varepsilon
}-\tilde{Y}_{r}|\right]  \leq0
\end{array}
\label{uniq1c}%
\end{equation}
since for all $a,b\geq0$, $\delta\geq0$ and $\beta\geq-1/2$ we have%
\[
\left[  \left(  a^{2}+\delta\right)  ^{\beta}a-\left(  b^{2}+\delta\right)
^{\beta}b\right]  \left(  a-b\right)  \geq0.
\]
We use inequalities (\ref{uniq1b}) and (\ref{uniq1c}) in (\ref{uniq1}) for
$M=M^{\varepsilon}$ and it follows:%
\begin{equation}%
\begin{array}
[c]{l}%
\left[  e^{qV_{\sigma}}(\hat{\Gamma}_{\sigma}^{\varepsilon})^{q}%
+e^{qV_{\sigma}}(\tilde{\Gamma}_{\sigma}^{\varepsilon})^{q}\right]
+q~\mathbb{E}^{\mathcal{F}_{\sigma}}%
{\displaystyle\int_{\sigma}^{\theta}}
e^{qV_{r}}\left[  (\hat{\Gamma}_{r}^{\varepsilon})^{q}+(\tilde{\Gamma}%
_{r}^{\varepsilon})^{q}\right]  dV_{r}\medskip\\
+\dfrac{q\left(  q-1\right)  }{2}~\mathbb{E}^{\mathcal{F}_{\sigma}}{%
{\displaystyle\int_{\sigma}^{\theta}}
}e^{qV_{r}}\left[  \dfrac{1+\beta}{\beta}\,|M^{\varepsilon}-Y|^{2}%
+\dfrac{1+\beta}{4}\,|\hat{Y}_{r}-\tilde{Y}_{r}|^{2}+\delta_{q}\right]
^{\left(  q-2\right)  /2}|\hat{Z}_{r}-\tilde{Z}_{r}|^{2}dr\medskip\\
\quad+{q~\mathbb{E}^{\mathcal{F}_{\sigma}}%
{\displaystyle\int_{\sigma}^{\theta}}
}e^{qV_{r}}\left[  (\hat{\Gamma}_{r}^{\varepsilon})^{q-2}{\Psi(}r,\hat{Y}%
_{r})+(\tilde{\Gamma}_{r}^{\varepsilon})^{q-2}{\Psi(}r,\tilde{Y}_{r})\right]
dQ_{r}\medskip\\
\leq~\mathbb{E}^{\mathcal{F}_{\sigma}}\left[  e^{qV_{\theta}}(\hat{\Gamma
}_{\theta}^{\varepsilon})^{q}+e^{qV_{\theta}}(\tilde{\Gamma}_{\theta
}^{\varepsilon})^{q}\right]  +{q~\mathbb{E}^{\mathcal{F}_{\sigma}}%
{\displaystyle\int_{\sigma}^{\theta}}
}e^{qV_{r}}\left[  (\hat{\Gamma}_{r}^{\varepsilon})^{q-2}+(\tilde{\Gamma}%
_{r}^{\varepsilon})^{q-2}\right]  {\Psi}\left(  r,M_{r}^{\varepsilon}\right)
dQ_{r}\medskip\\
\quad+{q~\mathbb{E}^{\mathcal{F}_{\sigma}}%
{\displaystyle\int_{\sigma}^{\theta}}
}e^{qV_{r}}\left[  (\hat{\Gamma}_{r}^{\varepsilon})^{q-2}\langle
M_{r}^{\varepsilon}-\hat{Y}_{r},-\hat{H}(r,\hat{Y}_{r},\hat{Z}_{r}%
)\rangle+(\tilde{\Gamma}_{r}^{\varepsilon})^{q-2}\langle M_{r}^{\varepsilon
}-\tilde{Y}_{r},-\tilde{H}(r,\tilde{Y}_{r},\tilde{Z}_{r})\rangle\right]
dQ_{r}.
\end{array}
\label{uniq1d}%
\end{equation}
Let $0\leq t\leq s\leq T$, $0\leq u\leq r\leq v$ and the stopping times
$v^{\ast}=Q_{v}^{-1},$ $u^{\ast}=Q_{u}^{-1},$ $r^{\ast}=Q_{r}^{-1}$ , where
$Q_{\cdot}^{-1}\left(  \omega\right)  $ is the inverse mapping of the function
$r\longmapsto Q_{r}\left(  \omega\right)  :[0,\infty)\rightarrow
\lbrack0,\infty)$ and, for each $k,i\in\mathbb{N}^{\ast},$ the stopping times%
\[%
\begin{array}
[c]{l}%
\alpha_{k}=\inf\Big\{u\geq0:\left\updownarrow V\right\updownarrow _{u\wedge
T}+\sup\limits_{r\in\left[  0,u\wedge T\right]  }|e^{V_{r}}\hat{Y}_{r}-\hat
{Y}_{0}|+\sup\limits_{r\in\left[  0,u\wedge T\right]  }|e^{V_{r}}\tilde{Y}%
_{r}-\tilde{Y}_{0}|+%
{\displaystyle\int_{0}^{u\wedge T}}
e^{2V_{r}}|\hat{Z}_{r}|^{2}dr\medskip\\
\quad\quad\quad\quad+%
{\displaystyle\int_{0}^{u\wedge T}}
e^{2V_{r}}|\tilde{Z}_{r}|^{2}dr+%
{\displaystyle\int_{0}^{u\wedge T}}
e^{V_{r}}|\hat{H}(r,\hat{Y}_{r},\hat{Z}_{r})|dQ_{r}+%
{\displaystyle\int_{0}^{u\wedge T}}
e^{V_{r}}|\tilde{H}(r,\tilde{Y}_{r},\tilde{Z}_{r})|dQ_{r}\medskip\\
\quad\quad\quad\quad+%
{\displaystyle\int_{0}^{u\wedge T}}
e^{2V_{r}}{\Psi(}r,\hat{Y}_{r})dQ_{r}+%
{\displaystyle\int_{0}^{u\wedge T}}
e^{2V_{r}}{\Psi(}r,\tilde{Y}_{r})dQ_{r}\geq k\Big\}.
\end{array}
\]
and
\[
u_{k}^{\ast}:=t\wedge u^{\ast}\wedge\alpha_{k}\quad\quad\text{and}\quad\quad
v_{k+i}^{\ast}:=s\wedge v^{\ast}\wedge\alpha_{k+i}~.
\]
We put in (\ref{uniq1d})
\[
\sigma=u_{k}^{\ast}~\quad\quad\text{and}\quad\quad\theta=v_{k+i}^{\ast}%
\]
Now passing to $\liminf_{\varepsilon\searrow0}$ in (\ref{uniq1d}) we obtain
(using Proposition \ref{p1-approx}-$\left(  3.\right)  ,$ Fatou's Lemma and
Lebesgue dominated convergence theorem):\medskip\newline$%
\begin{array}
[c]{c}%
\quad
\end{array}
2~e^{qV_{u_{k}^{\ast}}}\left(  \dfrac{1}{4}\,\big|\hat{Y}_{u_{k}^{\ast}%
}-\tilde{Y}_{u_{k}^{\ast}}\big|^{2}+\delta_{q}\right)  ^{q/2}+2q~{\mathbb{E}%
}^{\mathcal{F}_{u_{k}^{\ast}}}{%
{\displaystyle\int_{u_{k}^{\ast}}^{v_{k+i}^{\ast}}}
e^{qV_{r}}}\left(  \dfrac{1}{4}\,\big|\hat{Y}_{r}-\tilde{Y}_{r}\big|^{2}%
+\delta_{q}\right)  ^{q/2}dV_{r}$\medskip\newline$%
\begin{array}
[c]{c}%
\quad\quad
\end{array}
+\dfrac{q\left(  q-1\right)  }{2}~{\mathbb{E}}^{\mathcal{F}_{u_{k}^{\ast}}}~{%
{\displaystyle\int_{u_{k}^{\ast}}^{v_{k+i}^{\ast}}}
}e^{qV_{r}}\left(  \dfrac{1+\beta}{4}\,|\hat{Y}_{r}-\tilde{Y}_{r}|^{2}%
+\delta_{q}\right)  ^{\left(  q-2\right)  /2}|\hat{Z}_{r}-\tilde{Z}_{r}%
|^{2}dr$\medskip\newline$%
\begin{array}
[c]{c}%
\quad\quad
\end{array}
+q~{\mathbb{E}}^{\mathcal{F}_{u_{k}^{\ast}}}~{%
{\displaystyle\int_{u_{k}^{\ast}}^{v_{k+i}^{\ast}}}
}e^{qV_{r}}\left(  \dfrac{1}{4}\,|\hat{Y}_{r}-\tilde{Y}_{r}|^{2}+\delta
_{q}\right)  ^{\left(  q-2\right)  /2}\left[  \Psi(r,\hat{Y}_{r}%
)+\Psi(r,\tilde{Y}_{r})\right]  {dQ_{r}}$\medskip\newline$%
\begin{array}
[c]{c}%
\quad
\end{array}
\leq2~{\mathbb{E}}^{\mathcal{F}_{u_{k}^{\ast}}}~e^{qV_{v_{k+i}^{\ast}}}\left(
\dfrac{1}{4}\,\big|\hat{Y}_{v_{k+i}^{\ast}}-\tilde{Y}_{v_{k+i}^{\ast}%
}\big|^{2}+\delta_{q}\right)  ^{q/2}$\medskip\newline$%
\begin{array}
[c]{c}%
\quad\quad
\end{array}
+q~{\mathbb{E}}^{\mathcal{F}_{u_{k}^{\ast}}}~{%
{\displaystyle\int_{u_{k}^{\ast}}^{v_{k+i}^{\ast}}}
}e^{qV_{r}}\left(  \dfrac{1}{4}\,|\hat{Y}_{r}-\tilde{Y}_{r}|^{2}+\delta
_{q}\right)  ^{\left(  q-2\right)  /2}2{\Psi}\left(  r,Y_{r}\right)  {dQ_{r}}%
$\medskip\newline$%
\begin{array}
[c]{c}%
\quad\quad
\end{array}
+\dfrac{{q}}{2}~{\mathbb{E}}^{\mathcal{F}_{u_{k}^{\ast}}}~{%
{\displaystyle\int_{u_{k}^{\ast}}^{v_{k+i}^{\ast}}}
}e^{qV_{r}}\left(  \dfrac{1}{4}\,|\hat{Y}_{r}-\tilde{Y}_{r}|^{2}+\delta
_{q}\right)  ^{\left(  q-2\right)  /2}\langle\hat{Y}_{r}-\tilde{Y}_{r},\hat
{H}(r,\hat{Y}_{r},\hat{Z}_{r})-\tilde{H}(r,\tilde{Y}_{r},\tilde{Z}_{r})\rangle
dQ_{r}~.$\medskip\newline By by Fatou's Lemma passing to limit as
$\beta\rightarrow0$ this last inequality is true for $\beta=0.$

We remark now that%
\[
2{\Psi}\left(  r,Y_{r}\right)  =2{\Psi}\Big(r,\frac{1}{2}\hat{Y}_{r}+\frac
{1}{2}\tilde{Y}_{r}\Big)\leq\Psi(r,\hat{Y}_{r})+\Psi(r,\tilde{Y}_{r}).
\]
and%
\begin{align*}
&  \langle\hat{Y}_{r}-\tilde{Y}_{r},\hat{H}(r,\hat{Y}_{r},\hat{Z}_{r}%
)-\tilde{H}(r,\hat{Y}_{r},\hat{Z}_{r})+\tilde{H}(r,\hat{Y}_{r},\hat{Z}%
_{r})-\tilde{H}(r,\tilde{Y}_{r},\tilde{Z}_{r})\rangle dQ_{r}\\
&  \leq\langle\hat{Y}_{r}-\tilde{Y}_{r},\hat{H}(r,\hat{Y}_{r},\hat{Z}%
_{r})-\tilde{H}(r,\hat{Y}_{r},\hat{Z}_{r})\rangle dQ_{r}+|\hat{Y}_{r}%
-\tilde{Y}_{r}|^{2}dV_{r}+\dfrac{n_{p}\lambda}{2}\,|\hat{Z}_{r}-\tilde{Z}%
_{r}|^{2}dr.
\end{align*}
Hence%
\begin{equation}%
\begin{array}
[c]{l}%
2\mathbb{E}~e^{qV_{u_{k}^{\ast}}}\left(  \dfrac{1}{4}\,\big|\hat{Y}%
_{u_{k}^{\ast}}-\tilde{Y}_{u_{k}^{\ast}}\big|^{2}+\delta_{q}\right)
^{q/2}+2q\delta_{q}~{\mathbb{E}}^{\mathcal{F}_{u_{k}^{\ast}}}~{%
{\displaystyle\int_{u_{k}^{\ast}}^{v_{k+i}^{\ast}}}
e^{qV_{r}}}\left(  \dfrac{1}{4}\,|\hat{Y}_{r}-\tilde{Y}_{r}|^{2}+\delta
_{q}\right)  ^{\left(  q-2\right)  /2}dV_{r}\medskip\\
\quad+\dfrac{q}{2}\,\left(  q-1-\dfrac{n_{p}\lambda}{2}\right)  ~{\mathbb{E}%
}^{\mathcal{F}_{u_{k}^{\ast}}}~{%
{\displaystyle\int_{u_{k}^{\ast}}^{v_{k+i}^{\ast}}}
}e^{qV_{r}}\left(  \dfrac{1}{4}\,|\hat{Y}_{r}-\tilde{Y}_{r}|^{2}+\delta
_{q}\right)  ^{\left(  q-2\right)  /2}|\hat{Z}_{r}-\tilde{Z}_{r}%
|^{2}dr\medskip\\
\leq2~{\mathbb{E}}^{\mathcal{F}_{u_{k}^{\ast}}}~e^{qV_{v_{k+i}^{\ast}}}\left(
\dfrac{1}{4}\,\big|\hat{Y}_{v_{k+i}^{\ast}}-\tilde{Y}_{v_{k+i}^{\ast}%
}\big|^{2}+\delta_{q}\right)  ^{q/2}\medskip\\
\quad+\dfrac{{q}}{2}~{\mathbb{E}}^{\mathcal{F}_{u_{k}^{\ast}}}~{%
{\displaystyle\int_{u_{k}^{\ast}}^{v_{k+i}^{\ast}}}
}e^{qV_{r}}\left(  \dfrac{1}{4}\,|\hat{Y}_{r}-\tilde{Y}_{r}|^{2}+\delta
_{q}\right)  ^{\left(  q-2\right)  /2}\langle\hat{Y}_{r}-\tilde{Y}_{r},\hat
{H}(r,\hat{Y}_{r},\hat{Z}_{r})-\tilde{H}(r,\hat{Y}_{r},\hat{Z}_{r})\rangle
dQ_{r},
\end{array}
\label{uniq1e}%
\end{equation}
for all $\delta>0$ with $\delta_{q}=\delta\mathbf{1}_{[1,2)}\left(  q\right)
.$

Passing to limit as $\delta\rightarrow0_{+}$ and taking in account that
$q-1=n_{p}$ for $q=p\wedge2$ we get%
\begin{equation}%
\begin{array}
[c]{l}%
~e^{qV_{u_{k}^{\ast}}}\big|\hat{Y}_{u_{k}^{\ast}}-\tilde{Y}_{u_{k}^{\ast}%
}\big|^{q}+\dfrac{1}{2}\,q\left(  q-1\right)  \left(  2-\lambda\right)
~{\mathbb{E}}^{\mathcal{F}_{u_{k}^{\ast}}}~{%
{\displaystyle\int_{u_{k}^{\ast}}^{v_{k+i}^{\ast}}}
}e^{qV_{r}}|\hat{Y}_{r}-\tilde{Y}_{r}|^{q-2}|\hat{Z}_{r}-\tilde{Z}_{r}%
|^{2}dr\medskip\\
\leq{\mathbb{E}}^{\mathcal{F}_{u_{k}^{\ast}}}~e^{qV_{v_{k+i}^{\ast}}}%
\big|\hat{Y}_{v_{k+i}^{\ast}}-\tilde{Y}_{v_{k+i}^{\ast}}\big|^{q}%
+q~{\mathbb{E}}^{\mathcal{F}_{u_{k}^{\ast}}}{%
{\displaystyle\int_{u_{k}^{\ast}}^{v_{k+i}^{\ast}}}
}e^{qV_{r}}|\hat{Y}_{r}-\tilde{Y}_{r}|^{q-1}\big|\hat{H}(r,\hat{Y}_{r},\hat
{Z}_{r})-\tilde{H}(r,\hat{Y}_{r},\hat{Z}_{r})\big|dQ_{r}~.
\end{array}
\label{uniq-33}%
\end{equation}

We remark that for any nonnegative measurable function $K$ we have
\[
{%
{\displaystyle\int_{u_{k}^{\ast}}^{v_{k+i}^{\ast}}}
}e^{V_{r}}K_{r}dQ_{r}=\int_{t\wedge u^{\ast}\wedge\alpha_{k}}^{s\wedge
v^{\ast}\wedge\alpha_{k+i}}K_{r}dQ_{r}\leq\int_{t\wedge u^{\ast}\wedge
\alpha_{k}}^{s}K_{r}dQ_{r}.
\]
Passing first to $\lim_{i\rightarrow\infty}$ and then $\lim_{k\rightarrow
\infty}$ in (\ref{uniq-33}) we infer (using Fatou's Lemma and Lebesgue
dominated convergence theorem via the condition (\ref{def0-1})):%
\begin{equation}%
\begin{array}
[c]{l}%
e^{qV_{t\wedge u^{\ast}}}\big|\hat{Y}_{t\wedge u^{\ast}}-\tilde{Y}_{t\wedge
u^{\ast}}\big|^{q}+\dfrac{1}{2}\,q\left(  q-1\right)  \left(  2-\lambda
\right)  \,{\mathbb{E}}^{\mathcal{F}_{t\wedge u^{\ast}}}~{%
{\displaystyle\int_{t\wedge u^{\ast}}^{s}}
}\dfrac{1}{\left(  e^{V_{r}}|\hat{Y}_{r}-\tilde{Y}_{r}|+1\right)  ^{2-q}%
}\,e^{2V_{r}}|\hat{Z}_{r}-\tilde{Z}_{r}|^{2}dr\medskip\\
\leq{\mathbb{E}}^{\mathcal{F}_{t\wedge u^{\ast}}}~e^{qV_{s\wedge v^{\ast}}%
}\big|\hat{Y}_{s\wedge v^{\ast}}-\tilde{Y}_{s\wedge v^{\ast}}\big|^{q}%
\medskip\\
\quad+q\,{\mathbb{E}}^{\mathcal{F}_{t\wedge u^{\ast}}}~{%
{\displaystyle\int_{t\wedge u^{\ast}}^{s}}
}e^{qV_{r}}|\hat{Y}_{r}-\tilde{Y}_{r}|^{q-1}\big|\hat{H}(r,\hat{Y}_{r},\hat
{Z}_{r})-\tilde{H}(r,\hat{Y}_{r},\hat{Z}_{r})\big|dQ_{r}~\text{,}%
\quad\mathbb{P}\text{-a.s.}%
\end{array}
\label{uniq-44}%
\end{equation}
(we also used the continuity of the natural filtration $\left\{
\mathcal{F}_{u}:u\geq0\right\}  $).

Moreover, by the Lebesgue dominated convergence theorem (based on
(\ref{def0-1}) ) passing to $\lim_{u,v\rightarrow\infty}$ it follows that for
all $0\leq t\leq s\leq T,$ $\mathbb{P}$--a.s.,%
\begin{equation}%
\begin{array}
[c]{l}%
e^{qV_{t}}|\hat{Y}_{t}-\tilde{Y}_{t}|^{q}+\dfrac{1}{2}\,q\left(  q-1\right)
\left(  2-\lambda\right)  \,{\mathbb{E}}^{\mathcal{F}_{t}}{%
{\displaystyle\int_{t}^{s}}
}\dfrac{1}{\left(  e^{V_{r}}|\hat{Y}_{r}-\tilde{Y}_{r}|+1\right)  ^{2-q}%
}e^{2V_{r}}|\hat{Z}_{r}-\tilde{Z}_{r}|^{2}dr\medskip\\
\leq{\mathbb{E}}^{\mathcal{F}_{t}}~e^{qV_{s}}|\hat{Y}_{s}-\tilde{Y}_{s}%
|^{q}+q\,{\mathbb{E}}^{\mathcal{F}_{t}}~{%
{\displaystyle\int_{t}^{s}}
}e^{qV_{r}}|\hat{Y}_{r}-\tilde{Y}_{r}|^{q-1}\big|\hat{H}(r,\hat{Y}_{r},\hat
{Z}_{r})-\tilde{H}(r,\hat{Y}_{r},\hat{Z}_{r})\big|dQ_{r}\,.
\end{array}
\label{uniq-66}%
\end{equation}
Using here Proposition \ref{an-prop-yd} and inequality (\ref{def-11ccc}) for
$|\hat{Y}-\tilde{Y}|^{q}\leq2^{\left(  q-1\right)  }\left(  |\hat{Y}%
|^{q}+|\tilde{Y}|^{q}\right)  $ we easily obtain (\ref{cont-2}).$\medskip$

The uniqueness property follows, since, if $\hat{\eta}=\tilde{\eta}$ and
$\hat{H}=\tilde{H}$ we have from (\ref{cont-2}), $\hat{Y}=\tilde{Y}$ in
$S_{m}^{0}$ and $\hat{Z}=\tilde{Z}$ in $\Lambda_{m\times k}^{0}\,.$\hfill
\end{proof}

\subsection{Existence of the solution}

\begin{lemma}
[Strong solution]\label{l1-strong sol}We suppose that assumptions $\left(
\mathrm{A}_{1}-\mathrm{A}_{6}\right)  $ are satisfied. Let $0<\lambda<1<p$,
$n_{p}=\left(  p-1\right)  \wedge1$ and%
\begin{equation}
V_{t}^{\left(  +\right)  }%
\xlongequal{\hspace{-3pt}\textrm{def}\hspace{-3pt}}\int_{0}^{t}\left[  \left(
\mu_{r}+\frac{1}{2n_{p}\lambda}\,\ell_{r}^{2}\right)  ^{+}dr+\nu_{r}^{+}%
dA_{r}\right]  . \label{defV_2}%
\end{equation}
Moreover, we assume:

\begin{enumerate}
\item[$\left(  i\right)  $] there exists $L>0$ such that%
\begin{equation}
\left\vert \eta\right\vert +\ell_{t}+F_{1}^{\#}\left(  t\right)  +G_{1}%
^{\#}\left(  t\right)  \leq L,\quad\text{a.e. }t\in\left[  0,T\right]
,\;\text{a.s.;} \label{ap-eq-1b}%
\end{equation}

\item[$\left(  ii\right)  $] there exists $\delta>0$ such that%
\begin{equation}
\mathbb{E}\exp\left[  \left(  2+\delta\right)  V_{T}^{\left(  +\right)
}\right]  <\infty; \label{ap-eq-1bb}%
\end{equation}

\item[$\left(  iii\right)  $] there exists $\tilde{L}>0$ such that%
\begin{equation}
\big|e^{V_{T}^{\left(  +\right)  }}\eta\big|^{2}+\Big(%
{\displaystyle\int_{0}^{T}}
e^{V_{r}^{\left(  +\right)  }}\left(  F_{1}^{\#}\left(  r\right)
dr+G_{1}^{\#}\left(  r\right)  dA_{r}\right)  \Big)^{2}\leq\tilde{L}%
,\quad\mathbb{P}-\text{a.s.;} \label{ap-eq-1c}%
\end{equation}

\item[$\left(  iv\right)  $] for $\rho_{0}=(C_{\lambda}\tilde{L})^{1/2}%
>0$\ \footnote{The constant $\tilde{L}$ is given by (\ref{ap-eq-1c}) and the
constant $C_{\lambda}=C_{p,\lambda}$ is given by (\ref{an3a}) with $p=2.$} it
holds
\begin{equation}
\mathbb{E}%
{\displaystyle\int_{0}^{T}}
e^{2V_{r}^{\left(  +\right)  }}\left[  \left(  F_{1+\rho_{0}}^{\#}\left(
r\right)  \right)  ^{2}dr+\left(  G_{1+\rho_{0}}^{\#}\left(  r\right)
\right)  ^{2}dA_{r}\right]  <\infty. \label{ap-eq-6}%
\end{equation}

\end{enumerate}

Then the multivalued BSDE%
\[
\left\{
\begin{array}
[c]{l}%
\displaystyle Y_{t}+{\int_{t}^{T}}dK_{r}=Y_{T}+{\int_{t}^{T}}H\left(
r,Y_{r},Z_{r}\right)  dQ_{r}-{\int_{t}^{T}}Z_{r}dB_{r},\;\text{a.s., for all
}t\in\left[  0,T\right]  ,\medskip\\
\displaystyle dK_{r}=U_{r}^{\left(  1\right)  }dr+U_{r}^{\left(  2\right)
}dA_{r}\,,\medskip\\
\displaystyle U^{\left(  1\right)  }dr\in\partial\varphi\left(  Y_{r}\right)
dr\quad\text{and}\quad U^{\left(  2\right)  }dA_{r}\in\partial\psi\left(
Y_{r}\right)  dA_{r}%
\end{array}
\,\right.
\]
has a strong a solution $\left(  Y,Z,U^{\left(  1\right)  },U^{\left(
2\right)  }\right)  \in S_{m}^{0}\times\Lambda_{m\times k}^{0}\times
\Lambda_{m}^{0}$ $\times\Lambda_{m}^{0}$ such that
\[
\mathbb{E}\sup_{t\in\left[  0,T\right]  }{e^{2V_{r}^{\left(  +\right)  }}%
}\left\vert Y_{r}\right\vert ^{2}+\mathbb{E}%
{\displaystyle\int_{0}^{T}}
{e^{2V_{r}^{\left(  +\right)  }}}\left\vert Z_{r}\right\vert ^{2}dr+\mathbb{E}%
{\displaystyle\int_{0}^{T}}
{e^{2V_{r}^{\left(  +\right)  }}}\left\vert U_{r}^{\left(  1\right)
}\right\vert ^{2}dr+\mathbb{E}%
{\displaystyle\int_{0}^{T}}
{e^{2V_{r}^{\left(  +\right)  }}}\left\vert U_{r}^{\left(  2\right)
}\right\vert ^{2}dA_{r}<\infty.
\]

\end{lemma}

\begin{proof}
Let $0<\varepsilon\leq1.$ We consider the approximating backward stochastic
equation%
\begin{equation}
\displaystyle Y_{t}^{\varepsilon}+{\int_{t}^{T}}\nabla_{y}\Psi^{\varepsilon
}(r,Y_{r}^{\varepsilon})dQ_{r}=\eta+{\int_{t}^{T}}H_{\varepsilon}%
(r,Y_{r}^{\varepsilon},Z_{r}^{\varepsilon})dQ_{r}-{\int_{t}^{T}}%
Z_{r}^{\varepsilon}\,dB_{r}\,,\quad\mathbb{P}\text{--a.s., }t\in\left[
0,T\right]  , \label{ap-eq -1}%
\end{equation}
where%
\begin{equation}%
\begin{array}
[c]{l}%
\displaystyle\Psi^{\varepsilon}\left(  \omega,r,y\right)  :=\alpha_{r}\left(
\omega\right)  \varphi_{\varepsilon}\left(  y\right)  +\left(  1-\alpha
_{r}\left(  \omega\right)  \right)  \psi_{\varepsilon}\left(  y\right)
\medskip\\
\displaystyle\nabla_{y}\Psi^{\varepsilon}\left(  \omega,r,y\right)  =\left[
\alpha_{r}\left(  \omega\right)  \nabla_{y}\varphi_{\varepsilon}\left(
y\right)  +\left(  1-\alpha_{r}\left(  \omega\right)  \right)  \nabla_{y}%
\psi_{\varepsilon}\left(  y\right)  \right]  \,\mathbf{1}_{[0,\frac
{1}{\varepsilon}]}\left(  A_{r}\right)  ,\medskip\\
\displaystyle H_{\varepsilon}\left(  \omega,r,y,z\right)  :=\left[  \alpha
_{r}\left(  \omega\right)  F_{\varepsilon}\left(  \omega,r,y,z\right)
+\left(  1-\alpha_{r}\left(  \omega\right)  \right)  G_{\varepsilon}\left(
\omega,r,y\right)  \right]  \,\mathbf{1}_{[0,\frac{1}{\varepsilon}]}\left(
A_{r}\right)  ,
\end{array}
\label{ap-eq-1a}%
\end{equation}
with $\varphi_{\varepsilon}$ and $\psi_{\varepsilon}$ being the
Maureau-Yosida's regularization given by (\ref{fi-MY}) and $F_{\varepsilon
},G_{\varepsilon}$ being the mollifier approximations introduced in the
Appendix Section \ref{Annex-MA}.

By (\ref{ma-1}) and (\ref{fi-lip}), the function%
\[
\Phi_{\varepsilon}\left(  \omega,r,y,z\right)  :=H_{\varepsilon}\left(
\omega,r,y,z\right)  -\nabla_{y}\Psi^{\varepsilon}\left(  \omega,r,y\right)
\]
is a Lipschitz function:$\medskip$\newline$%
\begin{array}
[c]{c}%
\text{\quad}%
\end{array}
\left\vert \Phi_{\varepsilon}\left(  \omega,r,y,z\right)  -\Phi_{\varepsilon
}\left(  \omega,r,\hat{y},\hat{z}\right)  \right\vert \medskip$\newline$%
\begin{array}
[c]{c}%
\text{\quad}%
\end{array}
\leq\bigg[\alpha_{r}\left(  \omega\right)  \left(  \mathcal{\ell}%
_{t}\left\vert z-\hat{z}\right\vert +\dfrac{\kappa\left(  1+\ell_{t}\right)
}{\varepsilon^{2}}\left\vert y-\hat{y}\right\vert \right)  +\left(
1-\alpha_{r}\left(  \omega\right)  \right)  \dfrac{\kappa}{\varepsilon^{2}%
}\left\vert y-\hat{y}\right\vert \medskip$\newline$%
\begin{array}
[c]{c}%
\text{\quad\quad}%
\end{array}
+\dfrac{1}{\varepsilon}\alpha_{r}\left(  \omega\right)  \left\vert y-\hat
{y}\right\vert +\dfrac{1}{\varepsilon}\left(  1-\alpha_{r}\left(
\omega\right)  \right)  \left\vert y-\hat{y}\right\vert \bigg]\,\mathbf{1}%
_{[0,\frac{1}{\varepsilon}]}\left(  A_{r}\right)  \medskip$\newline$%
\begin{array}
[c]{c}%
\text{\quad}%
\end{array}
\leq\left[  \alpha_{r}\left(  \omega\right)  \dfrac{\kappa L+\kappa
+1}{\varepsilon^{2}}\left\vert y-\hat{y}\right\vert +\left(  1-\alpha
_{r}\left(  \omega\right)  \right)  \dfrac{\kappa+1}{\varepsilon^{2}%
}\left\vert y-\hat{y}\right\vert +\alpha_{r}\left(  \omega\right)  L\left\vert
z-\hat{z}\right\vert \right]  \,\mathbf{1}_{[0,\frac{1}{\varepsilon}]}\left(
A_{r}\right)  \medskip$\newline$%
\begin{array}
[c]{c}%
\text{\quad\quad}%
\end{array}
\leq\left[  \dfrac{\kappa L+\kappa+1}{\varepsilon^{2}}\left\vert y-\hat
{y}\right\vert +\alpha_{r}\left(  \omega\right)  L\left\vert z-\hat
{z}\right\vert \right]  \,\mathbf{1}_{[0,\frac{1}{\varepsilon}]}\left(
A_{r}\right)  .\medskip$\newline The assumptions of \cite[Lemma 5.20]%
{pa-ra/14} are satisfied for all $p^{\prime}\geq2.$ Therefore equation
(\ref{ap-eq -1}) has a unique solution $\left(  Y^{\varepsilon},Z^{\varepsilon
}\right)  \in S_{m}^{p^{\prime}}\left[  0,T\right]  \times\Lambda_{m\times
k}^{p^{\prime}}\left(  0,T\right)  $ and consequently, for all $p^{\prime}%
\geq2,$%
\[
\mathbb{E}\sup_{t\in\left[  0,T\right]  }\left\vert Y_{t}^{\varepsilon
}\right\vert ^{p^{\prime}}<\infty.
\]
Remark that, by (\ref{ma-2}),%
\begin{align*}
&  \left\langle Y_{t}^{\varepsilon},\Phi_{\varepsilon}\left(  t,Y_{t}%
^{\varepsilon},Z_{t}^{\varepsilon}\right)  \right\rangle dQ_{t}\\
&  =\left\langle Y_{t}^{\varepsilon},F_{\varepsilon}\left(  t,Y_{t}%
^{\varepsilon},Z_{t}^{\varepsilon}\right)  \,\mathbf{1}_{[0,\frac
{1}{\varepsilon}]}\left(  A_{r}\right)  \right\rangle dt+\left\langle
Y_{t}^{\varepsilon},G_{\varepsilon}\left(  t,Y_{t}^{\varepsilon}\right)
\,\mathbf{1}_{[0,\frac{1}{\varepsilon}]}\left(  A_{r}\right)  \right\rangle
dA_{t}\\
&  \quad-\left\langle Y_{t}^{\varepsilon},\nabla\varphi_{\varepsilon}\left(
t,Y_{t}^{\varepsilon}\right)  \right\rangle \,\mathbf{1}_{[0,\frac
{1}{\varepsilon}]}\left(  A_{r}\right)  \,dt-\left\langle Y_{t}^{\varepsilon
},\nabla\psi_{\varepsilon}\left(  t,Y_{t}^{\varepsilon}\right)  \right\rangle
\,\mathbf{1}_{[0,\frac{1}{\varepsilon}]}\left(  A_{r}\right)  \,dA_{t}\\
&  \leq\left[  \left\vert Y_{t}^{\varepsilon}\right\vert F_{1}^{\#}\left(
t\right)  +\left(  \mu_{t}+\frac{1}{2n_{p}\lambda}\,\ell_{t}^{2}\right)
^{+}\left\vert Y_{t}^{\varepsilon}\right\vert ^{2}+\dfrac{n_{p}\lambda}%
{2}\left\vert Z_{t}^{\varepsilon}\right\vert ^{2}\right]  dt+\left[
\left\vert Y_{t}^{\varepsilon}\right\vert G_{1}^{\#}\left(  t\right)  +\nu
_{t}^{+}\left\vert Y_{t}^{\varepsilon}\right\vert ^{2}\right]  dA_{t}\\
&  \leq\left\vert Y_{t}^{\varepsilon}\right\vert \bar{H}_{1}^{\#}\left(
t\right)  dQ_{t}+\left\vert Y_{t}^{\varepsilon}\right\vert ^{2}dV_{t}^{\left(
+\right)  }+\dfrac{\lambda}{2}\left\vert Z_{t}^{\varepsilon}\right\vert
^{2}dt,
\end{align*}
where%
\[
\bar{H}_{1}^{\#}\left(  t\right)
\xlongequal{\hspace{-3pt}\textrm{def}\hspace{-3pt}}\alpha_{t}F_{1}^{\#}\left(
t\right)  +\left(  1-\alpha_{t}\right)  G_{1}^{\#}\left(  t\right)  .
\]
Since by (\ref{ap-eq-1bb})%
\begin{align*}
\mathbb{E}\sup_{t\in\left[  0,T\right]  }e^{2V_{t}^{\left(  +\right)  }%
}\left\vert Y_{t}^{\varepsilon}\right\vert ^{2}  &  \leq\mathbb{E}\left[
\left(  \exp2V_{T}^{\left(  +\right)  }\right)  \sup_{t\in\left[  0,T\right]
}\left\vert Y_{t}^{\varepsilon}\right\vert ^{2}\right] \\
&  \leq\left[  \frac{2}{2+\delta}\,\mathbb{E}\exp\left(  2+\delta\right)
V_{T}^{\left(  +\right)  }+\frac{\delta}{2+\delta}\,\mathbb{E}\sup
_{t\in\left[  0,T\right]  }\left\vert Y_{t}^{\varepsilon}\right\vert ^{\left(
4+2\delta\right)  /\delta}\right]  <\infty,
\end{align*}
by Proposition \ref{Appendix_result 1} we have%
\[%
\begin{array}
[c]{l}%
\mathbb{E}^{\mathcal{F}_{t}}\Big(\sup\limits_{r\in\left[  t,T\right]
}\big|e^{V_{r}^{\left(  +\right)  }}Y_{r}^{\varepsilon}\big|^{2}%
\Big)+\mathbb{E}^{\mathcal{F}_{T}}\Big(%
{\displaystyle\int_{t}^{T}}
e^{2V_{r}^{\left(  +\right)  }}\left\vert Z_{r}^{\varepsilon}\right\vert
^{2}dr\Big)\medskip\\
\leq C_{\lambda}~\mathbb{E}^{\mathcal{F}_{t}}\left[  \big|e^{V_{T}^{\left(
+\right)  }}\eta\big|^{2}+\Big(%
{\displaystyle\int_{t}^{T}}
e^{V_{r}^{\left(  +\right)  }}\bar{H}_{1}^{\#}\left(  r\right)  dQ_{r}%
\Big)^{2}\right]
\end{array}
\]
($C_{\lambda}=C_{p,\lambda}$ is the constant given by (\ref{an3a}) with $p=2$).

By assumption (\ref{ap-eq-1c}) we get%
\begin{equation}%
\begin{array}
[c]{ll}%
\left(  a\right)  &
\begin{array}
[c]{l}%
\left\vert Y_{t}^{\varepsilon}\right\vert \leq\big|e^{V_{r}^{\left(  +\right)
}}Y_{r}^{\varepsilon}\big|\\
\text{\quad\quad\quad}\leq\left[  \mathbb{E}^{\mathcal{F}_{t}}\Big(\sup
\limits_{r\in\left[  t,T\right]  }\big|e^{V_{r}^{\left(  +\right)  }}%
Y_{r}^{\varepsilon}\big|^{2}\Big)\right]  ^{1/2}\leq(C_{\lambda}\tilde
{L})^{1/2}=\rho_{0}\text{,}\quad\text{a.s., for all }t\in\left[  0,T\right]
,\medskip
\end{array}
\\
\left(  b\right)  & \mathbb{E}\Big(%
{\displaystyle\int_{0}^{T}}
e^{2V_{r}^{\left(  +\right)  }}\left\vert Z_{r}^{\varepsilon}\right\vert
^{2}dr\Big)\leq\rho_{0}^{2},\medskip\\
\left(  c\right)  & \left\vert F_{\varepsilon}\left(  t,Y_{t}^{\varepsilon
},Z_{t}^{\varepsilon}\right)  \right\vert \leq\ell_{r}\left\vert
Z_{r}^{\varepsilon}\right\vert +F_{1+\rho_{0}}^{\#}\left(  r\right)
,\quad\left\vert G_{\varepsilon}\left(  t,Y_{t}^{\varepsilon}\right)
\right\vert \leq G_{1+\rho_{0}}^{\#}\left(  r\right)  \medskip\\
\left(  d\right)  & \left\vert H_{\varepsilon}(r,Y_{r}^{\varepsilon}%
,Z_{r}^{\varepsilon})\right\vert \leq\left[  \alpha_{r}\left(  \ell
_{r}\left\vert Z_{r}^{\varepsilon}\right\vert +F_{1+\rho_{0}}^{\#}\left(
r\right)  \right)  +\left(  1-\alpha_{r}\right)  G_{1+\rho_{0}}^{\#}\left(
r\right)  \right]  \,\mathbf{1}_{[0,\frac{1}{\varepsilon}]}\left(
A_{r}\right)  .
\end{array}
\label{ap-eq-3}%
\end{equation}
Following the ideas from \cite{ma-ra/10}, \cite{ma-ra/15} and Section 5.6.2
from \cite{pa-ra/14} and using the stochastic subdifferential inequality from
Lemma 2.38 and Remark 2.39 from \cite{pa-ra/14}, for all $0\leq t\leq s\leq T$%
\[%
\begin{array}
[c]{r}%
e^{2V_{t}^{\left(  +\right)  }}\varphi_{\varepsilon}(Y_{t}^{\varepsilon})\leq
e^{2V_{s}^{\left(  +\right)  }}\varphi_{\varepsilon}(Y_{s}^{\varepsilon})+%
{\displaystyle\int_{t}^{s}}
e^{2V_{r}^{\left(  +\right)  }}\left\langle \nabla\varphi_{\varepsilon}%
(Y_{r}^{\varepsilon}),\Phi_{\varepsilon}(r,Y_{r}^{\varepsilon},Z_{r}%
^{\varepsilon})\right\rangle dQ_{r}\medskip\\
-%
{\displaystyle\int_{t}^{s}}
e^{2V_{r}^{\left(  +\right)  }}\left\langle \nabla\varphi_{\varepsilon}%
(Y_{r}^{\varepsilon}),Z_{r}^{\varepsilon}dB_{r}\right\rangle
\end{array}
\]
(and similar inequality for $\psi_{\varepsilon}$) we deduce that:%
\begin{equation}%
\begin{array}
[c]{l}%
e^{2V_{t}^{\left(  +\right)  }}\left[  \varphi_{\varepsilon}(Y_{t}%
^{\varepsilon})+\psi_{\varepsilon}(Y_{t}^{\varepsilon})\right]  \medskip\\
\quad+%
{\displaystyle\int_{t}^{s}}
e^{2V_{r}^{\left(  +\right)  }}\Big[\alpha_{r}|\nabla\varphi_{\varepsilon
}(Y_{r}^{\varepsilon})|^{2}+\left\langle \nabla\varphi_{\varepsilon}%
(Y_{r}^{\varepsilon}),\nabla\psi_{\varepsilon}(Y_{r}^{\varepsilon
})\right\rangle +\left(  1-\alpha_{r}\right)  |\nabla\psi_{\varepsilon}%
(Y_{r}^{\varepsilon})|^{2}\Big]dQ_{r}\medskip\\
\leq e^{2V_{s}^{\left(  +\right)  }}\left[  \varphi_{\varepsilon}%
(Y_{s}^{\varepsilon})+\psi_{\varepsilon}(Y_{s}^{\varepsilon})\right]  +%
{\displaystyle\int_{t}^{s}}
e^{2V_{r}^{\left(  +\right)  }}\left\langle \nabla\varphi_{\varepsilon}%
(Y_{r}^{\varepsilon})+\nabla\psi_{\varepsilon}(Y_{r}^{\varepsilon
}),H_{\varepsilon}(r,Y_{r}^{\varepsilon},Z_{r}^{\varepsilon})\right\rangle
dQ_{r}\medskip\\
\quad-%
{\displaystyle\int_{t}^{s}}
e^{2V_{r}^{\left(  +\right)  }}\left\langle \nabla\varphi_{\varepsilon}%
(Y_{r}^{\varepsilon})+\nabla\psi_{\varepsilon}(Y_{r}^{\varepsilon}%
),Z_{r}^{\varepsilon}dB_{r}\right\rangle .
\end{array}
\label{ap-eq-4}%
\end{equation}
Remark that compatibility assumptions (\ref{compAssumpt}) yield for
$\left\vert y\right\vert \leq\rho_{0}$:\medskip\newline$%
\begin{array}
[c]{c}%
\quad
\end{array}
\left\langle \nabla\psi_{\varepsilon}(y),F_{\varepsilon}(t,y,z)\right\rangle
$\medskip\newline$%
\begin{array}
[c]{c}%
\quad
\end{array}
=%
{\displaystyle\int_{\overline{B\left(  0,1\right)  }}}
\left\langle \nabla\psi_{\varepsilon}(y)-\nabla\psi_{\varepsilon
}(y-\varepsilon u),F\left(  t,y-\varepsilon u,\beta_{\varepsilon}\left(
z\right)  \right)  \right\rangle \,\mathbf{1}_{\left[  0,1\right]  }\left(
\varepsilon\left\vert F\left(  t,y-\varepsilon u\right)  ,0\right\vert
\right)  \rho\left(  u\right)  du$\medskip\newline$%
\begin{array}
[c]{c}%
\quad\quad
\end{array}
+%
{\displaystyle\int_{\overline{B\left(  0,1\right)  }}}
\left\langle \nabla\psi_{\varepsilon}(y-\varepsilon u),F\left(
t,y-\varepsilon u,\beta_{\varepsilon}\left(  z\right)  \right)  \right\rangle
\,\mathbf{1}_{\left[  0,1\right]  }\left(  \varepsilon\left\vert F\left(
t,y-\varepsilon u\right)  ,0\right\vert \right)  \rho\left(  u\right)
du$\medskip\newline$%
\begin{array}
[c]{c}%
\quad
\end{array}
\leq%
{\displaystyle\int_{\overline{B\left(  0,1\right)  }}}
\frac{1}{\varepsilon}\left\vert \varepsilon u\right\vert \left\vert F\left(
t,y-\varepsilon u,\beta_{\varepsilon}\left(  z\right)  \right)  \right\vert
\,\mathbf{1}_{\left[  0,1\right]  }\left(  \varepsilon\left\vert F\left(
t,y-\varepsilon u\right)  ,0\right\vert \right)  \rho\left(  u\right)
du$\medskip\newline$%
\begin{array}
[c]{c}%
\quad\quad
\end{array}
+%
{\displaystyle\int_{\overline{B\left(  0,1\right)  }}}
\left\vert \nabla\varphi_{\varepsilon}(y-\varepsilon u)\right\vert \left\vert
F\left(  t,y-\varepsilon u,\beta_{\varepsilon}\left(  z\right)  \right)
\right\vert \,\mathbf{1}_{\left[  0,1\right]  }\left(  \varepsilon\left\vert
F\left(  t,y-\varepsilon u\right)  ,0\right\vert \right)  \rho\left(
u\right)  du$\medskip\newline$%
\begin{array}
[c]{c}%
\quad
\end{array}
\leq\left\vert F\right\vert _{\varepsilon}\left(  t,y,z\right)  +%
{\displaystyle\int_{\overline{B\left(  0,1\right)  }}}
\left[  \left\vert \nabla\varphi_{\varepsilon}(y-\varepsilon u)-\nabla
\varphi_{\varepsilon}(y)\right\vert +\left\vert \nabla\varphi_{\varepsilon
}(y)\right\vert \right]  $\medskip\newline$%
\begin{array}
[c]{c}%
\quad\quad\quad\quad\quad\quad\quad\quad\quad
\end{array}
\times\left\vert F\left(  t,y-\varepsilon u,\beta_{\varepsilon}\left(
z\right)  \right)  \right\vert \,\mathbf{1}_{\left[  0,1\right]  }\left(
\varepsilon\left\vert F\left(  t,y-\varepsilon u\right)  ,0\right\vert
\right)  \rho\left(  u\right)  du$\medskip\newline$%
\begin{array}
[c]{c}%
\quad
\end{array}
\leq\left\vert F\right\vert _{\varepsilon}\left(  t,y,z\right)  +\left(
1+\left\vert \nabla\varphi_{\varepsilon}(y)\right\vert \right)  \left\vert
F\right\vert _{\varepsilon}\left(  t,y,z\right)  $\medskip\newline$%
\begin{array}
[c]{c}%
\quad
\end{array}
=\left(  2+\left\vert \nabla\varphi_{\varepsilon}(y)\right\vert \right)
\left\vert F\right\vert _{\varepsilon}\left(  t,y,z\right)  $\medskip\newline$%
\begin{array}
[c]{c}%
\quad
\end{array}
\leq\ell_{t}\left\vert z\right\vert +\left\vert F\right\vert _{\varepsilon
}\left(  t,y,0\right)  +\left(  1+\left\vert \nabla\varphi_{\varepsilon
}(y)\right\vert \right)  \left[  \ell_{t}\left\vert z\right\vert +\left\vert
F\right\vert _{\varepsilon}\left(  t,y,0\right)  \right]  $\medskip\newline$%
\begin{array}
[c]{c}%
\quad
\end{array}
=\left(  2+\left\vert \nabla\varphi_{\varepsilon}(y)\right\vert \right)
\left(  \ell_{t}\left\vert z\right\vert +\left\vert F\right\vert
_{\varepsilon}\left(  t,y,0\right)  \right)  $\medskip\newline$%
\begin{array}
[c]{c}%
\quad
\end{array}
\leq2L\left\vert z\right\vert +2F_{1+\rho_{0}}^{\#}\left(  t\right)
+\left\vert \nabla\varphi_{\varepsilon}(y)\right\vert \left(  L\left\vert
z\right\vert +F_{1+\rho_{0}}^{\#}\left(  t\right)  \right)  $\medskip\newline
and similarly%
\begin{align*}
\left\langle \nabla\varphi_{\varepsilon}(y),G_{\varepsilon}(t,y)\right\rangle
&  \leq\left(  2+\left\vert \nabla\psi_{\varepsilon}(y)\right\vert \right)
\left\vert G\right\vert _{\varepsilon}\left(  t,y\right) \\
&  \leq2G_{1+\rho_{0}}^{\#}\left(  t\right)  +\left\vert \nabla\psi
_{\varepsilon}(y)\right\vert G_{1+\rho_{0}}^{\#}\left(  t\right)  .
\end{align*}
Hence using definition of the function $H_{\varepsilon}(t,y,z)$ we have for
$\left\vert y\right\vert \leq\rho_{0}:$%
\[%
\begin{array}
[c]{l}%
\left\langle \nabla\varphi_{\varepsilon}(y)+\nabla\psi_{\varepsilon
}(y),H_{\varepsilon}(s,y,z)\right\rangle \medskip\\
=\left\langle \nabla\varphi_{\varepsilon}(y)+\nabla\psi_{\varepsilon
}(y),\alpha_{s}F_{\varepsilon}\left(  s,y,z\right)  +\left(  1-\alpha
_{s}\right)  G_{\varepsilon}\left(  s,y\right)  \right\rangle \mathbf{1}%
_{[0,\frac{1}{\varepsilon}]}\left(  A_{r}\right)  \medskip\\
\leq\alpha_{t}\Big[\left(  2+2\left\vert \nabla\varphi_{\varepsilon
}(y)\right\vert \right)  \big(\left\vert F\right\vert _{\varepsilon}\left(
t,y,z\right)  \big)\Big]+\left(  1-\alpha_{t}\right)  \Big[\left(
2+2\left\vert \nabla\psi_{\varepsilon}(y)\right\vert \right)  \left\vert
G\right\vert _{\varepsilon}\left(  t,y\right)  \Big]\medskip\\
\leq\alpha_{t}\Big[\dfrac{1}{2}\left\vert \nabla\varphi_{\varepsilon
}(y)\right\vert ^{2}+1+3\left[  \left\vert F\right\vert _{\varepsilon}\left(
t,y,z\right)  \right]  ^{2}\Big]+\left(  1-\alpha_{t}\right)  \Big[\dfrac
{1}{2}\left\vert \nabla\psi_{\varepsilon}(y)\right\vert +1+3\left[  \left\vert
G\right\vert _{\varepsilon}\left(  t,y\right)  \right]  ^{2}\Big].
\end{array}
\]
Consequently, using inequality%
\begin{equation}
\mathbb{E}\left[  e^{2V_{T}^{\left(  +\right)  }}\left(  \varphi_{\varepsilon
}(Y_{T}^{\varepsilon})+\psi_{\varepsilon}(Y_{T}^{\varepsilon})\right)
\right]  \leq\mathbb{E}\left[  e^{2V_{T}^{\left(  +\right)  }}\left(
\varphi(\eta)+\psi(\eta)\right)  \right]  \label{ap-eq-8}%
\end{equation}
and the fact that
\[
M_{t}^{\varepsilon}=\int_{0}^{t}e^{2V_{r}^{\left(  +\right)  }}\left\langle
\nabla\varphi_{\varepsilon}(Y_{r}^{\varepsilon})+\nabla\psi_{\varepsilon
}(Y_{r}^{\varepsilon}),Z_{r}^{\varepsilon}dB_{r}\right\rangle
\]
is a martingale we obtain that for all $0\leq t\leq s\leq T:$%
\begin{equation}%
\begin{array}
[c]{l}%
\mathbb{E}e^{2V_{t}^{\left(  +\right)  }}\varphi_{\varepsilon}(Y_{t}%
^{\varepsilon})+\mathbb{E}e^{2V_{t}^{\left(  +\right)  }}\psi_{\varepsilon
}(Y_{t}^{\varepsilon})+\dfrac{1}{2}\,\mathbb{E}%
{\displaystyle\int_{t}^{T}}
e^{2V_{r}^{\left(  +\right)  }}\Big[|\nabla\varphi_{\varepsilon}%
(Y_{r}^{\varepsilon})|^{2}dr+|\nabla\psi_{\varepsilon}(Y_{r}^{\varepsilon
})|^{2}dA_{r}\Big]\medskip\\
\leq\mathbb{E}\left[  e^{2V_{T}^{\left(  +\right)  }}\left(  \varphi
(\eta)+\psi(\eta)\right)  \right]  \medskip\\
+\mathbb{E}%
{\displaystyle\int_{t}^{T}}
e^{2V_{r}^{\left(  +\right)  }}\left(  1+3\left[  \left\vert F\right\vert
_{\varepsilon}\left(  r,Y_{r}^{\varepsilon},Z_{r}^{\varepsilon}\right)
\right]  ^{2}\right)  dr+\mathbb{E}%
{\displaystyle\int_{t}^{T}}
e^{2V_{r}^{\left(  +\right)  }}\left(  1+3\left[  \left\vert G\right\vert
_{\varepsilon}\left(  r,Y_{r}^{\varepsilon}\right)  \right]  ^{2}\right)
dA_{r}\medskip\\
\leq\mathbb{E}\left[  e^{2V_{T}^{\left(  +\right)  }}\left(  \varphi
(\eta)+\psi(\eta)\right)  \right]  \medskip\\
\quad+\mathbb{E}%
{\displaystyle\int_{t}^{T}}
e^{2V_{r}^{\left(  +\right)  }}\left[  1+6L^{2}\left\vert Z_{r}^{\varepsilon
}\right\vert ^{2}+6\left(  F_{1+\rho_{0}}^{\#}\left(  r\right)  \right)
^{2}\right]  dr+\mathbb{E}%
{\displaystyle\int_{t}^{T}}
e^{2V_{r}^{\left(  +\right)  }}\left[  1+3\left(  G_{1+\rho_{0}}^{\#}\left(
r\right)  \right)  ^{2}\right]  dA_{r}%
\end{array}
\label{ap-eq-5a}%
\end{equation}
which clearly yields%
\begin{equation}%
\begin{array}
[c]{ll}%
\left(  a\right)  & \sup_{t\in\left[  0,T\right]  }\left[  \mathbb{E}%
e^{2V_{t}^{\left(  +\right)  }}\varphi_{\varepsilon}(Y_{t}^{\varepsilon
})+\mathbb{E}e^{2V_{t}^{\left(  +\right)  }}\psi_{\varepsilon}(Y_{t}%
^{\varepsilon})\right]  \leq C=C_{\rho_{0},L,T,\lambda}\medskip\\
\left(  b\right)  & \mathbb{E}%
{\displaystyle\int_{0}^{T}}
e^{2V_{r}^{\left(  +\right)  }}\Big[|\nabla\varphi_{\varepsilon}%
(Y_{r}^{\varepsilon})|^{2}dr+|\nabla\psi_{\varepsilon}(Y_{r}^{\varepsilon
})|^{2}dA_{r}\Big]\leq C
\end{array}
\label{ap-eq-7}%
\end{equation}
(with $C$ a constant independent of $\varepsilon$).

Let $\varepsilon,\delta\in(0,1].$ We have%
\[
Y_{t}^{\varepsilon}-Y_{t}^{\delta}=\int_{t}^{T}dK_{r}^{\varepsilon,\delta
}-\int_{t}^{T}(Z_{r}^{\varepsilon}-Z_{r}^{\delta})dB_{r}\,,\quad\text{a.s.,}%
\]
with\medskip\newline$%
\begin{array}
[c]{c}%
\quad
\end{array}
dK_{r}^{\varepsilon,\delta}=\Big[H_{\varepsilon}(r,Y_{r}^{\varepsilon}%
,Z_{r}^{\varepsilon})-H_{\delta}(r,Y_{r}^{\delta},Z_{r}^{\delta})-\left[
\nabla_{y}\Psi^{\varepsilon}(r,Y_{r}^{\varepsilon})-\nabla_{y}\Psi^{\delta
}(r,Y_{r}^{\delta})\right]  \Big]dQ_{r}$\medskip\newline$%
\begin{array}
[c]{c}%
\quad
\end{array}
=\alpha_{r}\left[  F_{\varepsilon}(r,Y_{r}^{\varepsilon},Z_{r}^{\varepsilon
})-F_{\delta}(r,Y_{r}^{\delta},Z_{r}^{\delta})\right]  \mathbf{1}_{\left[
0,\frac{1}{\varepsilon}\right]  }\left(  A_{r}\right)  dr$\medskip\newline$%
\begin{array}
[c]{c}%
\quad\quad
\end{array}
+\left(  1-\alpha_{r}\right)  \left[  G_{\varepsilon}(r,Y_{r}^{\varepsilon
})-G_{\delta}(r,Y_{r}^{\delta})\right]  \mathbf{1}_{\left[  0,\frac
{1}{\varepsilon}\right]  }\left(  A_{r}\right)  dA_{r}$\medskip\newline$%
\begin{array}
[c]{c}%
\quad\quad
\end{array}
+\alpha_{r}F_{\delta}(r,Y_{r}^{\delta},Z_{r}^{\delta})\left(  \mathbf{1}%
_{\left[  0,\frac{1}{\varepsilon}\right]  }\left(  A_{r}\right)
-\mathbf{1}_{\left[  0,\frac{1}{\delta}\right]  }\left(  A_{r}\right)
\right)  dr$\medskip\newline$%
\begin{array}
[c]{c}%
\quad\quad
\end{array}
+\left(  1-\alpha_{r}\right)  G_{\delta}(r,Y_{r}^{\delta})\left(
\mathbf{1}_{\left[  0,\frac{1}{\varepsilon}\right]  }\left(  A_{r}\right)
-\mathbf{1}_{\left[  0,\frac{1}{\delta}\right]  }\left(  A_{r}\right)
\right)  dA_{r}$\medskip\newline$%
\begin{array}
[c]{c}%
\quad\quad
\end{array}
-\alpha_{r}\left[  \nabla\varphi_{\varepsilon}(Y_{r}^{\varepsilon}%
)-\nabla\varphi_{\delta}(Y_{r}^{\delta})\right]  dr-\left(  1-\alpha
_{r}\right)  \left[  \nabla\psi_{\varepsilon}(Y_{r}^{\varepsilon})-\nabla
\psi_{\delta}(Y_{r}^{\delta})\right]  dA_{r}\,.$\medskip\newline By
(\ref{fi-Cauchy}) we have%
\[
\big\langle Y_{r}^{\varepsilon}-Y_{r}^{\delta},dK_{r}^{\varepsilon,\delta
}\big\rangle\leq dR_{r}^{\varepsilon,\delta}+|Y_{r}^{\varepsilon}%
-Y_{r}^{\delta}|dN_{r}^{\varepsilon,\delta}+|Y_{r}^{\varepsilon}-Y_{r}%
^{\delta}|^{2}dV_{r}^{\left(  +\right)  }+\dfrac{\lambda}{2}|Z_{r}%
^{\varepsilon}-Z_{r}^{\delta}|^{2}dr
\]
where%
\[%
\begin{array}
[c]{r}%
\displaystyle dR_{r}^{\varepsilon,\delta}=\left\vert \varepsilon
-\delta\right\vert \left[  \mu_{r}^{+}\left\vert \varepsilon-\delta\right\vert
+2F_{1+\rho_{0}}^{\#}\left(  r\right)  +2\ell_{r}\left\vert Z_{r}%
^{\varepsilon}\right\vert \right]  dr+\left\vert \varepsilon-\delta\right\vert
\left[  \nu_{r}^{+}\left\vert \varepsilon-\delta\right\vert +2G_{1+\rho_{0}%
}^{\#}\left(  r\right)  \right]  dA_{r}\medskip\\
\displaystyle+\dfrac{\varepsilon+\delta}{2}\big(|\nabla\varphi_{\varepsilon
}(Y_{r}^{\delta})|^{2}+|\nabla\varphi_{\delta}(Y_{r}^{\varepsilon}%
)|^{2}\big)dr+\dfrac{\varepsilon+\delta}{2}\big(|\nabla\psi_{\varepsilon
}(Y_{r}^{\delta})|^{2}+|\nabla\psi_{\delta}(Y_{r}^{\varepsilon})|^{2}%
\big)dA_{r}%
\end{array}
\]
and%
\[%
\begin{array}
[c]{l}%
\displaystyle dN_{r}^{\varepsilon,\delta}=\Big[2\left\vert \mu_{r}\right\vert
\left\vert \varepsilon-\delta\right\vert +\ell_{r}\left\vert Z_{r}^{\delta
}\right\vert \,\mathbf{1}_{[\frac{1}{\varepsilon}\wedge\frac{1}{\delta}%
,\infty)}\left(  \left\vert Z_{r}^{\delta}\right\vert +A_{r}\right)
\,\mathbf{1}_{\varepsilon\neq\delta}\medskip\\
\displaystyle\quad+\left[  F_{1+\rho_{0}}^{\#}\left(  r\right)  +\ell
_{r}\left\vert Z_{r}^{\delta}\right\vert \right]  \,\mathbf{1}_{[\frac
{1}{\varepsilon}\wedge\frac{1}{\delta},\infty)}\left(  F_{1+\rho_{0}}%
^{\#}\left(  r\right)  +A_{r}\right)  \Big]dr\medskip\\
\displaystyle\quad+\left[  2\left\vert \nu_{r}\right\vert \left\vert
\varepsilon-\delta\right\vert +\left[  G_{1+\rho_{0}}^{\#}\left(  r\right)
+\ell_{r}\left\vert Z_{r}^{\delta}\right\vert \right]  \,\mathbf{1}_{[\frac
{1}{\varepsilon}\wedge\frac{1}{\delta},\infty)}\left(  G_{1+\rho_{0}}%
^{\#}\left(  r\right)  +A_{r}\right)  \right]  dA_{r}\,.
\end{array}
\]
By Proposition \ref{Appendix_result 1} we get%
\[%
\begin{array}
[c]{l}%
\mathbb{E}\sup\nolimits_{r\in\left[  0,T\right]  }e^{2V_{r}^{\left(  +\right)
}}|Y_{r}^{\varepsilon}-Y_{r}^{\delta}|^{2}+\mathbb{E}{%
{\displaystyle\int_{0}^{T}}
}e^{2V_{r}^{\left(  +\right)  }}|Z_{r}^{\varepsilon}-Z_{r}^{\delta}%
|^{2}dr\medskip\\
\leq C_{\lambda}\,\left[  \mathbb{E}%
{\displaystyle\int_{0}^{T}}
e^{2V_{r}^{\left(  +\right)  }}dR_{r}^{\varepsilon,\delta}+\mathbb{E~}\left(
{\displaystyle\int_{0}^{T}}
e^{V_{r}^{\left(  +\right)  }}dN_{r}^{\varepsilon,\delta}\right)  ^{2}\right]
.
\end{array}
\]
Boundedness assumptions (\ref{ap-eq-1b}), (\ref{ap-eq-1bb}), (\ref{ap-eq-1c}),
(\ref{ap-eq-6}) and boundedness result (\ref{ap-eq-7}) yield that
\[
\lim_{\varepsilon,\delta\rightarrow0}\mathbb{E}%
{\displaystyle\int_{0}^{T}}
e^{2V_{r}^{\left(  +\right)  }}dR_{r}^{\varepsilon,\delta}+\mathbb{E}\left(
{\displaystyle\int_{0}^{T}}
e^{V_{r}^{\left(  +\right)  }}dN_{r}^{\varepsilon,\delta}\right)  ^{2}=0.
\]
Consequently there exists $(Y,Z)\in S_{m}^{0}\times\Lambda_{m\times k}^{0}$
such that%
\[
\mathbb{E}\sup\limits_{r\in\left[  0,T\right]  }e^{2V_{r}^{\left(  +\right)
}}|Y_{r}^{\varepsilon}-Y_{r}|^{2}+\mathbb{E}\int_{0}^{T}e^{2V_{r}^{\left(
+\right)  }}|Z_{r}^{\varepsilon}-Z_{r}|^{2}%
dr\xrightarrow[]{\;\;\;\;}0,\;\text{as }\varepsilon\rightarrow0.
\]
From (\ref{ap-eq-7}) there exist two p.m.s.p. $U^{\left(  1\right)  }$ and
$U^{\left(  2\right)  }$, such that along a sequence $\varepsilon
_{n}\rightarrow0$, we have%
\begin{align*}
e^{V^{\left(  +\right)  }}\nabla\varphi_{\varepsilon}(Y^{\varepsilon_{n}})  &
\rightharpoonup e^{V^{\left(  +\right)  }}U^{\left(  1\right)  }%
,\quad\text{weakly in }L^{2}\left(  \Omega\times\left[  0,T\right]
,d\mathbb{P}\otimes dt;\mathbb{R}^{m}\right)  ,\medskip\\
e^{V^{\left(  +\right)  }}\nabla\psi_{\varepsilon}(Y^{\varepsilon_{n}})  &
\rightharpoonup e^{V^{\left(  +\right)  }}U^{\left(  2\right)  }%
,\quad\text{weakly in }L^{2}\left(  \Omega\times\left[  0,T\right]
,d\mathbb{P}\otimes dA_{t};\mathbb{R}^{m}\right)  .
\end{align*}
Passing to limit in the approximating equation for $\varepsilon=\varepsilon
_{n}\rightarrow0$ we infer%
\[
Y_{t}+\int_{t}^{T}U_{r}dQ_{r}=\eta+\int_{t}^{T}H(r,Y_{r},Z_{r})dQ_{r}-\int
_{t}^{T}Z_{r}dB_{r}\,,\;\text{a.s.}%
\]
where%
\[
U_{r}=\left[  \alpha_{r}U_{r}^{1}+\left(  1-\alpha_{r}\right)  U_{r}%
^{2}\right]  ,\;\text{for }r\in\left[  0,T\right]  .
\]
Since $\nabla\varphi_{\varepsilon}(y)\in\partial\varphi\left(  y-\varepsilon
\nabla\varphi_{\varepsilon}(y)\right)  $ then for all $E\in\mathcal{F}$,
$0\leq t\leq s\leq T$ and $X\in S_{m}^{2}\left[  0,T\right]  $%
\[%
\begin{array}
[c]{r}%
\mathbb{E}%
{\displaystyle\int_{t}^{s}}
\big\langle e^{2V_{r}^{\left(  +\right)  }}\nabla\varphi_{\varepsilon_{n}%
}(Y_{r}^{\varepsilon_{n}}),X_{r}-Y_{r}^{\varepsilon_{n}}%
\big\rangle\,\mathbf{1}_{E}\,dr+\mathbb{E}%
{\displaystyle\int_{t}^{s}}
e^{2V_{r}^{\left(  +\right)  }}\varphi(Y_{r}^{\varepsilon_{n}}-\varepsilon
\nabla\varphi_{\varepsilon}(Y_{r}^{\varepsilon_{n}}))\,\mathbf{1}%
_{E}\,dr\medskip\\
\leq\mathbb{E}%
{\displaystyle\int_{t}^{s}}
e^{2V_{r}^{\left(  +\right)  }}\varphi(X_{r})\,\mathbf{1}_{E}\,dr.
\end{array}
\]
Passing to $\liminf_{n\rightarrow+\infty}$ in the above inequality we obtain
$U_{s}^{\left(  1\right)  }\in\partial\varphi(Y_{s}),d\mathbb{P}\otimes
ds$-a.e. and, with similar arguments, $U_{s}^{\left(  2\right)  }\in
\partial\psi(Y_{s}),d\mathbb{P}\otimes dA_{s}$-a.e. Summarizing the above
conclusions we conclude that $(Y,Z,U)\in S_{m}^{0}\left[  0,T\right]
\times\Lambda_{m\times k}^{0}\left[  0,T\right]  \times\Lambda_{m\times k}%
^{0}\left[  0,T\right]  $ is strong solution of
\begin{equation}
\left\{
\begin{array}
[c]{l}%
Y_{t}+%
{\displaystyle\int_{t}^{T}}
\left(  U_{s}^{\left(  1\right)  }ds+U_{s}^{\left(  2\right)  }dA_{s}\right)
=\eta+%
{\displaystyle\int_{t}^{T}}
\left[  F\left(  s,Y_{s},Z_{s}\right)  ds+G\left(  s,Y_{s}\right)
dA_{s}\right]  \medskip\\
\quad\quad\quad\quad\quad\quad\quad\quad\quad\quad\quad\quad\quad\quad
\quad\quad\quad\quad\quad\quad\quad\;-%
{\displaystyle\int_{t}^{T}}
Z_{s}dB_{s}\,,\;t\in\left[  0,T\right]  ,\medskip\\
U_{s}^{\left(  1\right)  }\in\partial\varphi(Y_{s})~,d\mathbb{P}\otimes
ds-\text{a.e.}\quad\text{and }\quad U_{s}^{\left(  2\right)  }\in\partial
\psi(Y_{s}),\quad d\mathbb{P}\otimes dA_{s}-\text{a.e.}\quad\text{on }\left[
0,T\right]  .
\end{array}
\right.  \label{BSVI-0T}%
\end{equation}
Moreover, from (\ref{ap-eq-3}),%
\begin{equation}%
\begin{array}
[c]{ll}%
\left(  a\right)  & \left\vert Y_{t}\right\vert \leq e^{V_{t}^{\left(
+\right)  }}\left\vert Y_{t}\right\vert \leq\left[  \mathbb{E}^{\mathcal{F}%
_{t}}\Big(\sup\limits_{r\in\left[  t,T\right]  }\left\vert e^{V_{r}^{\left(
+\right)  }}Y_{r}\right\vert ^{2}\Big)\right]  ^{1/2}\leq(C_{\lambda}\tilde
{L})^{1/2}=\rho_{0}\quad\text{and}\medskip\\
\left(  b\right)  & \mathbb{E}\Big(%
{\displaystyle\int_{0}^{T}}
e^{2V_{r}^{\left(  +\right)  }}\left\vert Z_{r}\right\vert ^{2}dr\Big)\leq
\rho_{0}^{2}\,.
\end{array}
\label{b-1}%
\end{equation}
Since%
\begin{align*}
0  &  \leq\left[  \left\vert F\right\vert _{\varepsilon}\left(  r,Y_{r}%
^{\varepsilon},Z_{r}^{\varepsilon}\right)  \right]  ^{2}\leq6L^{2}\left\vert
Z_{r}^{\varepsilon}\right\vert ^{2}+6\left(  F_{1+\rho_{0}}^{\#}\left(
r\right)  \right)  ^{2}\\
0  &  \leq\left[  \left\vert G\right\vert _{\varepsilon}\left(  r,Y_{r}%
^{\varepsilon}\right)  \right]  ^{2}\leq\left(  G_{1+\rho_{0}}^{\#}\left(
r\right)  \right)  ^{2},
\end{align*}
passing to $\liminf_{\varepsilon\rightarrow0}$ in (\ref{ap-eq-5a}), we have by
Fatou's Lemma and by the Lebesgue dominated convergence theorem that%
\begin{equation}%
\begin{array}
[c]{l}%
\dfrac{1}{2}\,\mathbb{E}%
{\displaystyle\int_{0}^{T}}
e^{2V_{r}^{\left(  +\right)  }}\Big[\big|U_{r}^{\left(  1\right)  }%
\big|^{2}dr+\big|U_{r}^{\left(  2\right)  }\big|^{2}dA_{r}\Big]\medskip\\
\leq\mathbb{E}\left[  e^{2V_{T}^{\left(  +\right)  }}\left(  \varphi
(\eta)+\psi(\eta)\right)  \right]  +\mathbb{E}%
{\displaystyle\int_{0}^{T}}
e^{2V_{r}^{\left(  +\right)  }}\left(  1+6\left\vert Z_{r}\right\vert
^{2}+6\left\vert F\left(  r,Y_{r},0\right)  \right\vert ^{2}\right)
dr\medskip\\
\quad+\mathbb{E}%
{\displaystyle\int_{0}^{T}}
e^{2V_{r}^{\left(  +\right)  }}\left(  1+3\left\vert G\left(  r,Y_{r}\right)
\right\vert ^{2}\right)  dA_{r}%
\end{array}
\label{b-1a}%
\end{equation}

\hfill
\end{proof}

\begin{proposition}
[$L^{p}$-- variational solution]\label{p1-wvs}We suppose that assumptions
$\left(  \mathrm{A}_{1}-\mathrm{A}_{6}\right)  $ are satisfied. Let
$0<\lambda<1<p,$ $n_{p}=\left(  p-1\right)  \wedge1$ and $V^{\left(  +\right)
}$ be given by (\ref{defV_2}).

Moreover, we assume:

\begin{itemize}
\item[$\left(  i\right)  $] there exists $\hat{L}>0$ such that%
\begin{equation}
\big|e^{V_{T}^{\left(  +\right)  }}\eta\big|^{2}+\Big(%
{\displaystyle\int_{0}^{T}}
e^{V_{r}^{\left(  +\right)  }}\left(  \left\vert F\left(  r,0,0\right)
\right\vert dr+\left\vert G\left(  r,0\right)  \right\vert dA_{r}\right)
\Big)^{2}\leq\hat{L}; \label{assump-H(t,0,0)}%
\end{equation}

\item[$\left(  ii\right)  $] there exists $a\in(1+n_{p}\lambda,p\wedge2)$ such
that%
\begin{equation}%
\begin{array}
[c]{rl}%
\left(  a\right)  & \mathbb{E}\big(%
{\displaystyle\int_{0}^{T}}
\ell_{s}^{2}ds\big)^{\frac{a}{2-a}}<\infty,\medskip\\
\left(  b\right)  & \mathbb{E}\left[
{\displaystyle\int_{0}^{T}}
e^{V_{s}^{\left(  +\right)  }}\left(  F_{1+\hat{\rho}}^{\#}\left(  s\right)
ds+G_{1+\hat{\rho}}^{\#}\left(  s\right)  dA_{s}\right)  \right]  ^{a}<\infty,
\end{array}
\label{ip-1a}%
\end{equation}
where\footnote{The constant $\hat{L}$ is given by (\ref{assump-H(t,0,0)}) and
the constant $C_{\lambda}=C_{p,\lambda}$ is given by (\ref{an3a}) with $p=2.$}
$\hat{\rho}=(C_{\lambda}\hat{L})^{1/2};$

\item[$\left(  iii\right)  $] there exists a positive p.m.s.p. $\left(
\Theta_{t}\right)  _{t\in\left[  0,T\right]  }$ and, for each $\rho\geq0$,
there exist an non-decreasing function $K_{\rho}:\mathbb{R}_{+}\rightarrow
\mathbb{R}_{+}$ such that%
\begin{equation}
F_{\rho}^{\#}\left(  t\right)  +G_{\rho}^{\#}\left(  t\right)  \leq K_{\rho
}\left(  \Theta_{t}\right)  \text{,}\quad\text{a.e. }t\in\left[  0,T\right]
,\;\mathbb{P}-a.s.. \label{ip-1b}%
\end{equation}

\end{itemize}

Then the multivalued BSDE
\[
\left\{
\begin{array}
[c]{r}%
\displaystyle Y_{t}+{\int_{t}^{T}}dK_{r}=\eta+{\int_{t}^{T}}H\left(
r,Y_{r},Z_{r}\right)  dQ_{r}-{\int_{t}^{T}}Z_{r}dB_{r},\;\text{a.s., for all
}t\in\left[  0,T\right]  ,\medskip\\
\multicolumn{1}{l}{\displaystyle dK_{r}=U_{r}dQ_{r}\in\partial_{y}\Psi\left(
r,Y_{r}\right)  dQ_{r}}%
\end{array}
\,\right.
\]
has a unique $L^{p}$--variational solution, in the sense of Definition
\ref{definition_weak solution}.

Moreover, for all $t\in\left[  0,T\right]  ,$ $\mathbb{P}$--a.s.,%
\begin{equation}%
\begin{array}
[c]{l}%
\displaystyle\mathbb{E}^{\mathcal{F}_{t}}\Big(\sup\limits_{s\in\left[
t,T\right]  }\big|e^{V_{s}^{\left(  +\right)  }}Y_{s}\big|^{p}\Big)+\mathbb{E}%
^{\mathcal{F}_{t}}\Big(%
{\displaystyle\int_{t}^{T}}
e^{2V_{s}^{\left(  +\right)  }}\left(  \varphi\left(  Y_{s}\right)
ds+\psi\left(  Y_{s}\right)  dA_{s}\right)  \Big)^{p/2}\medskip\\
\displaystyle\quad\quad\quad\quad\quad\quad\quad\quad\quad\quad\quad\quad
\quad\quad\quad\quad\quad\quad+\mathbb{E}^{\mathcal{F}_{t}}\Big(%
{\displaystyle\int_{0}^{T}}
e^{2V_{s}^{\left(  +\right)  }}\left\vert Z_{s}\right\vert ^{2}ds\Big)^{p/2}%
\medskip\\
\displaystyle\leq C_{p,\lambda}~\mathbb{E}^{\mathcal{F}_{t}}\left[
e^{pV_{T}^{\left(  +\right)  }}\left\vert \eta\right\vert ^{p}+\Big(%
{\displaystyle\int_{t}^{T}}
e^{V_{s}^{\left(  +\right)  }}\left(  \left\vert F\left(  r,0,0\right)
\right\vert dr+\left\vert G\left(  t,0\right)  \right\vert dA_{r}\right)
\Big)^{p}\right]  .
\end{array}
\label{es-p}%
\end{equation}

\end{proposition}

\begin{proof}
Let $t\in\left[  0,T\right]  ,$ $n\in\mathbb{N}^{\ast}$ and%
\[
\displaystyle\beta_{t}=t+A_{t}+\left\vert \mu_{t}\right\vert +\left\vert
\nu_{t}\right\vert +\ell_{t}+V_{t}^{\left(  +\right)  }+F_{1+\hat{\rho}}%
^{\#}\left(  t\right)  +G_{1+\hat{\rho}}^{\#}\left(  t\right)  +\Theta_{t}.
\]
Consider the BSDE%
\begin{equation}
\left\{
\begin{array}
[c]{l}%
Y_{t}^{\left(  n\right)  }+%
{\displaystyle\int_{t}^{T}}
U_{s}^{\left(  n\right)  }dQ_{s}=\eta^{\left(  n\right)  }+%
{\displaystyle\int_{t}^{T}}
H^{\left(  n\right)  }\left(  s,Y_{s}^{\left(  n\right)  },Z_{s}^{\left(
n\right)  }\right)  dQ_{s}-%
{\displaystyle\int_{t}^{T}}
Z_{s}^{\left(  n\right)  }dB_{s}\,,\;t\in\left[  0,T\right]  ,\medskip\\
U_{s}^{\left(  n\right)  }=\alpha_{r}U_{r}^{\left(  1,n\right)  }+\left(
1-\alpha_{r}\right)  U_{r}^{\left(  2,n\right)  }\medskip\\
U_{s}^{\left(  1,n\right)  }\in\partial\varphi(Y_{s}^{\left(  n\right)
})~,d\mathbb{P}\otimes ds-a.e.\quad\text{and }\quad U_{s}^{\left(  2,n\right)
}\in\partial\psi(Y_{s}^{\left(  n\right)  }),\quad d\mathbb{P}\otimes
dA_{s}-a.e.\quad\text{on }\left[  0,T\right]  ,
\end{array}
\right.  \label{st2-1}%
\end{equation}
where%
\[%
\begin{array}
[c]{l}%
\displaystyle\eta^{\left(  n\right)  }:=\eta\mathbf{1}_{\left[  0,n\right]
}\left(  \left\vert \eta\right\vert +V_{T}^{\left(  +\right)  }\right)
,\medskip\\
\displaystyle F^{\left(  n\right)  }\left(  t,y,z\right)  :=F\left(
t,y,z\right)  \mathbf{1}_{\left[  0,n\right]  }\left(  \beta_{t}\right)
,\quad\quad G^{\left(  n\right)  }\left(  t,y,z\right)  :=G\left(  t,y\right)
\mathbf{1}_{\left[  0,n\right]  }\left(  \beta_{t}\right)  ,\medskip\\
\displaystyle H^{\left(  n\right)  }\left(  s,y,z\right)  :=\alpha
_{s}F^{\left(  n\right)  }\left(  s,y,z\right)  +\left(  1-\alpha_{s}\right)
G^{\left(  n\right)  }\left(  s,y\right)  .
\end{array}
\]
We have%
\[
\left\langle y-\hat{y},H^{\left(  n\right)  }(t,y,z)-H^{\left(  n\right)
}(t,\hat{y},z)\right\rangle \leq\left(  \alpha_{t}\mu_{t}+\left(  1-\alpha
_{t}\right)  \nu_{t}\right)  \mathbf{1}_{\left[  0,n\right]  }\left(
\beta_{t}\right)  \left\vert y-\hat{y}\right\vert ^{2}%
\]
and%
\[
\left\vert H^{\left(  n\right)  }(t,y,z)-H^{\left(  n\right)  }(t,y,\hat
{z})\right\vert \leq\mathbf{1}_{\left[  0,n\right]  }\left(  \beta_{t}\right)
\alpha_{t}\ell_{t}\left\vert z-\hat{z}\right\vert .
\]
Remark that%
\[
\mu_{t}^{\left(  n\right)  }=\mu_{t}\,\mathbf{1}_{\left[  0,n\right]  }\left(
\beta_{t}\right)  ,\quad\nu_{t}^{\left(  n\right)  }=\nu_{t}\,\mathbf{1}%
_{\left[  0,n\right]  }\left(  \beta_{t}\right)  ,\quad\ell_{t}^{\left(
n\right)  }=\ell_{t}\,\mathbf{1}_{\left[  0,n\right]  }\left(  \beta
_{t}\right)
\]
and%
\[
F_{1}^{\left(  n\right)  \#}\left(  t\right)  =\sup_{\left\vert u\right\vert
\leq1}\left\vert F^{\left(  n\right)  }\left(  t,u,0\right)  \right\vert \leq
n\mathbf{1}_{\left[  0,n\right]  }\left(  \beta_{t}\right)  ,\quad\quad
G_{1}^{\left(  n\right)  \#}\left(  t\right)  =\sup_{\left\vert u\right\vert
\leq1}\left\vert G^{\left(  n\right)  }\left(  t,u\right)  \right\vert \leq
n\mathbf{1}_{\left[  0,n\right]  }\left(  \beta_{t}\right)  .
\]
Let $\theta_{n}=\inf\left\{  r\geq0:r+A_{r}+V_{r}^{\left(  +\right)
}>n\right\}  .$ Then $\mathbf{1}_{\left[  0,n\right]  }\left(  \beta
_{r}\right)  \leq\mathbf{1}_{\left[  0,\theta_{n}\right]  }\left(  r\right)  $
and
\begin{align*}
V_{t}^{\left(  n,+\right)  }  &  =\int_{0}^{t}\left[  \left(  \mu_{r}^{\left(
n\right)  }+\frac{1}{2n_{p}\lambda}\left(  \ell_{r}^{\left(  n\right)
}\right)  ^{2}\right)  ^{+}dr+\nu_{r}^{\left(  n\right)  +}dA_{r}\right] \\
&  \leq\int_{0}^{t}\mathbf{1}_{\left[  0,n\right]  }\left(  \beta_{r}\right)
\left[  \left(  \mu_{r}+\frac{1}{2n_{p}\lambda}\left(  \ell_{r}\right)
^{2}\right)  ^{+}dr+\nu_{r}^{+}dA_{r}\right] \\
&  =\int_{0}^{t\wedge\theta_{n}}\left[  \left(  \mu_{r}+\frac{1}{2n_{p}%
\lambda}\left(  \ell_{r}\right)  ^{2}\right)  ^{+}dr+\nu_{r}^{+}dA_{r}\right]
\\
&  =V_{t\wedge\theta_{n}}^{\left(  +\right)  }\leq V_{\theta_{n}}^{\left(
+\right)  }\leq n
\end{align*}
and%
\begin{align*}
&  \big|e^{V_{T}^{\left(  n,+\right)  }}\eta^{\left(  n\right)  }%
\big|^{2}+\Big(%
{\displaystyle\int_{0}^{T}}
e^{V_{r}^{\left(  n,+\right)  }}\left(  F_{1}^{\left(  n\right)  \#}\left(
r\right)  dr+G_{1}^{\left(  n\right)  \#}\left(  r\right)  dA_{r}\right)
\Big)^{2}\\
&  \leq n^{2}e^{2n}+e^{2n}n^{2}\left(
{\displaystyle\int_{0}^{T\wedge\theta_{n}}}
\left(  dr+dA_{r}\right)  \right)  ^{2}\leq e^{2n}n^{2}\left(  1+n^{2}\right)
=\tilde{L}^{\left(  n\right)  }%
\end{align*}
for every $\rho\geq0$
\begin{align*}
F_{\rho}^{\left(  n\right)  \#}\left(  t\right)  +G_{\rho}^{\left(  n\right)
\#}\left(  t\right)   &  \leq\left[  F_{\rho}^{\#}\left(  t\right)  +G_{\rho
}^{\#}\left(  t\right)  \right]  \mathbf{1}_{\left[  0,n\right]  }\left(
\beta_{t}\right) \\
&  \leq K_{\rho}\left(  \Theta_{t}\right)  \mathbf{1}_{\left[  0,n\right]
}\left(  \beta_{t}\right) \\
&  \leq K_{\rho}\left(  n\right)  .
\end{align*}
Therefore assumptions (\ref{ap-eq-1b}), (\ref{ap-eq-1bb}), (\ref{ap-eq-1c})
and (\ref{ap-eq-6}) are satisfied.

Hence by Lemma \ref{l1-strong sol} there exists a unique (strong) solution
$\left(  Y^{\left(  n\right)  },Z^{\left(  n\right)  },U^{\left(  n\right)
}\right)  \in S_{m}^{0}\left[  0,T\right]  \times\Lambda_{m\times k}%
^{0}\left(  0,T\right)  \times\Lambda_{m}^{0}\left(  0,T\right)  $ of BSDE
(\ref{st2-1}).

We have%
\begin{equation}%
\begin{array}
[c]{l}%
\left\langle Y_{t}^{\left(  n\right)  },H^{\left(  n\right)  }(t,Y_{t}%
^{\left(  n\right)  },Z_{t}^{\left(  n\right)  })-U_{s}^{\left(  n\right)
}\right\rangle dQ_{t}\medskip\\
\quad\leq\Big[\left(  \alpha_{t}\mu_{t}+\left(  1-\alpha_{t}\right)  \nu
_{t}+\alpha_{t}\dfrac{1}{2n_{p}\lambda}\ell_{t}^{2}\right)  \mathbf{1}%
_{\left[  0,n\right]  }\left(  \beta_{t}\right)  \big|Y_{t}^{\left(  n\right)
}\big|^{2}\medskip\\
\quad\quad+\alpha_{t}\mathbf{1}_{\left[  0,n\right]  }\left(  \beta
_{t}\right)  \dfrac{n_{p}\lambda}{2}\,\big|Z_{t}^{\left(  n\right)  }%
\big|^{2}+|H^{\left(  n\right)  }\left(  t,0,0\right)  |\big|Y_{t}^{\left(
n\right)  }\big|\Big]dQ_{t}\medskip\\
\quad\leq\big|Y_{t}^{\left(  n\right)  }\big|dN_{t}+\big|Y_{t}^{\left(
n\right)  }\big|^{2}dV_{t}^{\left(  +\right)  }+\dfrac{\lambda}{2}%
\,\big|Z_{t}^{\left(  n\right)  }\big|^{2}dr,
\end{array}
\label{yn-ineq}%
\end{equation}
where%
\[
N_{t}=\int_{0}^{t}\left[  \left\vert F\left(  r,0,0\right)  \right\vert
dr+\left\vert G\left(  t,0\right)  \right\vert dA_{r}\right]  \quad
\text{and}\quad V_{t}^{\left(  +\right)  }=\int_{0}^{t}\left[  \left(  \mu
_{r}+\frac{1}{2n_{p}\lambda}\ell_{r}^{2}\right)  ^{+}dr+\nu_{r}^{+}%
dA_{r}\right]  .
\]
Since by (\ref{b-1})%
\begin{align*}
\big|Y_{t}^{\left(  n\right)  }\big|^{2}  &  \leq\big|e^{V_{t}^{\left(
n,+\right)  }}Y_{t}^{\left(  n\right)  }\big|^{2}\leq\mathbb{E}^{\mathcal{F}%
_{t}}\Big(\sup\limits_{r\in\left[  t,T\right]  }\big|e^{V_{r}^{\left(
n,+\right)  }}Y_{r}^{\left(  n\right)  }\big|^{2}\Big)\\
&  \leq C_{\lambda}\tilde{L}^{\left(  n\right)  }=:\left(  \rho_{n}\right)
^{2},\quad\text{for all }t\in\left[  0,T\right]  \text{, }\mathbb{P}-a.s..
\end{align*}
and $\left\vert \eta^{\left(  n\right)  }\right\vert \leq\left\vert
\eta\right\vert $, we deduce from (\ref{yn-ineq}), by Proposition
\ref{Appendix_result 1}, that for all $t\in\left[  0,T\right]  :$%
\[%
\begin{array}
[c]{l}%
\mathbb{E}^{\mathcal{F}_{t}}\Big(\sup\limits_{r\in\left[  t,T\right]
}\big|e^{V_{r}^{\left(  +\right)  }}Y_{r}^{\left(  n\right)  }\big|^{2}%
\Big)+\mathbb{E}^{\mathcal{F}_{t}}\Big(%
{\displaystyle\int_{t}^{T}}
e^{2V_{r}^{\left(  +\right)  }}\big|Z_{r}^{\left(  n\right)  }\big|^{2}%
dr\Big)\medskip\\
\leq C_{\lambda}~\mathbb{E}^{\mathcal{F}_{t}}\left[  \big|e^{V_{T}^{\left(
+\right)  }}\eta\big|^{2}+\Big(%
{\displaystyle\int_{t}^{T}}
e^{V_{r}^{\left(  +\right)  }}dN_{r}\Big)^{2}\right]
\end{array}
\]
($C_{\lambda}=C_{p,\lambda}$ is the constant defined by (\ref{an3a}) for $p=2$ ).

By assumption (\ref{assump-H(t,0,0)}) we have%
\begin{align*}
\big|Y_{t}^{\left(  n\right)  }\big|  &  \leq\big|e^{V_{t}^{\left(  +\right)
}}Y_{t}^{\left(  n\right)  }\big|\leq\left[  \mathbb{E}^{\mathcal{F}_{t}%
}\Big(\sup\limits_{r\in\left[  t,T\right]  }\big|e^{V_{r}^{\left(  +\right)
}}Y_{r}^{\left(  n\right)  }\big|^{2}\Big)\right]  ^{1/2}\\
&  \leq(C_{\lambda}\hat{L})^{1/2}=\hat{\rho},\quad\text{for all }t\in\left[
0,T\right]  \text{, }\mathbb{P}-a.s.,
\end{align*}
and%
\[
\mathbb{E}\Big(%
{\displaystyle\int_{0}^{T}}
e^{2V_{r}^{\left(  +\right)  }}\big|Z_{r}^{\left(  n\right)  }\big|^{2}%
dr\Big)\leq\hat{\rho}^{2}.
\]
Let $n,i\in\mathbb{N}^{\ast}$. Then $Y^{\left(  n+i\right)  }-Y^{\left(
n\right)  }$ satisfies the following BSDE%
\begin{align*}
&  Y_{t}^{\left(  n+i\right)  }-Y_{t}^{\left(  n\right)  }+\int_{t}^{T}\left[
U_{s}^{\left(  n+i\right)  }-U_{s}^{\left(  n\right)  }\right]  dQ_{s}\\
&  =\eta^{\left(  n+i\right)  }-\eta^{\left(  n\right)  }+\int_{t}^{T}\left[
H^{\left(  n+i\right)  }\big(s,Y_{s}^{\left(  n+i\right)  },Z_{s}^{\left(
n+i\right)  }\big)-H^{\left(  n\right)  }\big(s,Y_{s}^{\left(  n\right)
},Z_{s}^{\left(  n\right)  }\big)\right]  dQ_{s}\\
&  \quad-\int_{t}^{T}\big(Z_{s}^{\left(  n+i\right)  }-Z_{s}^{\left(
n\right)  }\big)dB_{s}~.
\end{align*}
Since%
\[
\left\langle Y_{s}^{\left(  n+i\right)  }-Y_{s}^{\left(  n\right)
},\big(U_{s}^{\left(  n+i\right)  }-U_{s}^{\left(  n\right)  }\big)dQ_{s}%
\right\rangle \geq0
\]
and\medskip\newline$%
\begin{array}
[c]{c}%
\;
\end{array}
\left\langle Y_{s}^{\left(  n+i\right)  }-Y_{s}^{\left(  n\right)  },\left[
H^{\left(  n+i\right)  }\big(s,Y_{s}^{\left(  n+i\right)  },Z_{s}^{\left(
n+i\right)  }\big)-H^{\left(  n\right)  }\big(s,Y_{s}^{\left(  n\right)
},Z_{s}^{\left(  n\right)  }\big)\right]  dQ_{s}\right\rangle $\medskip
\newline$%
\begin{array}
[c]{c}%
\quad
\end{array}
=\left\langle Y_{s}^{\left(  n+i\right)  }-Y_{s}^{\left(  n\right)  },\left[
H^{\left(  n+i\right)  }\big(s,Y_{s}^{\left(  n+i\right)  },Z_{s}^{\left(
n+i\right)  }\big)-H^{\left(  n+i\right)  }\big(s,Y_{s}^{\left(  n\right)
},Z_{s}^{\left(  n\right)  }\big)\right]  dQ_{s}\right\rangle $\medskip
\newline$%
\begin{array}
[c]{c}%
\quad\quad
\end{array}
+\left\langle Y_{s}^{\left(  n+i\right)  }-Y_{s}^{\left(  n\right)  },\left[
H^{\left(  n+i\right)  }\big(s,Y_{s}^{\left(  n\right)  },Z_{s}^{\left(
n\right)  }\big)-H^{\left(  n\right)  }\big(s,Y_{s}^{\left(  n\right)  }%
,Z_{s}^{\left(  n\right)  }\big)\right]  dQ_{s}\right\rangle $\medskip
\newline$%
\begin{array}
[c]{c}%
\quad
\end{array}
\leq\mathbf{1}_{\left[  0,n+i\right]  }\left(  \beta_{s}\right)  \left(
\mu_{s}ds+\nu_{s}dA_{s}+\dfrac{1}{2n_{p}\lambda}\ell_{s}^{2}ds\right)
\big|Y_{s}^{\left(  n+i\right)  }-Y_{s}^{\left(  n\right)  }\big|^{2}%
+\dfrac{n_{p}\lambda}{2}\big|Z_{s}^{\left(  n+i\right)  }-Z_{s}^{\left(
n\right)  }\big|^{2}ds$\medskip\newline$%
\begin{array}
[c]{c}%
\quad\quad
\end{array}
+\big|Y_{s}^{\left(  n+i\right)  }-Y_{s}^{\left(  n\right)  }\big|\,\left\vert
\mathbf{1}_{\left[  0,n+i\right]  }\left(  \beta_{s}\right)  -\mathbf{1}%
_{\left[  0,n\right]  }\left(  \beta_{s}\right)  \right\vert \left[  2\ell
_{s}\big|Z_{s}^{\left(  n\right)  }\big|+F_{\hat{\rho}}^{\#}\left(  s\right)
ds+G_{\hat{\rho}}^{\#}\left(  s\right)  dA_{s}\right]  $\medskip\newline$%
\begin{array}
[c]{c}%
\quad
\end{array}
\leq\big|Y_{s}^{\left(  n+i\right)  }-Y_{s}^{\left(  n\right)  }%
\big|\,\mathbf{1}_{\left(  n,\infty\right)  }\left(  \beta_{s}\right)  \left[
2\ell_{s}\big|Z_{s}^{\left(  n\right)  }\big|+F_{\hat{\rho}}^{\#}\left(
s\right)  ds+G_{\hat{\rho}}^{\#}\left(  s\right)  dA_{s}\right]  $%
\medskip\newline$%
\begin{array}
[c]{c}%
\quad\quad
\end{array}
+\big|Y_{s}^{\left(  n+i\right)  }-Y_{s}^{\left(  n\right)  }\big|^{2}%
dV_{s}^{\left(  +\right)  }+\dfrac{n_{a}\lambda^{\prime}}{2}\big|Z_{s}%
^{\left(  n+i\right)  }-Z_{s}^{\left(  n\right)  }\big|^{2}ds$\medskip\newline
where using the assumption $\left(  ii\right)  $ of our Proposition we have
$1+n_{p}\lambda<a<p\wedge2$, $n_{a}=\left(  a-1\right)  \wedge1=a-1$ and
\[
n_{p}\lambda\leq\left(  a-1\right)  \frac{n_{p}\lambda+a-1}{2\left(
a-1\right)  }=n_{a}\lambda^{\prime}\quad\text{with }\lambda^{\prime}%
=\frac{n_{p}\lambda+a-1}{2\left(  a-1\right)  }\in\left(  0,1\right)
\]
Now by Proposition \ref{Appendix_result 1}, we obtain :$\medskip$\newline$%
\begin{array}
[c]{l}%
\text{\quad}%
\end{array}
\mathbb{E}\sup\limits_{s\in\left[  0,T\right]  }e^{aV_{s}^{\left(  +\right)
}}\big|Y_{s}^{\left(  n+i\right)  }-Y_{s}^{\left(  n\right)  }\big|^{a}%
+\mathbb{E}\left(
{\displaystyle\int_{0}^{T}}
e^{aV_{s}^{\left(  +\right)  }}\big|Z_{s}^{\left(  n+i\right)  }%
-Z_{s}^{\left(  n\right)  }\big|^{2}ds\right)  ^{a/2}\medskip$\newline$%
\begin{array}
[c]{l}%
\text{\quad}%
\end{array}
\leq C_{a,\lambda}\mathbb{E}\left[  e^{aV_{T}^{\left(  +\right)  }}\left\vert
\eta\right\vert ^{a}\mathbf{1}_{\left(  n,\infty\right)  }\left(  \left\vert
\eta\right\vert +V_{T}^{\left(  +\right)  }\right)  \right]  \medskip$%
\newline$%
\begin{array}
[c]{l}%
\text{\quad\quad}%
\end{array}
\mathbb{+}C_{a,\lambda}~\mathbb{E~}\left(
{\displaystyle\int_{0}^{T}}
e^{V_{s}^{\left(  +\right)  }}\mathbf{1}_{\left(  n,\infty\right)  }\left(
\beta_{s}\right)  \left[  2\ell_{s}\left\vert Z_{s}^{\left(  n\right)
}\right\vert +F_{\hat{\rho}}^{\#}\left(  s\right)  ds+G_{\hat{\rho}}%
^{\#}\left(  s\right)  dA_{s}\right]  \right)  ^{a}\medskip\newline%
\begin{array}
[c]{l}%
\text{\quad}%
\end{array}
\leq C_{a,\lambda}\mathbb{E}\left[  e^{aV_{T}^{\left(  +\right)  }}\left\vert
\eta\right\vert ^{a}\mathbf{1}_{\left(  n,\infty\right)  }\left(  \left\vert
\eta\right\vert +V_{T}^{\left(  +\right)  }\right)  \right]  \mathbb{+}%
2C_{a,\lambda}~\mathbb{E~}\left(
{\displaystyle\int_{0}^{T}}
e^{V_{s}^{\left(  +\right)  }}\mathbf{1}_{\left(  n,\infty\right)  }\left(
\beta_{s}\right)  2\ell_{s}\left\vert Z_{s}^{\left(  n\right)  }\right\vert
ds\right)  ^{a}\medskip$\newline$%
\begin{array}
[c]{l}%
\text{\quad\quad}%
\end{array}
+2C_{a,\lambda}~\mathbb{E~}\left(
{\displaystyle\int_{0}^{T}}
e^{V_{s}^{\left(  +\right)  }}\mathbf{1}_{\left(  n,\infty\right)  }\left(
\beta_{s}\right)  \left(  F_{\hat{\rho}}^{\#}\left(  s\right)  ds+G_{\hat
{\rho}}^{\#}\left(  s\right)  dA_{s}\right)  \right)  ^{a}\medskip$\newline$%
\begin{array}
[c]{l}%
\text{\quad}%
\end{array}
\leq C_{a,\lambda}\mathbb{E}\left[  e^{aV_{T}^{\left(  +\right)  }}\left\vert
\eta\right\vert ^{a}\mathbf{1}_{\left(  n,\infty\right)  }\left(  \left\vert
\eta\right\vert +V_{T}^{\left(  +\right)  }\right)  \right]  +C_{a,\lambda
}^{\prime}\mathbb{E}\left[  \big(%
{\displaystyle\int_{0}^{T}}
\ell_{s}^{2}\mathbf{1}_{(n,\infty)}\left(  \beta_{s}\right)  ds\big)^{a/2}%
\big(%
{\displaystyle\int_{0}^{T}}
e^{2V_{s}^{\left(  +\right)  }}\left\vert Z_{s}^{n}\right\vert ^{2}%
ds\big)^{a/2}\right]  \medskip$\newline$%
\begin{array}
[c]{l}%
\text{\quad\quad}%
\end{array}
+2C_{a,\lambda}~\mathbb{E~}\left(
{\displaystyle\int_{0}^{T}}
e^{V_{s}^{\left(  +\right)  }}\mathbf{1}_{\left(  n,\infty\right)  }\left(
\beta_{s}\right)  \left(  F_{\hat{\rho}}^{\#}\left(  s\right)  ds+G_{\hat
{\rho}}^{\#}\left(  s\right)  dA_{s}\right)  \right)  ^{a}\medskip$\newline$%
\begin{array}
[c]{l}%
\text{\quad}%
\end{array}
\leq C_{a,\lambda}\hat{L}^{a/2}\mathbb{E}\left[  \mathbf{1}_{\left(
n,\infty\right)  }\left(  \left\vert \eta\right\vert +V_{T}^{\left(  +\right)
}\right)  \right]  +C_{a,\lambda}^{\prime}\left[  \mathbb{E}\big(%
{\displaystyle\int_{0}^{T}}
\ell_{s}^{2}\mathbf{1}_{(n,\infty)}\left(  \beta_{s}\right)  ds\big)^{\frac
{a}{2-a}}\right]  ^{\frac{2-a}{2}}\left(  \mathbb{E}\big(%
{\displaystyle\int_{0}^{T}}
e^{2V_{s}^{\left(  +\right)  }}\left\vert Z_{s}^{n}\right\vert ^{2}%
ds\big)\right)  ^{\frac{a}{2}}\medskip$\newline$%
\begin{array}
[c]{l}%
\text{\quad\quad}%
\end{array}
+2C_{a,\lambda}~\mathbb{E}\left(
{\displaystyle\int_{0}^{T}}
e^{V_{s}^{\left(  +\right)  }}\mathbf{1}_{\left(  n,\infty\right)  }\left(
\beta_{s}\right)  \left(  F_{\hat{\rho}}^{\#}\left(  s\right)  ds+G_{\hat
{\rho}}^{\#}\left(  s\right)  dA_{s}\right)  \right)  ^{a}.\medskip$

Hence there exists $\left(  Y,Z\right)  \in S_{m}^{0}\left[  0,T\right]
\times\Lambda_{m\times k}^{0}\left(  0,T\right)  $ such that%
\begin{equation}%
\begin{array}
[c]{rl}%
\left(  j\right)  \quad & \left\vert Y_{t}\right\vert \leq e^{V_{t}^{\left(
+\right)  }}\left\vert Y_{t}\right\vert \leq(C_{\lambda}\hat{L})^{1/2}%
=\hat{\rho},\quad\text{for all }t\in\left[  0,T\right]  ,\;\mathbb{P}%
\text{-a.s.},\medskip\\
\left(  jj\right)  \quad & \mathbb{E}%
{\displaystyle\int_{0}^{T}}
e^{2V_{r}^{\left(  +\right)  }}\left\vert R_{r}\right\vert ^{2}dr\leq\hat
{\rho}^{2},\medskip\\
\left(  jjj\right)  \quad & \lim\limits_{n\rightarrow\infty}\mathbb{E}%
\sup\limits_{s\in\left[  0,T\right]  }e^{aV_{s}^{\left(  +\right)  }%
}\left\vert Y_{s}^{n}-Y_{s}\right\vert ^{a}+\mathbb{E}\left(
{\displaystyle\int_{0}^{T}}
e^{aV_{s}^{\left(  +\right)  }}\left\vert Z_{s}^{n}-Z_{s}\right\vert
^{2}ds\right)  ^{a/2}=0,\medskip\\
\left(  jv\right)  \quad & \left(  Y_{t},Z_{t}\right)  =\left(  \eta,0\right)
,\text{ for all }t>T.
\end{array}
\label{c-1}%
\end{equation}
We remark that%
\[
\varphi\big(Y_{t}^{\left(  n\right)  }\big)dt+\psi\big(Y_{t}^{\left(
n\right)  }\big)dA_{t}\leq\left\langle Y_{t}^{\left(  n\right)  }%
,U_{t}^{\left(  1,n\right)  }\right\rangle dt+\left\langle Y_{t}^{\left(
n\right)  },U_{t}^{\left(  2,n\right)  }\right\rangle dA_{t}%
\]
and\medskip\newline$%
\begin{array}
[c]{c}%
\quad
\end{array}
\varphi\big(Y_{t}^{\left(  n\right)  }\big)dt+\psi\big(Y_{t}^{\left(
n\right)  }\big)dA_{t}+\left\langle Y_{t}^{\left(  n\right)  },H(t,Y_{t}%
^{\left(  n\right)  },Z_{t}^{\left(  n\right)  })-U_{t}^{\left(  n\right)
}\right\rangle dQ_{t}$\medskip\newline$%
\begin{array}
[c]{c}%
\quad
\end{array}
\leq\left\langle Y_{t}^{\left(  n\right)  },H(t,Y_{t}^{\left(  n\right)
},Z_{t}^{\left(  n\right)  })\right\rangle dQ_{t}$\medskip\newline$%
\begin{array}
[c]{c}%
\quad
\end{array}
\leq\Big[\left(  \alpha_{t}\mu_{t}+\left(  1-\alpha_{t}\right)  \nu_{t}%
+\alpha_{t}\dfrac{1}{2n_{p}\lambda}\ell_{t}^{2}\right)  \,\mathbf{1}_{\left[
0,n\right]  }\left(  \beta_{t}\right)  \big|Y_{t}^{\left(  n\right)
}\big|^{2}$\medskip\newline$%
\begin{array}
[c]{c}%
\quad
\end{array}
\quad+\alpha_{t}\mathbf{1}_{\left[  0,n\right]  }\left(  \beta_{t}\right)
\,\dfrac{n_{p}\lambda}{2}\,\big|Z_{t}^{\left(  n\right)  }\big|^{2}+\left\vert
H^{\left(  n\right)  }\left(  t,0,0\right)  \right\vert \big|Y_{t}^{\left(
n\right)  }\big|\Big]dQ_{t}$\medskip\newline$%
\begin{array}
[c]{c}%
\quad
\end{array}
\leq\big|Y_{t}^{\left(  n\right)  }\big|dN_{t}+\big|Y_{t}^{\left(  n\right)
}\big|^{2}dV_{t}^{\left(  +\right)  }+\dfrac{n_{p}\lambda}{2}\,\big|Z_{t}%
^{\left(  n\right)  }\big|^{2}dr,$\medskip\newline where%
\[
N_{t}=\int_{0}^{t}\left[  \left\vert F\left(  r,0,0\right)  \right\vert
dr+\left\vert G\left(  t,0\right)  \right\vert dA_{r}\right]  .
\]
Also by (\ref{c-1}-j) and the assumption (\ref{ip-mnl}) we have%
\[
\mathbb{E}\sup_{t\in\left[  0,T\right]  }e^{pV_{t}^{\left(  +\right)  }%
}\big|Y_{t}^{\left(  n\right)  }\big|^{p}\leq\hat{\rho}^{p}~\mathbb{E}%
\exp\left(  p\int_{0}^{T}\left(  \left\vert \mu_{s}\right\vert +\frac
{1}{2n_{p}\lambda}\,\ell_{s}^{2}\right)  ds+p\int_{0}^{T}\left\vert \nu
_{s}\right\vert dA_{s}\right)  <\infty.
\]
Hence by Proposition \ref{Appendix_result 1} we deduce for all $t\in\left[
0,T\right]  $%
\begin{equation}%
\begin{array}
[c]{l}%
\displaystyle\mathbb{E}^{\mathcal{F}_{t}}\Big(\sup\limits_{s\in\left[
t,T\right]  }\big|e^{V_{s}^{\left(  +\right)  }}Y_{s}^{\left(  n\right)
}\big|^{p}\Big)+\mathbb{E}^{\mathcal{F}_{t}}\Big(%
{\displaystyle\int_{t}^{T}}
e^{2V_{s}^{\left(  +\right)  }}\left(  \varphi\left(  Y_{s}^{\left(  n\right)
}\right)  ds+\psi\left(  Y_{s}^{\left(  n\right)  }\right)  dA_{s}\right)
\Big)^{p/2}\medskip\\
\quad+\mathbb{E}^{\mathcal{F}_{t}}\Big(%
{\displaystyle\int_{t}^{T}}
e^{2V_{s}^{\left(  +\right)  }}\big|Z_{s}^{\left(  n\right)  }\big|^{2}%
ds\Big)^{p/2}\medskip\\
\displaystyle\leq C_{p,\lambda}~\mathbb{E}^{\mathcal{F}_{t}}\left[
e^{pV_{T}^{\left(  +\right)  }}\left\vert \eta\right\vert ^{p}+\Big(%
{\displaystyle\int_{t}^{T}}
e^{V_{s}^{\left(  +\right)  }}dN_{s}\Big)^{p}\right]  ,\;\;a.s..
\end{array}
\label{es-np}%
\end{equation}
By Remark \ref{s-w} and%
\begin{equation}
V_{t}=\int_{0}^{t}\left[  \left(  \mu_{r}+\dfrac{1}{2n_{p}\lambda}\ell_{r}%
^{2}\right)  dr+\nu_{r}dA_{r}\right]  \leq V_{t}^{\left(  +\right)  },
\label{es-p-a}%
\end{equation}
$\left(  Y^{\left(  n\right)  },Z^{\left(  n\right)  }\right)  $ as strong
solution of (\ref{st2-1}) is also an $L^{p}-$variational solution on $\left[
0,T\right]  $ for (\ref{st2-1}). \medskip

Hence for $q\in\{2,p\wedge2\},$ $\delta_{q}=\delta\mathbf{1}_{[1,2)}\left(
q\right)  $ and $\Gamma_{t}^{\left(  n\right)  }=\big(\big|M_{t}%
-Y_{t}^{\left(  n\right)  }\big|^{2}+\delta_{q}\big)^{1/2}$ it holds%
\begin{equation}%
\begin{array}
[c]{l}%
\big(\Gamma_{t}^{\left(  n\right)  }\big)^{q}+\dfrac{q\left(  q-1\right)  }%
{2}\,%
{\displaystyle\int_{t}^{s}}
{\big(\Gamma_{r}^{\left(  n\right)  }\big)^{q-2}}{\Large \,}\big|R_{r}%
-Z_{r}^{\left(  n\right)  }\big|^{2}dr+{q%
{\displaystyle\int_{t}^{s}}
}\big({\Gamma_{r}^{\left(  n\right)  }\big)^{q-2}\Psi}\big(r,Y_{r}^{\left(
n\right)  }\big)dQ_{r}\medskip\\
\leq\big(\Gamma_{s}^{\left(  n\right)  }\big)^{q}+{q%
{\displaystyle\int_{t}^{s}}
}\big({\Gamma_{r}^{\left(  n\right)  }\big)^{q-2}\Psi}\left(  r,M_{r}\right)
dQ_{r}\medskip\\
\quad+q%
{\displaystyle\int_{t}^{s}}
{\big(\Gamma_{r}^{\left(  n\right)  }\big)^{q-2}}\langle M_{r}-Y_{r}^{\left(
n\right)  },N_{r}-H\big(r,Y_{r}^{\left(  n\right)  },Z_{r}^{\left(  n\right)
}\big)\rangle dQ_{r}\medskip\\
\quad-q%
{\displaystyle\int_{t}^{s}}
\big({\Gamma_{r}^{\left(  n\right)  }\big)^{q-2}}\,\langle M_{r}%
-Y_{r}^{\left(  n\right)  },{\big(}R_{r}-Z_{r}^{\left(  n\right)  }%
{\big)}dB_{r}\rangle;
\end{array}
\label{an-vws}%
\end{equation}
for all $\delta\in(0,1]$, for all $0\leq t\leq s\leq T,$ for all $M\in
S_{m}^{0}\left(  \gamma,N,R;V\right)  .$

By convergence result (\ref{c-1}-$\left(  jjj\right)  $) and the assumptions
$(\mathrm{A}_{3}-\mathrm{A}_{5})$ we can pass to $\liminf_{n\rightarrow
+\infty}$ (on a subsequence) in (\ref{es-np}) and (\ref{an-vws}) to conclude
that $\left(  Y,Z\right)  $ is also an $L^{p}-$variational solution on
$\left[  0,T\right]  $ and the inequality (\ref{es-p}) holds.\hfill
\end{proof}

\begin{corollary}
Let the assumptions of Proposition \ref{p1-wvs} be satisfied. If, moreover,
$\varphi=\psi=0,$ then BSDE%
\begin{equation}
Y_{t}=\eta+{\int_{t}^{T}}H\left(  r,Y_{r},Z_{r}\right)  dQ_{r}-{\int_{t}^{T}%
}Z_{r}dB_{r}\,,\quad\text{a.s., for all }t\in\left[  0,T\right]  ,
\label{bsde-cl}%
\end{equation}
has a unique solution $\left(  Y,Z\right)  \in S_{m}^{p}\left[  0,T\right]
\times\Lambda_{m\times k}^{p}\left(  0,T\right)  .$
\end{corollary}

\begin{proof}
Based on the results from (\ref{c-1}) and the assumptions $(\mathrm{A}%
_{3}-\mathrm{A}_{5})$ we can pass to limit $\lim_{n\rightarrow\infty}$ in the
approximating equation (\ref{st2-1}) with $\varphi=\psi=0$ and $U=U^{\left(
1\right)  }=U^{\left(  2\right)  }=0$ to infer that $\left(  Y,Z\right)  $
satisfies (\ref{bsde-cl}). From (\ref{es-p}), (\ref{es-p-a}) and the
assumption (\ref{assump-H(t,0,0)}) we get $\left(  Y,Z\right)  \in S_{m}%
^{p}\left[  0,T\right]  \times\Lambda_{m\times k}^{p}\left(  0,T\right)  .$
Moreover by (\ref{c-1}-j)
\[
\left\vert Y_{t}\right\vert \leq\hat{\rho},\quad\text{for all }t\in\left[
0,T\right]  ,\;\mathbb{P}\text{-a.s..}%
\]

\hfill
\end{proof}

\begin{corollary}
Let the assumptions of Proposition \ref{p1-wvs} be satisfied. If, moreover,%
\begin{equation}%
\begin{array}
[c]{rl}%
\left(  i\right)  & \mathbb{E}\left[  e^{2V_{T}^{\left(  +\right)  }}\left(
\varphi(\eta)+\psi(\eta)\right)  \right]  <\infty,\medskip\\
\left(  ii\right)  & \mathbb{E}%
{\displaystyle\int_{0}^{T}}
e^{2V_{r}^{\left(  +\right)  }}dQ_{r}<\infty,\medskip\\
\left(  iii\right)  & \mathbb{E}%
{\displaystyle\int_{0}^{T}}
e^{2V_{r}^{\left(  +\right)  }}\left(  \left\vert F_{\hat{\rho}}^{\#}\left(
r\right)  \right\vert ^{2}dr+\left\vert G_{\hat{\rho}}^{\#}\left(  r\right)
\right\vert ^{2}dA_{r}\right)  <\infty
\end{array}
\label{ip-diez}%
\end{equation}
where\footnote{The constant $\hat{L}$ is given by (\ref{assump-H(t,0,0)}) and
the constant $C_{\lambda}=C_{p,\lambda}$ is given by (\ref{an3a}) with $p=2.$}
$\hat{\rho}=(C_{\lambda}\hat{L})^{1/2},$ then the BSDE%
\[
\left\{
\begin{array}
[c]{l}%
\displaystyle Y_{t}+{\int_{t}^{T}}dK_{r}=Y_{T}+{\int_{t}^{T}}H\left(
r,Y_{r},Z_{r}\right)  dQ_{r}-{\int_{t}^{T}}Z_{r}dB_{r},\;\text{a.s., for all
}t\in\left[  0,T\right]  ,\medskip\\
\displaystyle dK_{r}=U_{r}^{\left(  1\right)  }dr+U_{r}^{\left(  2\right)
}dA_{r}\,,\medskip\\
\displaystyle U^{\left(  1\right)  }dr\in\partial\varphi\left(  Y_{r}\right)
dr\quad\text{and}\quad U^{\left(  2\right)  }dA_{r}\in\partial\psi\left(
Y_{r}\right)  dA_{r}%
\end{array}
\,\right.
\]
has a unique strong a solution $\left(  Y,Z,U^{\left(  1\right)  },U^{\left(
2\right)  }\right)  \in S_{m}^{0}\times\Lambda_{m\times k}^{0}\times
\Lambda_{m}^{0}$ $\times\Lambda_{m}^{0}$ such that
\begin{equation}
\mathbb{E}\sup_{t\in\left[  0,T\right]  }e^{2V_{t}}\left\vert Y_{t}\right\vert
^{2}+\mathbb{E}\left(
{\displaystyle\int_{0}^{T}}
{e^{2V_{r}}}\left\vert Z_{r}\right\vert ^{2}dr\right)  +\mathbb{E}\left(
{\displaystyle\int_{0}^{T}}
{e^{2V_{r}}}\left\vert U_{r}^{\left(  1\right)  }\right\vert ^{2}dr\right)
+\mathbb{E}\left(
{\displaystyle\int_{0}^{T}}
{e^{2V_{r}}}\left\vert U_{r}^{\left(  2\right)  }\right\vert ^{2}%
dA_{r}\right)  <\infty. \label{0}%
\end{equation}
Moreover
\[
\left\vert Y_{t}\right\vert \leq e^{V_{t}^{\left(  +\right)  }}\left\vert
Y_{t}\right\vert \leq(C_{\lambda}\hat{L})^{1/2}=\hat{\rho},\quad\text{for all
}t\in\left[  0,T\right]  ,\;\mathbb{P}\text{-a.s.}.
\]

\end{corollary}

\begin{proof}
We expand the proof of Proposition \ref{p1-wvs}. By (\ref{b-1a})\medskip
\newline$%
\begin{array}
[c]{c}%
\quad
\end{array}
\dfrac{1}{2}\,\mathbb{E}%
{\displaystyle\int_{0}^{T}}
e^{2V_{r}^{\left(  n,+\right)  }}\Big[\big|U_{r}^{\left(  1,n\right)
}\big|^{2}dr+\big|U_{r}^{\left(  2,n\right)  }\big|^{2}dA_{r}\Big]$%
\medskip\newline$%
\begin{array}
[c]{c}%
\quad
\end{array}
\leq\mathbb{E}\left[  e^{2V_{T}^{\left(  n,+\right)  }}\left(  \varphi
(\eta^{\left(  n\right)  })+\psi(\eta^{\left(  n\right)  })\right)  \right]
+\mathbb{E}%
{\displaystyle\int_{0}^{T}}
e^{2V_{r}^{\left(  n,+\right)  }}\left(  1+6\big|Z_{r}^{\left(  n\right)
}\big|^{2}+6\big|F^{\left(  n\right)  }(r,Y_{r}^{\left(  n\right)
},0)\big|^{2}\right)  dr$\medskip\newline$%
\begin{array}
[c]{c}%
\quad\quad
\end{array}
+\mathbb{E}%
{\displaystyle\int_{0}^{T}}
e^{2V_{r}^{\left(  n,+\right)  }}\left(  1+3\big|G^{\left(  n\right)
}(r,Y_{r}^{\left(  n\right)  })\big|^{2}\right)  dA_{r}$\medskip\newline$%
\begin{array}
[c]{c}%
\quad
\end{array}
\leq\mathbb{E}\left[  e^{2V_{T}^{\left(  +\right)  }}\left(  \varphi
(\eta)+\psi(\eta)\right)  \right]  +\mathbb{E}%
{\displaystyle\int_{0}^{T}}
e^{2V_{r}^{\left(  +\right)  }}\left(  dr+dA_{r}\right)  +6\hat{\rho}^{2}%
$\medskip\newline$%
\begin{array}
[c]{c}%
\quad\quad
\end{array}
+6\mathbb{E}%
{\displaystyle\int_{0}^{T}}
e^{2V_{r}^{\left(  +\right)  }}\big|F_{\hat{\rho}}^{\#}\left(  r\right)
\big|^{2}dr+3\mathbb{E}%
{\displaystyle\int_{0}^{T}}
e^{2V_{r}^{\left(  +\right)  }}\big|G_{\hat{\rho}}^{\#}\left(  r\right)
\big|^{2}dA_{r}.$\medskip\newline Hence there exists $\left(  U^{\left(
1\right)  },U^{\left(  2\right)  }\right)  =\left(  e^{-V^{\left(  +\right)
}}\hat{U}^{\left(  1\right)  },e^{-V^{\left(  +\right)  }}\hat{U}^{\left(
2\right)  }\right)  \in\Lambda_{m}^{0}\left(  0,T\right)  \times\Lambda
_{m}^{0}\left(  0,T\right)  $ such that on a subsequence also denoted
$\left\{  U^{\left(  1,n\right)  },U^{\left(  2,n\right)  };\;n\in
\mathbb{N}^{\ast}\right\}  $
\[%
\begin{array}
[c]{l}%
e^{V_{r}^{\left(  +\right)  }}U^{\left(  1,n\right)  }\rightharpoonup
e^{V^{\left(  +\right)  }}U^{\left(  1\right)  },\quad\text{weakly in }%
L^{2}\left(  \Omega\times\left[  0,T\right]  ,d\mathbb{P}\otimes
dt;\mathbb{R}^{m}\right)  ,\medskip\\
e^{V^{\left(  +\right)  }}U^{\left(  2,n\right)  }\rightharpoonup
e^{V^{\left(  +\right)  }}U^{\left(  2\right)  },\quad\text{weakly in }%
L^{2}\left(  \Omega\times\left[  0,T\right]  ,d\mathbb{P}\otimes
dA_{t};\mathbb{R}^{m}\right)  .
\end{array}
\]
and
\[%
\begin{array}
[c]{l}%
\dfrac{1}{2}\,\mathbb{E}%
{\displaystyle\int_{0}^{T}}
e^{2V_{r}^{\left(  +\right)  }}\Big[\big|U_{r}^{\left(  1\right)  }%
\big|^{2}dr+\big|U_{r}^{\left(  2\right)  }\big|^{2}dA_{r}\Big]\medskip\\
\leq\mathbb{E}\left[  e^{2V_{T}^{\left(  +\right)  }}\left(  \varphi
(\eta)+\psi(\eta)\right)  \right]  +\mathbb{E}%
{\displaystyle\int_{0}^{T}}
e^{2V_{r}^{\left(  +\right)  }}\left(  1+6\left\vert Z_{r}\right\vert
^{2}+6\left\vert F\left(  r,Y_{r},0\right)  \right\vert ^{2}\right)
dr\medskip\\
\quad+\mathbb{E}%
{\displaystyle\int_{0}^{T}}
e^{2V_{r}^{\left(  +\right)  }}\left(  1+3\left\vert G\left(  r,Y_{r}\right)
\right\vert ^{2}\right)  dA_{r}\,.
\end{array}
\]
Passing to $\lim_{n\rightarrow\infty}$ in (\ref{st2-1}) using the results from
the proof of Proposition \ref{p1-wvs} we infer%
\begin{equation}
\left\{
\begin{array}
[c]{l}%
Y_{t}+%
{\displaystyle\int_{t}^{T}}
U_{s}dQ_{s}=\eta+%
{\displaystyle\int_{t}^{T}}
H\left(  s,Y_{s},Z_{s}\right)  dQ_{s}-%
{\displaystyle\int_{t}^{T}}
Z_{s}dB_{s}\,,\;t\in\left[  0,T\right]  ,\medskip\\
U_{s}=\alpha_{r}U_{r}^{\left(  1\right)  }+\left(  1-\alpha_{r}\right)
U_{r}^{\left(  2\right)  }\medskip\\
U_{s}^{\left(  1\right)  }\in\partial\varphi(Y_{s})~,\quad d\mathbb{P}\otimes
ds-a.e.\quad\text{and }\quad U_{s}^{\left(  2\right)  }\in\partial\psi
(Y_{s}),\quad d\mathbb{P}\otimes dA_{s}-a.e.\quad\text{on }\left[  0,T\right]
,
\end{array}
\right.  \label{st2-11}%
\end{equation}
and the conclusion follows.\hfill
\end{proof}

\begin{theorem}
[$L^{p}$-- variational solution]\label{t2exist}Let $0<\lambda<1<p,$
$n_{p}=\left(  p-1\right)  \wedge1$ and $q\in\left\{  2,p\wedge2\right\}  .$
We suppose that assumptions $\left(  \mathrm{A}_{1}-\mathrm{A}_{6}\right)  $
are satisfied and
\begin{equation}
\mathbb{E}\left[  e^{pV_{T}}\left\vert \eta\right\vert ^{p}+\Big(%
{\displaystyle\int_{0}^{T}}
e^{V_{s}}\left(  \left\vert F\left(  r,0,0\right)  \right\vert dr+\left\vert
G\left(  t,0\right)  \right\vert dA_{r}\right)  \Big)^{p}\right]  <\infty,
\label{ip-t2-2}%
\end{equation}
where $V$ is defined by (\ref{defV_1}). We also assume \newline$\left(
i\right)  \quad$ there exists $a\in\left(  1+n_{p}\lambda,p\wedge2\right)  $
such that%
\begin{equation}%
\begin{array}
[c]{rl}%
\left(  a\right)  \quad & \mathbb{E}\Big(\displaystyle\int_{0}^{T}\ell_{s}%
^{2}ds\Big)^{\frac{a}{2-a}}<\infty,\medskip\\
\left(  b\right)  \quad & \mathbb{E}\left[
{\displaystyle\int_{0}^{T}}
e^{V_{s}^{\left(  +\right)  }}\left(  F_{\rho}^{\#}\left(  s\right)
ds+G_{\rho}^{\#}\left(  s\right)  dA_{s}\right)  \right]  ^{a}<\infty
,\quad\text{for all }\rho>0,
\end{array}
\label{ip-t2-1}%
\end{equation}
where $V^{\left(  +\right)  }$ be given by (\ref{defV_2}) and $F_{\rho}^{\#}$,
$G_{\rho}^{\#}$ are defined by (\ref{def F sharp}), \newline$\left(
ii\right)  \quad$there exists a p.m.s.p. $\left(  \Theta_{t}\right)
_{t\in\left[  0,T\right]  }$ and for each $\rho\geq0$ there exist an
non-decreasing function $K_{\rho}:\mathbb{R}_{+}\rightarrow\mathbb{R}_{+}$
such that%
\begin{equation}
F_{\rho}^{\#}\left(  t\right)  +G_{\rho}^{\#}\left(  t\right)  \leq K_{\rho
}\left(  \Theta_{t}\right)  \text{,}\quad\text{a.e. }t\in\left[  0,T\right]  .
\label{ip-t2-1a}%
\end{equation}
Then the multivalued BSDE
\[
\left\{
\begin{array}
[c]{r}%
\displaystyle Y_{t}+{\int_{t}^{T}}dK_{r}=\eta+{\int_{t}^{T}}H\left(
r,Y_{r},Z_{r}\right)  dQ_{r}-{\int_{t}^{T}}Z_{r}dB_{r}\,,\quad\text{a.s., for
all }t\in\left[  0,T\right]  ,\medskip\\
\multicolumn{1}{l}{\displaystyle dK_{r}=U_{r}dQ_{r}\in\partial_{y}\Psi\left(
r,Y_{r}\right)  dQ_{r}}%
\end{array}
\,\right.
\]
has a unique $L^{p}$--variational solution, in the sense of Definition
\ref{definition_weak solution}.

Moreover this solution satisfies%
\[%
\begin{array}
[c]{l}%
\mathbb{E}\Big(\sup\limits_{t\in\left[  0,T\right]  }e^{pV_{t}}\left\vert
Y_{t}\right\vert ^{p}\Big)+\mathbb{E}\left(
{\displaystyle\int_{0}^{T}}
{e^{2V_{r}}}\left\vert Z_{r}\right\vert ^{2}dr\right)  ^{p/2}+\mathbb{E}%
\left(
{\displaystyle\int_{0}^{T}}
e^{2V_{r}}{\Psi}\left(  r,Y_{r}\right)  dQ_{r}\right)  ^{p/2}\medskip\\
\quad+\mathbb{E}\left(
{\displaystyle\int_{0}^{T}}
{e^{qV_{r}}\left\vert Y_{r}\right\vert ^{q-2}}\left\vert Z_{r}\right\vert
^{2}dr\right)  ^{p/q}+\mathbb{E}\left(
{\displaystyle\int_{0}^{T}}
{e^{qV_{r}}\left\vert Y_{r}\right\vert ^{q-2}\Psi}\left(  r,Y_{r}\right)
dQ_{r}\right)  ^{p/q}\medskip\\
\leq C_{p,\lambda,q}~\mathbb{E}\left[  e^{pV_{T}}\left\vert \eta\right\vert
^{p}+\Big(%
{\displaystyle\int_{0}^{T}}
e^{V_{r}}\left\vert H\left(  r,0,0\right)  \right\vert dQ_{r}\Big)^{p}\right]
.
\end{array}
\]

\end{theorem}

\begin{proof}
Let $t\in\left[  0,T\right]  $ and%
\[
\beta_{t}=t+A_{t}+\left\vert \mu_{t}\right\vert +\left\vert \nu_{t}\right\vert
+\ell_{t}+V_{t}^{\left(  +\right)  }+\left\vert F\left(  t,0,0\right)
\right\vert +\left\vert G\left(  t,0\right)  \right\vert +\Theta_{t},
\]
Define, for $n\in\mathbb{N}^{\ast}$,%
\begin{align*}
\eta^{\left(  n\right)  }  &  =\eta\,\mathbf{1}_{\left[  0,n\right]  }\left(
\left\vert \eta\right\vert +V_{T}^{\left(  +\right)  }\right)  ,\\
F^{\left(  n\right)  }\left(  t,y,z\right)   &  =F\left(  t,y,z\right)
-F\left(  t,0,0\right)  ~\mathbf{1}_{\left(  n,\infty\right)  }\left(
\beta_{t}\right)  ,\\
G^{\left(  n\right)  }\left(  t,y\right)   &  =G\left(  t,y\right)  -G\left(
t,0\right)  ~\mathbf{1}_{\left(  n,\infty\right)  }\left(  \beta_{t}\right)
,\\
H^{\left(  n\right)  }\left(  t,y,z\right)   &  =\alpha_{t}F^{\left(
n\right)  }\left(  t,y,z\right)  +\left(  1-\alpha_{t}\right)  G^{\left(
n\right)  }\left(  t,y\right)  ~.
\end{align*}
We highlight the following the following properties of the function
$H^{\left(  n\right)  }$
\begin{equation}%
\begin{array}
[c]{rl}%
\left(  j\right)  \; & \left\langle y^{\prime}-y,H^{\left(  n\right)
}(t,y^{\prime},z)-H^{\left(  n\right)  }(t,y,z)\right\rangle \leq\left[
\mu_{t}\alpha_{t}+\nu_{t}\left(  1-\alpha_{t}\right)  \right]  \left\vert
y^{\prime}-y\right\vert ^{2},\medskip\\
\left(  jj\right)  \; & \left\vert H^{\left(  n\right)  }(t,y,z^{\prime
})-H^{\left(  n\right)  }(t,y,z)\right\vert \leq\alpha_{t}\ell_{t}\left\vert
z^{\prime}-z\right\vert ,\medskip\\
\left(  jjj\right)  \; & \left\vert H^{\left(  n+i\right)  }(t,y,z)-H^{\left(
n\right)  }(t,y,z)\right\vert \leq\left[  \alpha_{t}\left\vert F\left(
t,0,0\right)  \right\vert +\left(  1-\alpha_{t}\right)  \left\vert G\left(
t,0\right)  \right\vert \right]  \mathbf{1}_{\left(  n,\infty\right)  }\left(
\beta_{t}\right)  .
\end{array}
\label{t2-hn}%
\end{equation}
and the monotonicity properties%
\begin{align*}
\left\langle y,H^{\left(  n\right)  }\left(  t,y,z\right)  \right\rangle  &
\leq|y|\left[  \alpha_{t}\left\vert F\left(  t,0,0\right)  \right\vert
+\left(  1-\alpha_{t}\right)  \left\vert G\left(  t,0\right)  \right\vert
\right]  \mathbf{1}_{\left[  0,n\right]  }\left(  \beta_{t}\right)
+|y|^{2}dV_{s}+\alpha_{t}\dfrac{n_{p}\lambda}{2}\,\left\vert z\right\vert
^{2}\\
&  \leq|y|\left[  \alpha_{t}\left\vert F\left(  t,0,0\right)  \right\vert
+\left(  1-\alpha_{t}\right)  \left\vert G\left(  t,0\right)  \right\vert
\right]  \mathbf{1}_{\left[  0,n\right]  }\left(  \beta_{t}\right)
+|y|^{2}dV_{s}^{\left(  +\right)  }+\alpha_{t}\dfrac{n_{p}\lambda}%
{2}\,\left\vert z\right\vert ^{2}%
\end{align*}
and%
\begin{align*}
\left\langle Y_{t}^{\prime}-Y_{t},H^{\left(  n\right)  }(t,Y_{t}^{\prime
},Z_{t}^{\prime})-H^{\left(  n\right)  }(t,Y_{t},Z_{t})\right\rangle dQ_{t}
&  \leq\left\vert Y_{t}^{\prime}-Y_{t}\right\vert ^{2}dV_{t}+\dfrac
{n_{p}\lambda}{2}\,\left\vert Z_{t}^{\prime}-Z_{t}\right\vert ^{2}dt\\
&  \leq\left\vert Y_{t}^{\prime}-Y_{t}\right\vert ^{2}dV_{t}^{\left(
+\right)  }+\dfrac{n_{p}\lambda}{2}\,\left\vert Z_{t}^{\prime}-Z_{t}%
\right\vert ^{2}dt.
\end{align*}
Clearly, the assumptions of Proposition \ref{p1-wvs} are satisfied for the
approximating BSDE%
\begin{equation}
\left\{
\begin{array}
[c]{l}%
Y_{t}^{\left(  n\right)  }+%
{\displaystyle\int_{t}^{T}}
dK_{s}=\eta^{\left(  n\right)  }+%
{\displaystyle\int_{t}^{T}}
H^{\left(  n\right)  }\big(s,Y_{s}^{\left(  n\right)  },Z_{s}^{\left(
n\right)  }\big)dQ_{s}-%
{\displaystyle\int_{t}^{T}}
Z_{s}^{\left(  n\right)  }dB_{s}\,,\;t\in\left[  0,T\right]  ,\medskip\\
dK_{s}^{\left(  n\right)  }\in\partial_{y}\Psi\big(r,Y_{r}^{\left(  n\right)
}\big)dQ_{r}=\alpha_{r}\partial\varphi\big(Y_{r}^{\left(  n\right)
}\big)dr+\left(  1-\alpha_{r}\right)  \partial\psi\big(Y_{r}^{\left(
n\right)  }\big)dA_{r}%
\end{array}
\right.  \label{t2-ae}%
\end{equation}
(with $\eta:=\eta^{\left(  n\right)  }$, $F:=F^{\left(  n\right)  }$,
$G:=G^{\left(  n\right)  }$, $H:=H^{\left(  n\right)  }\,$).

Hence by Proposition \ref{p1-wvs} the approximating BSDE (\ref{t2-ae}) has a
unique $L^{p}-$variational solution $\left(  Y^{\left(  n\right)  },Z^{\left(
n\right)  }\right)  .$ Therefore
\[
\mathbb{E}\left(  \sup\limits_{r\in\left[  0,T\right]  }{e^{pV_{r}}}%
\big|Y_{r}^{\left(  n\right)  }\big|^{p}\right)  +\mathbb{E~}\left(
{\displaystyle\int_{0}^{T}}
{e^{2V_{r}}}\big|Z_{r}^{\left(  n\right)  }\big|^{2}dr\right)  ^{p/2}%
+\mathbb{E~}\left(
{\displaystyle\int_{0}^{T}}
e^{2V_{r}}{\Psi}\left(  r,Y_{r}^{\left(  n\right)  }\right)  dQ_{r}\right)
^{p/2}<\infty
\]
and for $q\in\{2,p\wedge2\},$ $\delta_{q}=\delta\mathbf{1}_{[1,2)}\left(
q\right)  $ and $\Gamma_{t}^{\left(  n\right)  }=\left(  \big|M_{t}%
-Y_{t}^{\left(  n\right)  }\big|^{2}+\delta_{q}\right)  ^{1/2}$ it holds%
\begin{equation}%
\begin{array}
[c]{l}%
\big(\Gamma_{t}^{\left(  n\right)  }\big)^{q}+\dfrac{q\left(  q-1\right)  }{2}%
{\displaystyle\int_{t}^{s}}
\big({\Gamma_{r}^{\left(  n\right)  }\big)^{q-2}}{\Large \,}\big|R_{r}%
-Z_{r}^{\left(  n\right)  }\big|^{2}dr+{q%
{\displaystyle\int_{t}^{s}}
}\big({\Gamma_{r}^{\left(  n\right)  }\big)^{q-2}\Psi\big(}r,Y_{r}^{\left(
n\right)  }\big)dQ_{r}\medskip\\
\leq\big(\Gamma_{s}^{\left(  n\right)  }\big)^{q}+{q%
{\displaystyle\int_{t}^{s}}
}\big({\Gamma_{r}^{\left(  n\right)  }\big)^{q-2}\Psi}\left(  r,M_{r}\right)
dQ_{r}\medskip\\
\quad+q%
{\displaystyle\int_{t}^{s}}
{\big(\Gamma_{r}^{\left(  n\right)  }\big)^{q-2}}\langle M_{r}-Y_{r}^{\left(
n\right)  },N_{r}-H^{\left(  n\right)  }{\big(}r,Y_{r}^{\left(  n\right)
},Z_{r}^{\left(  n\right)  }\big)\rangle dQ_{r}\medskip\\
\quad-q%
{\displaystyle\int_{t}^{s}}
\big({\Gamma_{r}^{\left(  n\right)  }\big)^{q-2}}\,\langle M_{r}%
-Y_{r}^{\left(  n\right)  },{\big(}R_{r}-Z_{r}^{\left(  n\right)  }%
\big)dB_{r}\rangle
\end{array}
\label{vw-ap2}%
\end{equation}
for all $\delta\in(0,1]$, for all $0\leq t\leq s\leq T,$ for all $M\in
S_{m}^{0}\left(  \gamma,N,R;V\right)  .$

Since $\mathbb{E}\left(  \sup\nolimits_{r\in\left[  0,T\right]  }{e^{pV_{r}}%
}\big|Y_{r}^{\left(  n\right)  }\big|^{p}\right)  <\infty$ and inequality
(\ref{vw-ap2}) holds for $1<q=p\wedge2\leq p,$ inequalities (\ref{def-11ccc})
and (\ref{def-11aa}) yield%
\begin{equation}%
\begin{array}
[c]{l}%
\mathbb{E}\Big(\sup\limits_{t\in\left[  0,T\right]  }e^{pV_{t}}\big|Y_{t}%
^{\left(  n\right)  }\big|^{p}\Big)+\mathbb{E~}\left(
{\displaystyle\int_{0}^{T}}
{e^{2V_{r}}}\big|Z_{r}^{\left(  n\right)  }\big|^{2}dr\right)  ^{p/2}%
++\mathbb{E~}\left(
{\displaystyle\int_{0}^{T}}
e^{2V_{r}}{\Psi}\left(  r,Y_{r}^{\left(  n\right)  }\right)  dQ_{r}\right)
^{p/2}\medskip\\
\quad+\mathbb{E}\left(
{\displaystyle\int_{0}^{T}}
{e^{qV_{r}}\big|Y_{r}^{\left(  n\right)  }\big|^{q-2}}\big|Z_{r}^{\left(
n\right)  }\big|^{2}dr\right)  ^{p/q}+\mathbb{E}~\left(
{\displaystyle\int_{0}^{T}}
{e^{qV_{r}}\big|Y_{r}^{\left(  n\right)  }\big|^{q-2}\Psi}\left(
r,Y_{r}^{\left(  n\right)  }\right)  dQ_{r}\right)  ^{p/q}\medskip\\
\leq C_{p,\lambda,q}~\mathbb{E}\left[  e^{pV_{T}}\left\vert \eta\right\vert
^{p}+\Big(%
{\displaystyle\int_{0}^{T}}
e^{V_{r}}\left\vert H\left(  r,0,0\right)  \right\vert dQ_{r}\Big)^{p}\right]
.
\end{array}
\label{vw-ap3}%
\end{equation}
From (\ref{cont-2}) for $q=p\wedge2$ we have for all $0<\alpha<1$%
\begin{equation}%
\begin{array}
[c]{l}%
\mathbb{E}\sup\limits_{t\in\left[  0,T\right]  }e^{\alpha qV_{t}}%
\big|Y_{t}^{\left(  n+i\right)  }-Y_{t}^{\left(  n\right)  }\big|^{\alpha
q}+\left(  \mathbb{E}%
{\displaystyle\int_{0}^{T}}
e^{2V_{r}}\dfrac{\big|Z_{r}^{\left(  n+i\right)  }-Z_{r}^{\left(  n\right)
}\big|^{2}}{\left(  e^{V_{r}}\big|Y_{t}^{\left(  n+i\right)  }-Y_{t}^{\left(
n\right)  }\big|+1\right)  ^{2-q}}dr\right)  ^{\alpha}\medskip\\
\leq C_{\alpha,q,\lambda}\bigg[\,\mathbb{E~}e^{qV_{T}}\left\vert \eta^{\left(
n+i\right)  }-\eta^{\left(  n\right)  }\right\vert ^{q}\medskip\\
\quad+K\,\left(  \mathbb{E}\left(
{\displaystyle\int_{0}^{T}}
e^{V_{r}}\big|H^{\left(  n+i\right)  }(t,Y_{t}^{\left(  n\right)  }%
,Z_{t}^{\left(  n\right)  })-H^{\left(  n\right)  }(t,Y_{t}^{\left(  n\right)
},Z_{t}^{\left(  n\right)  })\big|dQ_{r}\right)  ^{q}\right)  ^{1/q}%
\bigg]^{\alpha},
\end{array}
\label{Cauchy-3}%
\end{equation}

where%
\[%
\begin{array}
[c]{l}%
\displaystyle K=\bigg[{\mathbb{E~}}\left(  e^{qV_{T}}\big|\eta^{\left(
n+i\right)  }\big|^{q}+\Big(%
{\displaystyle\int_{0}^{T}}
e^{V_{r}}\big|H^{\left(  n+i\right)  }\left(  r,0,0\right)  \big|dQ_{r}%
\Big)^{q}\right)  \medskip\\
\displaystyle\quad+{\mathbb{E~}}\left(  e^{qV_{T}}\big|\eta^{\left(  n\right)
}\big|^{q}+\Big(%
{\displaystyle\int_{0}^{T}}
e^{V_{r}}\big|H^{\left(  n\right)  }\left(  r,0,0\right)  \big|dQ_{r}%
\Big)^{q}\right)  \bigg]^{\left(  q-1\right)  /q}\medskip\\
\displaystyle\leq2^{\left(  q-1\right)  /q}\bigg[{\mathbb{E}}~\left(
e^{qV_{T}}\left\vert \eta\right\vert ^{q}+~\Big(%
{\displaystyle\int_{0}^{T}}
e^{V_{r}}\left\vert F\left(  r,0,0\right)  \right\vert dr+\left\vert G\left(
r,0\right)  \right\vert dA_{r}\Big)^{q}\right)  \bigg]^{\left(  q-1\right)
/q}%
\end{array}
\]
and $C_{\alpha,q,\lambda}$ is a positive constant depending only $\alpha,q$
and $\lambda.$

First we remark%
\[
\mathbb{E}~e^{qV_{T}}\big|\eta^{\left(  n+i\right)  }-\eta^{\left(  n\right)
}\big|^{q}\leq\mathbb{E}~e^{qV_{T}}\left\vert \eta\right\vert ^{q}%
\mathbf{1}_{\left(  n,\infty\right)  }\left(  \left\vert \eta\right\vert
+V_{T}^{\left(  +\right)  }\right)  \longrightarrow0,\quad\mathbb{P}%
-\text{a.s.,}\;\text{for }n\rightarrow\infty,
\]
since by $1<q\leq p$ and assumption (\ref{ip-t2-2}) we have%
\[
\mathbb{E}~e^{qV_{T}}\left\vert \eta\right\vert ^{q}\leq\left(  \mathbb{E}%
~e^{pV_{T}}\left\vert \eta\right\vert ^{p}\right)  ^{q/p}<\infty.
\]
Secondly, we remark that under assumption (\ref{ip-t2-2})%
\[%
\begin{array}
[c]{l}%
\displaystyle\mathbb{E}\left(
{\displaystyle\int_{0}^{T}}
e^{V_{r}}\big|H^{\left(  n+i\right)  }(t,Y_{t}^{\left(  n\right)  }%
,Z_{t}^{\left(  n\right)  })-H^{\left(  n\right)  }(t,Y_{t}^{\left(  n\right)
},Z_{t}^{\left(  n\right)  })\big|dQ_{r}\right)  ^{q}\medskip\\
\displaystyle\leq\mathbb{E}\left(
{\displaystyle\int_{0}^{T}}
e^{V_{r}}\left[  \left\vert F\left(  r,0,0\right)  \right\vert \mathbf{1}%
_{\left(  n,\infty\right)  }\left(  \beta_{r}\right)  dr+\left\vert G\left(
r,0\right)  \right\vert \mathbf{1}_{\left(  n,\infty\right)  }\left(
\beta_{r}\right)  dA_{r}\right]  \right)  ^{q}\medskip\\
\displaystyle\leq2^{q-1}\,\Big[\mathbb{E}\Big(%
{\displaystyle\int_{0}^{T}}
e^{V_{r}}\left\vert F\left(  r,0,0\right)  \right\vert \mathbf{1}_{\left(
n,\infty\right)  }\left(  \beta_{r}\right)  dr\Big)^{q}+\mathbb{E}\Big(%
{\displaystyle\int_{0}^{T}}
e^{V_{r}}\left\vert G\left(  r,0\right)  \right\vert \mathbf{1}_{\left(
n,\infty\right)  }\left(  \beta_{r}\right)  dA_{r}\Big)^{q}\Big]\medskip\\
\longrightarrow0,\quad\text{a.s.}\quad\text{for }n\rightarrow\infty.
\end{array}
\]
By (\ref{Cauchy-3}) we conclude that there exists $\left(  Y,Z\right)  \in
S_{m}^{0}\times\Lambda_{m\times k}^{0}$ such that (on a subsequence)%
\[
\sup_{t\in\left[  0,T\right]  }\big|Y_{t}^{\left(  n\right)  }-Y_{t}\big|+%
{\displaystyle\int_{0}^{T}}
\big|Z_{r}^{\left(  n\right)  }-Z_{r}\big|^{2}dr\longrightarrow0,\quad
\mathbb{P}-\text{a.s.,}\;\text{for }n\rightarrow\infty.
\]
Passing to $\liminf_{n\rightarrow+\infty}$ in (\ref{vw-ap3}) and
(\ref{vw-ap2}) we infer that $\left(  Y,Z\right)  $ is an $L^{p}$--variational solution.

\hfill
\end{proof}

\section{Appendix}

In this section we recall from \cite{pa-ra/14} some results frequently used in
our paper. These results concern inequalities for backward stochastic
differential equations and are interesting by themselves. For more details the
interested readers are referred to the monograph of Pardoux and
R\u{a}\c{s}canu \cite{pa-ra/14}.

Let $\left\{  B_{t}:t\geq0\right\}  $ be a $k$--dimensional Brownian motion
with respect to a given stochastic basis $\left(  \Omega,\mathcal{F}%
,\mathbb{P},\{\mathcal{F}_{t}\}_{t\geq0}\right)  $, where $\left(
\mathcal{F}_{t}\right)  _{t\geq0}$ is the natural filtration associated to
$\left\{  B_{t}:t\geq0\right\}  .$

\begin{notation}
If $p\geq1$ we denote $n_{p}:=1\wedge\left(  p-1\right)  $.
\end{notation}

\subsection{Backward stochastic inequalities}

Based on \cite[Proposition 6.80]{pa-ra/14} and its proof we adapt here the
Pardoux--R\u{a}\c{s}canu's inequalities (6.92) and (6.94) from \cite{pa-ra/14}
to the case of BSVI.

\begin{proposition}
\label{an-prop-dz}Let $\left(  Y,Z\right)  \in S_{m}^{0}\times\Lambda_{m\times
k}^{0}$ and $a\geq0,$ $\gamma\in\mathbb{R}$ such that, for all $0\leq t\leq
s<\infty,$%
\[
{%
{\displaystyle\int_{t}^{s}}
}\left\vert Z_{r}\right\vert ^{2}dr+{%
{\displaystyle\int_{t}^{s}}
}dD_{r}\leq a\left\vert Y_{s}\right\vert ^{2}+a{%
{\displaystyle\int_{t}^{s}}
}\left(  dR_{r}+\left\vert Y_{r}\right\vert dN_{r}\right)  +\gamma%
{\displaystyle\int_{t}^{s}}
\,\langle Y_{r},Z_{r}dB_{r}\rangle,\;\mathbb{P}\text{-a.s.,}%
\]
where $R,N$ and $D$ are increasing and continuous p.m.s.p. $R_{0}=N_{0}%
=D_{0}=0.$ Then for all $q>0$ and for all stopping times $0\leq\sigma
\leq\theta<\infty,$ the following inequality hold:%
\begin{equation}%
\begin{array}
[c]{l}%
\mathbb{E}^{\mathcal{F}_{\sigma}}\Big(%
{\displaystyle\int_{\sigma}^{\theta}}
\left\vert Z_{r}\right\vert ^{2}dr\Big)^{q/2}+\mathbb{E}^{\mathcal{F}_{\sigma
}}\Big(%
{\displaystyle\int_{\sigma}^{\theta}}
dD_{r}\Big)^{q/2}\medskip\\
\leq C_{a,\gamma,q}~\left[  \mathbb{E}^{\mathcal{F}_{\sigma}}\sup
\nolimits_{r\in\left[  \sigma,\theta\right]  }\left\vert Y_{r}\right\vert
^{q}+\Big(%
{\displaystyle\int_{\sigma}^{\theta}}
dR_{r}\Big)^{q/2}+\Big(%
{\displaystyle\int_{\sigma}^{\theta}}
\left\vert Y_{r}\right\vert dN_{r}\Big)^{q/2}\right]  \medskip\\
\leq2C_{a,\gamma,q}\mathbb{E}^{\mathcal{F}_{\sigma}}\left[  \sup
\nolimits_{r\in\left[  \sigma,\theta\right]  }\left\vert Y_{r}\right\vert
^{q}+\Big(%
{\displaystyle\int_{\sigma}^{\theta}}
dR_{r}\Big)^{q/2}+\Big(%
{\displaystyle\int_{\sigma}^{\theta}}
dN_{r}\Big)^{q}\right]  ,\quad\mathbb{P}\text{-a.s.,}%
\end{array}
\label{an4}%
\end{equation}
where $C_{a,\gamma,q}$ is a positive constant depending only on $a,\gamma$ and
$q.$
\end{proposition}

\begin{proof}
We follow the first part of the proof of \cite[Proposition 6.80]{pa-ra/14}.
Let the sequence of stopping times%
\begin{equation}
\theta_{n}=\theta\wedge\inf\left\{  s\geq\sigma:\sup\nolimits_{r\in\left[
\sigma,\sigma\vee s\right]  }\left\vert Y_{r}-Y_{\sigma}\right\vert +%
{\displaystyle\int_{\sigma}^{\sigma\vee s}}
\left\vert Z_{r}\right\vert ^{2}dr+%
{\displaystyle\int_{\sigma}^{\sigma\vee s}}
d\left(  D_{r}+R_{r}+N_{r}\right)  \geq n\right\}  . \label{an4-a}%
\end{equation}
We have for $q>0$%
\begin{equation}%
\begin{array}
[c]{l}%
\mathbb{E}^{\mathcal{F}_{\sigma}}\Big(%
{\displaystyle\int_{\sigma}^{\theta_{n}}}
\left\vert Z_{r}\right\vert ^{2}dr\Big)^{q/2}+\mathbb{E}^{\mathcal{F}_{\sigma
}}\Big(%
{\displaystyle\int_{\sigma}^{\theta_{n}}}
dD_{r}\Big)^{q/2}\medskip\\
\leq2\mathbb{E}^{\mathcal{F}_{\sigma}}\Big(%
{\displaystyle\int_{\sigma}^{\theta_{n}}}
\left\vert Z_{r}\right\vert ^{2}dr+%
{\displaystyle\int_{\sigma}^{\theta_{n}}}
dD_{r}\Big)^{q/2}\medskip\\
\leq C_{a,\gamma,q}^{\prime}~\mathbb{E}^{\mathcal{F}_{\sigma}}\left[
\left\vert Y_{\theta_{n}}\right\vert ^{q}+\Big(%
{\displaystyle\int_{\sigma}^{\theta_{n}}}
dR_{r}\Big)^{q/2}+\Big(%
{\displaystyle\int_{\sigma}^{\theta_{n}}}
\left\vert Y_{r}\right\vert dN_{r}\Big)^{q/2}+\Big|%
{\displaystyle\int_{\sigma}^{\theta_{n}}}
\,\langle Y_{r},Z_{r}dB_{r}\rangle\Big|^{q/2}\right]  .
\end{array}
\label{an5}%
\end{equation}
By Burkholder--Davis--Gundy inequality we get%
\begin{align*}
C_{a,\gamma,q}^{\prime}~\mathbb{E}^{\mathcal{F}_{\sigma}}\Big|%
{\displaystyle\int_{\sigma}^{\theta_{n}}}
\,\langle Y_{r},Z_{r}dB_{r}\rangle\Big|^{q/2}  &  \leq C_{a,\gamma,q}%
^{\prime\prime}~\mathbb{E}^{\mathcal{F}_{\sigma}}\Big(%
{\displaystyle\int_{\sigma}^{\theta_{n}}}
\,\left\vert Y_{r}\right\vert ^{2}\left\vert Z_{r}\right\vert ^{2}%
dr\Big)^{q/4}\\
&  \leq C_{a,\gamma,q}^{\prime\prime}~\mathbb{E}^{\mathcal{F}_{\sigma}}%
~\sup_{r\in\left[  \sigma,\theta_{n}\right]  }\left\vert Y_{r}\right\vert
^{q/2}\Big(%
{\displaystyle\int_{\sigma}^{\theta_{n}}}
\,\left\vert Z_{r}\right\vert ^{2}dr\Big)^{q/4}\\
&  \leq\frac{1}{2}\left(  C_{a,\gamma,q}^{\prime\prime}\right)  ^{2}%
~\mathbb{E}^{\mathcal{F}_{\sigma}}~\sup_{r\in\left[  \sigma,\theta_{n}\right]
}\left\vert Y_{r}\right\vert ^{q}+\frac{1}{2}\mathbb{E}^{\mathcal{F}_{\sigma}%
}\Big(%
{\displaystyle\int_{\sigma}^{\theta_{n}}}
\left\vert Z_{r}\right\vert ^{2}dr\Big)^{q/2}%
\end{align*}
and consequently from (\ref{an5}) the following inequality holds%
\begin{equation}%
\begin{array}
[c]{l}%
\mathbb{E}^{\mathcal{F}_{\sigma}}\Big(%
{\displaystyle\int_{\sigma}^{\theta_{n}}}
\left\vert Z_{r}\right\vert ^{2}dr\Big)^{q/2}+\mathbb{E}^{\mathcal{F}_{\sigma
}}\Big(%
{\displaystyle\int_{\sigma}^{\theta_{n}}}
dD_{r}\Big)^{q/2}\medskip\\
\leq C_{a,\gamma,q}~\left[  \mathbb{E}^{\mathcal{F}_{\sigma}}~\sup
\limits_{r\in\left[  \sigma,\theta\right]  }\left\vert Y_{r}\right\vert
^{q}+\Big(%
{\displaystyle\int_{\sigma}^{\theta}}
dR_{r}\Big)^{q/2}+\Big(%
{\displaystyle\int_{\sigma}^{\theta}}
\left\vert Y_{r}\right\vert dN_{r}\Big)^{q/2}\right]
\end{array}
\label{an5-aa}%
\end{equation}
Since%
\[
\Big(%
{\displaystyle\int_{\sigma}^{\theta}}
\left\vert Y_{r}\right\vert dN_{r}\Big)^{q/2}\leq\sup_{r\in\left[
\sigma,\theta\right]  }\left\vert Y_{r}\right\vert ^{q}+\Big(%
{\displaystyle\int_{\sigma}^{\theta}}
dN_{r}\Big)^{q}%
\]
then from (\ref{an5-aa}) we infer%
\begin{equation}%
\begin{array}
[c]{l}%
\mathbb{E}^{\mathcal{F}_{\sigma}}\Big(%
{\displaystyle\int_{\sigma}^{\theta_{n}}}
\left\vert Z_{r}\right\vert ^{2}dr\Big)^{q/2}+\mathbb{E}^{\mathcal{F}_{\sigma
}}\Big(%
{\displaystyle\int_{\sigma}^{\theta_{n}}}
dD_{r}\Big)^{q/2}\medskip\\
\leq C_{a,\gamma,q}~\left[  \mathbb{E}^{\mathcal{F}_{\sigma}}~\sup
\limits_{r\in\left[  \sigma,\theta\right]  }\left\vert Y_{r}\right\vert
^{q}+\Big(%
{\displaystyle\int_{\sigma}^{\theta}}
dR_{r}\Big)^{q/2}+\Big(%
{\displaystyle\int_{\sigma}^{\theta}}
dN_{r}\Big)^{q}\right]
\end{array}
\label{an5-ab}%
\end{equation}
Consequently by Fatou's Lemma, as $n\rightarrow\infty,$ inequality (\ref{an4})
follows.\hfill
\end{proof}

\begin{proposition}
\label{an-prop-ydz}Let $\left(  Y,Z\right)  \in S_{m}^{0}\times\Lambda
_{m\times k}^{0}$ , $a\geq0,$ $\gamma\in\mathbb{R}$ and $1<q\leq p$ satisfying
for all $0\leq t\leq s<\infty:$%
\[%
\begin{array}
[c]{l}%
\displaystyle\left\vert Y_{t}\right\vert ^{q}+%
{\displaystyle\int_{t}^{s}}
\left\vert Y_{r}\right\vert ^{q-2}\mathbf{1}_{Y_{r}\neq0}{\Large \,}\left\vert
Z_{r}\right\vert ^{2}dr+{%
{\displaystyle\int_{t}^{s}}
\left\vert Y_{r}\right\vert ^{q-2}\mathbf{1}_{Y_{r}\neq0}dD_{r}}\medskip\\
\displaystyle\leq a\left\vert Y_{s}\right\vert ^{q}+a%
{\displaystyle\int_{t}^{s}}
\left[  \left\vert Y_{r}\right\vert ^{q-2}\mathbf{1}_{Y_{r}\neq0}%
\mathbf{1}_{q\geq2}dR_{r}+\left\vert Y_{r}\right\vert ^{q-1}dN_{r}\right]
+\gamma%
{\displaystyle\int_{t}^{s}}
\,\langle\left\vert Y_{r}\right\vert ^{q-2}Y_{r},Z_{r}dB_{r}\rangle
,\quad\mathbb{P}-a.s.,
\end{array}
\]
where $R,N$ and $D$ are increasing and continuous p.m.s.p. $R_{0}=N_{0}%
=D_{0}=0.$ If $\sigma$ and $\theta$ are two stopping times such that
$0\leq\sigma\leq\theta<\infty$ and%
\[
\mathbb{E}\sup\limits_{r\in\left[  \sigma,\theta\right]  }\left\vert
Y_{r}\right\vert ^{p}<\infty
\]
then, $\mathbb{P}$--a.s.%
\begin{equation}
\mathbb{E}^{\mathcal{F}_{\sigma}}\sup\limits_{r\in\left[  \sigma
,\theta\right]  }\left\vert Y_{r}\right\vert ^{p}\leq C_{p,q,a,\gamma
}~\mathbb{E}^{\mathcal{F}_{\sigma}}\left[  \left\vert Y_{\theta}\right\vert
^{p}+\Big(%
{\displaystyle\int_{\sigma}^{\theta}}
\left\vert Y_{r}\right\vert ^{q-2}\mathbf{1}_{Y_{r}\neq0}\mathbf{1}_{q\geq
2}dR_{r}\Big)^{p/q}+\Big(%
{\displaystyle\int_{\sigma}^{\theta}}
\left\vert Y_{r}\right\vert ^{q-1}dN_{r}\Big)^{p/q}\right]  \label{an5-a}%
\end{equation}
and%
\begin{equation}%
\begin{array}
[c]{l}%
\mathbb{E}^{\mathcal{F}_{\sigma}}\Big(\sup\limits_{r\in\left[  \sigma
,\theta\right]  }\left\vert Y_{r}\right\vert ^{p}\Big)+\mathbb{E}%
^{\mathcal{F}_{\sigma}}~\Big(%
{\displaystyle\int_{\sigma}^{\theta}}
\left\vert Y_{r}\right\vert ^{q-2}{\Large \,}\mathbf{1}_{Y_{r}\neq0}\left\vert
Z_{r}\right\vert ^{2}dr\Big)^{p/q}+\mathbb{E}^{\mathcal{F}_{\sigma}}~\Big(%
{\displaystyle\int_{\sigma}^{\theta}}
{\left\vert Y_{r}\right\vert ^{q-2}\mathbf{1}_{Y_{r}\neq0}dD}_{r}%
\Big)^{p/q}\medskip\\
\leq C_{p,q,a,\gamma}~\mathbb{E}^{\mathcal{F}_{\sigma}}\left[  \left\vert
Y_{\theta}\right\vert ^{p}+\Big(%
{\displaystyle\int_{\sigma}^{\theta}}
\mathbf{1}_{q\geq2}dR_{r}\Big)^{p/2}+\Big(%
{\displaystyle\int_{\sigma}^{\theta}}
dN_{r}\Big)^{p}\right]  .
\end{array}
\label{an5-b}%
\end{equation}
with $C_{p,q,a,\gamma}$ a positive constant depending only $\left(
p,q,a,\gamma\right)  .$
\end{proposition}

\begin{proof}
We follow the proof of \cite[Proposition 6.80]{pa-ra/14}. Let the stopping
time $\theta_{n}$ be defined by
\[
\theta_{n}=\theta\wedge\inf\left\{  s\geq\sigma:\sup_{r\in\left[
\sigma,\sigma\vee s\right]  }\left\vert Y_{r}-Y_{\sigma}\right\vert +%
{\displaystyle\int_{\sigma}^{\sigma\vee s}}
\left\vert Z_{r}\right\vert ^{2}dr+%
{\displaystyle\int_{\sigma}^{\sigma\vee s}}
d\left(  D_{r}+R_{r}+N_{r}\right)  \geq n\right\}
\]
For any stopping time $\tau\in\left[  \sigma,\theta_{n}\right]  $ we have%
\begin{equation}%
\begin{array}
[c]{l}%
\left\vert Y_{\tau}\right\vert ^{q}+%
{\displaystyle\int_{\tau}^{\theta_{n}}}
\left\vert Y_{r}\right\vert ^{q-2}{\Large \,}\mathbf{1}_{Y_{r}\neq0}\left\vert
Z_{r}\right\vert ^{2}dr+%
{\displaystyle\int_{\tau}^{\theta_{n}}}
{\left\vert Y_{r}\right\vert ^{q-2}\mathbf{1}_{Y_{r}\neq0}dD}_{r}\medskip\\
\leq a\left\vert Y_{\theta_{n}}\right\vert ^{q}+a%
{\displaystyle\int_{\tau}^{\theta_{n}}}
\left(  \left\vert Y_{r}\right\vert ^{q-2}\mathbf{1}_{Y_{r}\neq0}%
\mathbf{1}_{q\geq2}dR_{r}+\left\vert Y_{r}\right\vert ^{q-1}dN_{r}\right)
+\gamma%
{\displaystyle\int_{\tau}^{\theta_{n}}}
\,\langle\left\vert Y_{r}\right\vert ^{q-2}Y_{r},Z_{r}dB_{r}\rangle.
\end{array}
\label{an6}%
\end{equation}
Remark that
\[
M_{s}=\int_{0}^{s}\mathbf{1}_{\left[  \sigma,\theta_{n}\right]  }\left(
r\right)  \langle\left\vert Y_{r}\right\vert ^{q-2}Y_{r},Z_{r}dB_{r}%
\rangle,\quad s\geq0
\]
is a martingale, since%
\[%
\begin{array}
[c]{l}%
\displaystyle\mathbb{E}\Big(%
{\displaystyle\int_{0}^{T}}
\mathbf{1}_{\left[  \sigma,\theta_{n}\right]  }\left(  r\right)  \left\vert
Y_{r}\right\vert ^{2q-2}\left\vert Z_{r}\right\vert ^{2}dr\Big)^{1/2}%
\leq\mathbb{E}\sup_{r\in\left[  \sigma,\theta_{n}\right]  }\left\vert
Y_{r}\right\vert ^{q-1}\Big(\int_{\sigma}^{\theta_{n}}\left\vert
Z_{r}\right\vert ^{2}dr\Big)^{1/2}\medskip\\
\displaystyle\leq\left[  \frac{q-1}{q}\mathbb{E}\sup_{r\in\left[
\sigma,\theta_{n}\right]  }\left\vert Y_{r}\right\vert ^{q}+\frac{1}%
{q}\mathbb{E}\Big(\int_{\sigma}^{\theta_{n}}\left\vert Z_{r}\right\vert
^{2}dr\Big)^{q/2}\right]  \medskip\\
\displaystyle\leq\frac{q-1}{q}~\mathbb{E}\left(  \left\vert Y_{\sigma
}\right\vert +n\right)  ^{q}+\frac{1}{q}n^{q/2}<\infty.
\end{array}
\]
Therefore from (\ref{an6})%
\begin{equation}%
\begin{array}
[c]{l}%
\mathbb{E}^{\mathcal{F}_{\tau}}\left[  \Big(%
{\displaystyle\int_{\tau}^{\theta_{n}}}
\left\vert Y_{r}\right\vert ^{q-2}{\Large \,}\left\vert Z_{r}\right\vert
^{2}dr\Big)^{p/q}+\Big(%
{\displaystyle\int_{\tau}^{\theta_{n}}}
{\left\vert Y_{r}\right\vert ^{q-2}dD}_{r}\Big)^{p/q}\right]  \medskip\\
\leq C_{p,q,a}~\mathbb{E}^{\mathcal{F}_{\tau}}\left[  \left\vert Y_{\theta
_{n}}\right\vert ^{p}+\Big(%
{\displaystyle\int_{\sigma}^{\theta_{n}}}
\left(  \left\vert Y_{r}\right\vert ^{q-2}\mathbf{1}_{Y_{r}\neq0}%
\mathbf{1}_{q\geq2}dR_{r}\right)  \Big)^{p/q}+\Big(%
{\displaystyle\int_{\sigma}^{\theta_{n}}}
\left\vert Y_{r}\right\vert ^{q-1}dN_{r}\Big)^{p/q}\right]
\end{array}
\label{an6-a}%
\end{equation}
and%
\begin{equation}%
\begin{array}
[c]{l}%
\mathbb{E}^{\mathcal{F}_{\sigma}}\sup\limits_{\tau\in\left[  \sigma,\theta
_{n}\right]  }\left\vert Y_{\tau}\right\vert ^{p}\leq C_{p,q,a,\gamma}%
^{\prime}~\bigg[\mathbb{E}^{\mathcal{F}_{\sigma}}\left\vert Y_{\theta_{n}%
}\right\vert ^{p}+~\mathbb{E}^{\mathcal{F}_{\sigma}}\Big(%
{\displaystyle\int_{\sigma}^{\theta_{n}}}
\left\vert Y_{r}\right\vert ^{q-2}\mathbf{1}_{Y_{r}\neq0}\mathbf{1}_{q\geq
2}dR_{r}\Big)^{p/q}\medskip\\
\quad+\mathbb{E}^{\mathcal{F}_{\sigma}}\Big(%
{\displaystyle\int_{\sigma}^{\theta_{n}}}
\left\vert Y_{r}\right\vert ^{q-1}dN_{r}\Big)^{p/q}+\mathbb{E}^{\mathcal{F}%
_{\sigma}}\sup\limits_{\tau\in\left[  \sigma,\theta_{n}\right]  }\left\vert
M_{\theta_{n}}-M_{\tau}\right\vert ^{p/q}\bigg]\medskip\\
\leq C_{p,q,a,\gamma}^{\prime\prime}~\bigg[\mathbb{E}^{\mathcal{F}_{\sigma}%
}\left\vert Y_{\theta_{n}}\right\vert ^{p}+~\mathbb{E}^{\mathcal{F}_{\sigma}%
}\Big(%
{\displaystyle\int_{\sigma}^{\theta_{n}}}
\left\vert Y_{r}\right\vert ^{q-2}\mathbf{1}_{Y_{r}\neq0}\mathbf{1}_{q\geq
2}dR_{r}\Big)^{p/q}\medskip\\
\quad\quad+~\mathbb{E}^{\mathcal{F}_{\sigma}}\Big(%
{\displaystyle\int_{\sigma}^{\theta_{n}}}
\left\vert Y_{r}\right\vert ^{q-1}dN_{r}\Big)^{p/q}+\mathbb{\mathbb{E}%
^{\mathcal{F}_{\sigma}}}\Big(%
{\displaystyle\int_{\sigma}^{\theta_{n}}}
\left\vert Y_{r}\right\vert ^{2q-2}\left\vert Z_{r}\right\vert ^{2}%
dr\Big)^{p/\left(  2q\right)  }\bigg].
\end{array}
\label{an7}%
\end{equation}

But%
\begin{align*}
&  C_{p,q,a,\gamma}^{\prime\prime}\mathbb{\mathbb{E}^{\mathcal{F}_{\sigma}}%
}\Big(%
{\displaystyle\int_{\sigma}^{\theta_{n}}}
\left\vert Y_{r}\right\vert ^{2q-2}\left\vert Z_{r}\right\vert ^{2}%
dr\Big)^{p/\left(  2q\right)  }\\
&  \leq C_{p,q,a,\gamma}^{\prime\prime}\mathbb{\mathbb{E}^{\mathcal{F}%
_{\sigma}}}\sup_{r\in\left[  \sigma,\theta_{n}\right]  }\left\vert
Y_{r}\right\vert ^{p/2}\Big(\int_{\sigma}^{\theta_{n}}\left\vert
Y_{r}\right\vert ^{q-2}\left\vert Z_{r}\right\vert ^{2}dr\Big)^{p/\left(
2q\right)  }\\
&  \leq\frac{1}{2}~\mathbb{\mathbb{E}^{\mathcal{F}_{\sigma}}}\sup_{r\in\left[
\sigma,\theta_{n}\right]  }\left\vert Y_{r}\right\vert ^{p}+\frac{\left(
C_{p,q,a,\gamma}^{\prime\prime}\right)  ^{2}}{2}\mathbb{\mathbb{E}%
^{\mathcal{F}_{\sigma}}}\Big(\int_{\sigma}^{\theta_{n}}\left\vert
Y_{r}\right\vert ^{q-2}\left\vert Z_{r}\right\vert ^{2}dr\Big)^{p/q}\\
&  \leq\frac{1}{2}~\mathbb{\mathbb{E}^{\mathcal{F}_{\sigma}}}\sup_{r\in\left[
\sigma,\theta_{n}\right]  }\left\vert Y_{r}\right\vert ^{p}\\
&  \quad+C_{p,q,a,\gamma}^{\prime\prime\prime}~\mathbb{E}^{\mathcal{F}_{\tau}%
}\left[  \left\vert Y_{\theta_{n}}\right\vert ^{p}+\Big(%
{\displaystyle\int_{\sigma}^{\theta_{n}}}
\left(  \left\vert Y_{r}\right\vert ^{q-2}\mathbf{1}_{Y_{r}\neq0}%
\mathbf{1}_{q\geq2}dR_{r}\right)  \Big)^{p/q}+\Big(%
{\displaystyle\int_{\sigma}^{\theta_{n}}}
\left\vert Y_{r}\right\vert ^{q-1}dN_{r}\Big)^{p/q}\right]
\end{align*}
Using this last inequality in (\ref{an7}) we obtain%
\begin{equation}
\mathbb{E}^{\mathcal{F}_{\sigma}}\sup\limits_{\tau\in\left[  \sigma,\theta
_{n}\right]  }\left\vert Y_{\tau}\right\vert ^{p}\leq C_{p,q,a,\gamma
}\mathbb{E}^{\mathcal{F}_{\sigma}}\left[  \left\vert Y_{\theta_{n}}\right\vert
^{p}+\Big(%
{\displaystyle\int_{\sigma}^{\theta_{n}}}
\left\vert Y_{r}\right\vert ^{q-2}\mathbf{1}_{Y_{r}\neq0}\mathbf{1}_{q\geq
2}dR_{r}\Big)^{p/q}+\Big(%
{\displaystyle\int_{\sigma}^{\theta_{n}}}
\left\vert Y_{r}\right\vert ^{q-1}dN_{r}\Big)^{p/q}\right]  \label{an8}%
\end{equation}
Now by H\"{o}lder's inequality%
\begin{align*}
&  C_{p,q,a,\gamma}\mathbb{E}^{\mathcal{F}_{\sigma}}\left[  \Big(%
{\displaystyle\int_{\sigma}^{\theta_{n}}}
\left\vert Y_{r}\right\vert ^{q-2}\mathbf{1}_{Y_{r}\neq0}\mathbf{1}_{q\geq
2}dR_{r}\Big)^{p/q}+\Big(%
{\displaystyle\int_{\sigma}^{\theta_{n}}}
\left\vert Y_{r}\right\vert ^{q-1}dN_{r}\Big)^{p/q}\right] \\
&  \leq C_{a,\gamma}~\mathbb{E}^{\mathcal{F}_{\sigma}}\left[  \Big(\sup
\nolimits_{r\in\left[  \sigma,\theta_{n}\right]  }\left(  \left\vert
Y_{r}\mathbf{1}_{Y_{r}\neq0}\right\vert ^{q-2}\mathbf{1}_{q\geq2}\right)
{\displaystyle\int_{\sigma}^{\theta_{n}}}
\mathbf{1}_{q\geq2}dR_{r}\Big)^{p/q}+\Big(\sup_{r\in\left[  \sigma,\theta
_{n}\right]  }\left\vert Y_{r}\right\vert ^{q-1}%
{\displaystyle\int_{\sigma}^{\theta_{n}}}
dN_{r}\Big)^{p/q}\right] \\
&  \leq\frac{1}{2}\mathbb{E}^{\mathcal{F}_{\sigma}}\sup\limits_{\tau\in\left[
\sigma,\theta_{n}\right]  }\left\vert Y_{\tau}\right\vert ^{p}+\hat
{C}_{p,q,a,\gamma}~\mathbb{E}^{\mathcal{F}_{\sigma}}\Big(%
{\displaystyle\int_{\sigma}^{\theta_{n}}}
\mathbf{1}_{q\geq2}dR_{r}\Big)^{p/2}+\hat{C}_{p,q,a,\gamma}~\mathbb{E}%
^{\mathcal{F}_{\sigma}}\Big(%
{\displaystyle\int_{\sigma}^{\theta_{n}}}
dN_{r}\Big)^{p}%
\end{align*}
that yields via (\ref{an8}):%
\begin{equation}
\mathbb{E}^{\mathcal{F}_{\sigma}}\sup\limits_{\tau\in\left[  \sigma,\theta
_{n}\right]  }\left\vert Y_{\tau}\right\vert ^{p}\leq C_{p,q,a,\gamma
}~\mathbb{E}^{\mathcal{F}_{\sigma}}\left[  \left\vert Y_{\theta_{n}%
}\right\vert ^{p}+\Big(%
{\displaystyle\int_{\sigma}^{\theta_{n}}}
\mathbf{1}_{q\geq2}dR_{r}\Big)^{p/2}+\Big(%
{\displaystyle\int_{\sigma}^{\theta_{n}}}
dN_{r}\Big)^{p}\right]  \label{an9}%
\end{equation}
Hence form this last two inequalities we have%
\begin{equation}%
\begin{array}
[c]{l}%
\mathbb{E}^{\mathcal{F}_{\sigma}}\left[  \Big(%
{\displaystyle\int_{\sigma}^{\theta_{n}}}
\left\vert Y_{r}\right\vert ^{q-2}\mathbf{1}_{Y_{r}\neq0}\mathbf{1}_{q\geq
2}dR_{r}\Big)^{p/q}+\Big(%
{\displaystyle\int_{\sigma}^{\theta_{n}}}
\left\vert Y_{r}\right\vert ^{q-1}dN_{r}\Big)^{p/q}\right]  \medskip\\
\leq\tilde{C}_{p,q,a,\gamma}\mathbb{E}^{\mathcal{F}_{\sigma}}\left[
\left\vert Y_{\theta_{n}}\right\vert ^{p}+\Big(%
{\displaystyle\int_{\sigma}^{\theta_{n}}}
\mathbf{1}_{q\geq2}dR_{r}\Big)^{p/2}+\Big(%
{\displaystyle\int_{\sigma}^{\theta_{n}}}
dN_{r}\Big)^{p}\right]
\end{array}
\label{an9-a}%
\end{equation}
By Beppo Levi monotone convergence theorem and by Lebesgue dominated
convergence theorem, as $n\rightarrow\infty,$ we deduce, from (\ref{an9-a}),
(\ref{an9}), (\ref{an8}) and (\ref{an6-a}), inequalities (\ref{an5-a}) and
(\ref{an5-b}).\hfill
\end{proof}

\begin{proposition}
\label{an-prop-yd}Let $\left(  Y,Z\right)  \in S_{m}^{0}\times\Lambda_{m\times
k}^{0}$~, $q>1$ and $b,L>0$ satisfying%
\[
\mathbb{E}\sup\limits_{r\in\left[  0,T\right]  }\left\vert Y_{r}\right\vert
^{q}\leq L
\]
and for all $0\leq t\leq T<\infty:$%
\[
\left\vert Y_{t}\right\vert ^{q}+\mathbb{E}^{\mathcal{F}_{t}}%
{\displaystyle\int_{t}^{T}}
{dD}_{r}\leq b\,\mathbb{E}^{\mathcal{F}_{t}}\left[  \left\vert Y_{T}%
\right\vert ^{q}+%
{\displaystyle\int_{0}^{T}}
\left(  \left\vert Y_{r}\right\vert ^{q-2}\mathbf{1}_{Y_{r}\neq0}%
\mathbf{1}_{q\geq2}dR_{r}+\left\vert Y_{r}\right\vert ^{q-1}dN_{r}\right)
\right]  ,\quad\mathbb{P}-a.s.\;
\]
where $R,N$ and $D$ are increasing and continuous p.m.s.p. $R_{0}=N_{0}%
=D_{0}=0.$\newline Then for all $0<\alpha<1$%
\begin{equation}%
\begin{array}
[c]{l}%
\mathbb{E}\sup\nolimits_{t\in\left[  0,T\right]  }\left\vert Y_{t}\right\vert
^{\alpha q}+\Big(\mathbb{E}%
{\displaystyle\int_{0}^{T}}
{dD}_{r}\Big)^{\alpha}\medskip\\
\leq\dfrac{2b^{\alpha}}{1-\alpha}~\left[  \mathbb{E}\left\vert Y_{T}%
\right\vert ^{q}+L^{\frac{q-2}{q}}\bigg(\mathbb{E}\Big(%
{\displaystyle\int_{0}^{T}}
\mathbf{1}_{q\geq2}dR_{r}\Big)^{q/2}\bigg)^{2/q}+L^{\frac{q-1}{q}%
}\bigg(\mathbb{E}\Big(%
{\displaystyle\int_{\sigma}^{\theta}}
dN_{r}\Big)^{q}\bigg)^{1/q}\right]  ^{\alpha}\;
\end{array}
\label{an9-b}%
\end{equation}

\end{proposition}

\begin{proof}
By \cite[Proposition 1.56]{pa-ra/14} we obtain for all $0<\alpha<1:$%
\begin{align*}
&  \mathbb{E~}\sup_{t\in\left[  0,T\right]  }\Big(\left\vert Y_{t}\right\vert
^{q}+~\mathbb{E}^{\mathcal{F}_{t}}%
{\displaystyle\int_{t}^{T}}
{dD}_{r}\Big)^{\alpha}\\
&  \leq\frac{b^{\alpha}}{1-\alpha}\left[  \mathbb{E}\left\vert Y_{T}%
\right\vert ^{q}+\mathbb{E}%
{\displaystyle\int_{0}^{T}}
\left\vert Y_{r}\right\vert ^{q-2}\mathbf{1}_{Y_{r}\neq0}\mathbf{1}_{q\geq
2}dR_{r}+\mathbb{E}%
{\displaystyle\int_{0}^{T}}
\left\vert Y_{r}\right\vert ^{q-1}dN_{r}\right]  ^{\alpha}\\
&  \leq\frac{b^{\alpha}}{1-\alpha}\bigg[\mathbb{E}\left\vert Y_{T}\right\vert
^{q}+\mathbb{E}\Big(\sup_{r\in\left[  0,T\right]  }\left(  \left\vert
Y_{r}\mathbf{1}_{Y_{r}\neq0}\right\vert ^{q-2}\mathbf{1}_{q\geq2}\right)
{\displaystyle\int_{0}^{T}}
\mathbf{1}_{q\geq2}dR_{r}\Big)\\
&  \quad+\mathbb{E}\Big(\sup_{r\in\left[  0,T\right]  }\left\vert
Y_{r}\right\vert ^{q-1}%
{\displaystyle\int_{0}^{T}}
dN_{r}\Big)\bigg]^{\alpha}\\
&  \leq\frac{b^{\alpha}}{1-\alpha}\bigg[\mathbb{E}\left\vert Y_{T}\right\vert
^{q}+\Big(\mathbb{E}\sup_{r\in\left[  0,T\right]  }\left(  \left\vert
Y_{r}\mathbf{1}_{Y_{r}\neq0}\right\vert ^{q}\mathbf{1}_{q\geq2}\right)
\Big)^{\left(  q-2\right)  /q}\bigg(\mathbb{E}\Big(%
{\displaystyle\int_{0}^{T}}
\mathbf{1}_{q\geq2}dR_{r}\Big)^{q/2}\bigg)^{2/q}\\
&  \quad+\Big(\mathbb{E}\sup_{r\in\left[  0,T\right]  }\left(  \left\vert
Y_{r}\right\vert ^{q}\right)  \Big)^{\left(  q-1\right)  /q}\bigg(\mathbb{E}%
\Big(%
{\displaystyle\int_{0}^{T}}
dN_{r}\Big)^{q}\bigg)^{1/q}\bigg]^{\alpha}\\
&  \leq\frac{b^{\alpha}}{1-\alpha}\bigg[\mathbb{E}\left\vert Y_{T}\right\vert
^{q}+L^{\left(  q-2\right)  /q}\bigg(\mathbb{E}\Big(%
{\displaystyle\int_{0}^{T}}
\mathbf{1}_{q\geq2}dR_{r}\Big)^{q/2}\bigg)^{2/q}+L^{\left(  q-1\right)
/q}\bigg(\mathbb{E}\Big(%
{\displaystyle\int_{0}^{T}}
dN_{r}\Big)^{q}\bigg)^{1/q}\bigg]^{\alpha}%
\end{align*}
and inequality (\ref{an9-b}) follows since%
\[
\sup_{t\in\left[  0,T\right]  }\left\vert Y_{t}\right\vert ^{\alpha q}\leq
\sup_{t\in\left[  0,T\right]  }\Big(\left\vert Y_{t}\right\vert ^{q}%
+~\mathbb{E}^{\mathcal{F}_{t}}%
{\displaystyle\int_{t}^{T}}
{dD}_{r}\Big)^{\alpha}%
\]
and%
\[
\Big(\mathbb{E}%
{\displaystyle\int_{0}^{T}}
{dD}_{r}\Big)^{\alpha}\leq\sup_{t\in\left[  0,T\right]  }\Big(\left\vert
Y_{t}\right\vert ^{q}+~\mathbb{E}^{\mathcal{F}_{t}}%
{\displaystyle\int_{t}^{T}}
{\left\vert Y_{r}\right\vert ^{q-2}\mathbf{1}_{Y_{r}\neq0}dD}_{r}%
\Big)^{\alpha}%
\]

\hfill
\end{proof}

\begin{proposition}
\label{an-prop-dk} Let:

\begin{itemize}
\item $\left(  Y,Z\right)  \in S_{m}^{0}\times\Lambda_{m\times k}^{0}$ ;

\item $K\in S_{m}^{0}$ and $K_{\cdot}\left(  \omega\right)  \in BV_{loc}%
\left(  \mathbb{R}_{+};\mathbb{R}^{m}\right)  $, $\mathbb{P}$--a.s.;

\item $D,R,N,\tilde{R}$ be some increasing continuous p.m.s.p. with
$D_{0}=R_{0}=N_{0}=0$;

\item $V$ be a bounded variation p.m.s.p. with $V_{0}=0;$

\item $\sigma$ and $\theta$ be two stopping times such that $0\leq\sigma
\leq\theta<\infty.\medskip$
\end{itemize}

\noindent\textbf{I.} If for all $0\leq t\leq s<\infty,\;\mathbb{P}$-a.s.%
\begin{equation}
\left\vert Y_{t}\right\vert ^{2}+{%
{\displaystyle\int_{t}^{s}}
}\left\vert Z_{r}\right\vert ^{2}dr+{%
{\displaystyle\int_{t}^{s}}
}dD_{r}\leq\left\vert Y_{s}\right\vert ^{2}+2{%
{\displaystyle\int_{t}^{s}}
}\left\langle Y_{r},dK_{r}\right\rangle -2%
{\displaystyle\int_{t}^{s}}
\,\langle Y_{r},Z_{r}dB_{r}\rangle, \label{an-10-a}%
\end{equation}
and for some $\lambda<1$%
\begin{equation}
{%
{\displaystyle\int_{t}^{s}}
}\left\langle Y_{r},dK_{r}\right\rangle \leq%
{\displaystyle\int_{t}^{s}}
\left(  dR_{r}+|Y_{r}|dN_{r}+|Y_{r}|^{2}dV_{r}\right)  +\dfrac{\lambda}{2}%
{\displaystyle\int_{t}^{s}}
\left\vert Z_{r}\right\vert ^{2}dr, \label{an-10-b}%
\end{equation}
then, for any $q>0$, there exists a positive constant $C_{q,\lambda}$ such
that $\mathbb{P}$--a.s.%
\begin{equation}%
\begin{array}
[c]{l}%
\mathbb{E}^{\mathcal{F}_{\sigma}}\Big(%
{\displaystyle\int_{\sigma}^{\theta}}
e^{2V_{r}}\left\vert Z_{r}\right\vert ^{2}ds\Big)^{q/2}+\mathbb{E}%
^{\mathcal{F}_{\sigma}}\Big(%
{\displaystyle\int_{\sigma}^{\theta}}
e^{2V_{r}}dD_{r}\Big)^{q/2}\medskip\\
\leq C_{q,\lambda}\mathbb{E}^{\mathcal{F}_{\sigma}}\left[  \sup\limits_{r\in
\left[  \sigma,\theta\right]  }\left\vert e^{V_{r}}Y_{r}\right\vert ^{q}+\Big(%
{\displaystyle\int_{\sigma}^{\theta}}
e^{2V_{r}}dR_{r}\Big)^{q/2}+\Big(%
{\displaystyle\int_{\sigma}^{\theta}}
e^{2V_{r}}\left\vert Y_{r}\right\vert dN_{r}\Big)^{q/2}\right]  \medskip\\
\leq2C_{q,\lambda}\mathbb{E}^{\mathcal{F}_{\sigma}}\left[  \sup\limits_{r\in
\left[  \sigma,\theta\right]  }\left\vert e^{V_{r}}Y_{r}\right\vert ^{q}+\Big(%
{\displaystyle\int_{\sigma}^{\theta}}
e^{2V_{r}}dR_{r}\Big)^{q/2}+\Big(%
{\displaystyle\int_{\sigma}^{\theta}}
e^{V_{r}}dN_{r}\Big)^{q}\right]  .
\end{array}
\label{an-11}%
\end{equation}
\noindent\textbf{II.} If $q>1$, $n_{q}=\left(  q-1\right)  \wedge1,$%
\begin{equation}%
\begin{array}
[c]{rl}%
\left(  i\right)  & \left\vert Y_{t}\right\vert ^{q}+\dfrac{q}{2}n_{q}%
{\displaystyle\int_{t}^{s}}
\left\vert Y_{r}\right\vert ^{q-2}\mathbf{1}_{Y_{r}\neq0}\left\vert
Z_{r}\right\vert ^{2}dr+%
{\displaystyle\int_{t}^{s}}
\left\vert Y_{r}\right\vert ^{q-2}\mathbf{1}_{Y_{r}\neq0}dD_{r}\medskip\\
& \leq\left\vert Y_{s}\right\vert ^{q}+q%
{\displaystyle\int_{t}^{s}}
\left\vert Y_{r}\right\vert ^{q-2}\mathbf{1}_{Y_{r}\neq0}\left[  d\tilde
{R}_{r}+\left\langle Y_{r},dK_{r}\right\rangle \right]  -q%
{\displaystyle\int_{t}^{s}}
\left\vert Y_{r}\right\vert ^{q-2}\mathbf{1}_{Y_{r}\neq0}\left\langle
Y_{r},Z_{r}dB_{r}\right\rangle ,\medskip\\
\left(  ii\right)  & \mathbb{E}\sup\nolimits_{r\in\left[  \sigma
,\theta\right]  }e^{qV_{r}}\left\vert Y_{r}\right\vert ^{q}<\infty
\end{array}
\label{an-12}%
\end{equation}
and for some $\lambda<1$%
\begin{equation}
d\tilde{R}_{r}+\left\langle Y_{r},dK_{r}\right\rangle \leq\left(
\mathbf{1}_{q\geq2}dR_{r}+|Y_{r}|dN_{r}+|Y_{r}|^{2}dV_{r}\right)
+\dfrac{n_{q}}{2}\lambda\left\vert Z_{r}\right\vert ^{2}dt, \label{an-13}%
\end{equation}
then there exists some positive constant $C_{q,\lambda},$ $C_{q,\lambda
}^{\prime}$ such that $\mathbb{P}$--a.s.,%
\begin{equation}%
\begin{array}
[c]{l}%
\mathbb{E}^{\mathcal{F}_{\sigma}}\Big(\sup\limits_{\tau\in\left[
\sigma,\theta\right]  }\left\vert e^{V_{r}}Y_{r}\right\vert ^{q}%
\Big)+\mathbb{E}^{\mathcal{F}_{\sigma}}\Big(%
{\displaystyle\int_{\sigma}^{\theta}}
e^{qV_{r}}\left\vert Y_{r}\right\vert ^{q-2}\mathbf{1}_{Y_{r}\neq0}\left\vert
Z_{r}\right\vert ^{2}ds\Big)+\mathbb{E}^{\mathcal{F}_{\sigma}}\Big(%
{\displaystyle\int_{\sigma}^{\theta}}
e^{qV_{r}}\left\vert Y_{r}\right\vert ^{q-2}\mathbf{1}_{Y_{r}\neq0}%
dD_{r}\Big)\medskip\\
\leq C_{q,\lambda}~\mathbb{E}^{\mathcal{F}_{\sigma}}\left[  \left\vert
e^{V_{\theta}}Y_{\theta}\right\vert ^{q}+\Big(%
{\displaystyle\int_{\sigma}^{\theta}}
e^{qV_{r}}\left\vert Y_{r}\right\vert ^{q-2}\mathbf{1}_{Y_{r}\neq0}%
\mathbf{1}_{q\geq2}dR_{r}\Big)+\Big(%
{\displaystyle\int_{\sigma}^{\theta}}
e^{qV_{r}}\left\vert Y_{r}\right\vert ^{q-1}dN_{r}\Big)\right]  \medskip\\
\leq C_{q,\lambda}^{\prime}~\mathbb{E}^{\mathcal{F}_{\sigma}}\left[
\left\vert e^{V_{\theta}}Y_{\theta}\right\vert ^{q}+\Big(%
{\displaystyle\int_{\sigma}^{\theta}}
e^{2V_{r}}\mathbf{1}_{q\geq2}dR_{r}\Big)^{q/2}+\Big(%
{\displaystyle\int_{\sigma}^{\theta}}
e^{V_{r}}dN_{r}\Big)^{q}\right]  .
\end{array}
\label{an-14}%
\end{equation}

\end{proposition}

\begin{proof}
Using inequality (\ref{an-10-b}) and (\ref{an-10-a}) we obtain for all $0\leq
t\leq s<\infty$%
\[
\left\vert Y_{t}\right\vert ^{2}+\left(  1-\lambda\right)  {%
{\displaystyle\int_{t}^{s}}
}\left\vert Z_{r}\right\vert ^{2}dr+{%
{\displaystyle\int_{t}^{s}}
}dD_{r}\leq\left\vert Y_{s}\right\vert ^{2}+{%
{\displaystyle\int_{t}^{s}}
}\left[  \left(  2dR_{r}+2|Y_{r}|dN_{r}\right)  +|Y_{r}|^{2}d\left(
2V_{r}\right)  \right]  -2%
{\displaystyle\int_{t}^{s}}
\,\langle Y_{r},Z_{r}dB_{r}\rangle
\]
which yields, applying \cite[Proposition 6.69]{pa-ra/14} (or \cite[Lemma
12]{ma-ra/07}),%
\begin{align*}
&  \left\vert e^{V_{t}}Y_{t}\right\vert ^{2}+\left(  1-\lambda\right)  {%
{\displaystyle\int_{t}^{s}}
}\left\vert e^{V_{r}}Z_{r}\right\vert ^{2}dr+{%
{\displaystyle\int_{t}^{s}}
}e^{2V_{r}}dD_{r}\\
&  \leq\left\vert e^{V_{s}}Y_{s}\right\vert ^{2}+2{%
{\displaystyle\int_{t}^{s}}
}\left[  e^{2V_{r}}dR_{r}+|e^{V_{r}}Y_{r}|e^{V_{r}}dN_{r}\right]  -2%
{\displaystyle\int_{t}^{s}}
\,\langle e^{V_{r}}Y_{r},e^{V_{r}}Z_{r}dB_{r}\rangle.
\end{align*}
Now inequality (\ref{an-11}) clearly follows by Proposition \ref{an-prop-dz}.

In the same manner, using (\ref{an-13}) and (\ref{an-12}), \cite[Proposition
6.69]{pa-ra/14} and Proposition \ref{an-prop-ydz} we infer%
\[%
\begin{array}
[c]{l}%
\left\vert e^{V_{t}}Y_{t}\right\vert ^{q}+\dfrac{q}{2}n_{q}\left(
1-\lambda\right)
{\displaystyle\int_{t}^{s}}
\left\vert e^{V_{r}}Y_{r}\right\vert ^{q-2}\mathbf{1}_{Y_{r}\neq0}\left\vert
e^{V_{r}}Z_{r}\right\vert ^{2}dr+%
{\displaystyle\int_{t}^{s}}
\left\vert e^{V_{r}}Y_{r}\right\vert ^{q-2}\mathbf{1}_{Y_{r}\neq0}e^{2V_{r}%
}dD_{r}\medskip\\
\leq\left\vert e^{V_{s}}Y_{s}\right\vert ^{q}+q%
{\displaystyle\int_{t}^{s}}
\left\vert e^{V_{r}}Y_{r}\right\vert ^{q-2}\mathbf{1}_{Y_{r}\neq0}%
\mathbf{1}_{q\geq2}e^{2V_{r}}dR_{r}+|e^{V_{r}}Y_{r}|^{q-1}e^{V_{r}}%
dN_{r}\medskip\\
\quad-q%
{\displaystyle\int_{t}^{s}}
\left\vert e^{V_{r}}Y_{r}\right\vert ^{q-2}\mathbf{1}_{Y_{r}\neq0}\left\langle
e^{V_{r}}Y_{r},e^{V_{r}}Z_{r}dB_{r}\right\rangle ,
\end{array}
\]
which yields (\ref{an-14}).\hfill$\medskip$
\end{proof}

With a similar proof we deduce (see also \cite[Corollary 6.81]{pa-ra/14}) the
next results.

\begin{proposition}
[{see \cite[Proposition 6.80]{pa-ra/14}}]\label{Appendix_result 1}Let $\left(
Y,Z\right)  \in S_{m}^{0}\times\Lambda_{m\times k}^{0}$ satisfying%
\[
Y_{t}=Y_{T}+\int_{t}^{T}dK_{s}-\int_{t}^{T}Z_{s}dB_{s},\;0\leq t\leq
T,\quad\mathbb{P}\text{--a.s.},
\]
where $K\in S_{m}^{0}$ and $K_{\cdot}\left(  \omega\right)  \in BV_{loc}%
\left(  \mathbb{R}_{+};\mathbb{R}^{m}\right)  $, $\mathbb{P}$--a.s.

Let $\tau$ and $\sigma$ be two stopping times such that $0\leq\tau\leq
\sigma<\infty$. Assume that there exists three increasing and continuous
p.m.s.p. $D,R,N$ with $D_{0}=R_{0}=N_{0}=0$ and a bounded variation p.m.s.p.
$V$ with $V_{0}=0$ such that for $\lambda<1,$%
\[
dD_{t}+\left\langle Y_{t},dK_{t}\right\rangle \leq dR_{t}+|Y_{t}|dN_{t}%
+|Y_{t}|^{2}dV_{t}+\dfrac{\lambda}{2}\left\vert Z_{t}\right\vert ^{2}dt.
\]
Then, for any $q>0$, there exists a positive constant $C_{q,\lambda}$ such
that $\mathbb{P}$--a.s.%
\[%
\begin{array}
[c]{l}%
\displaystyle\mathbb{E}^{\mathcal{F}_{\tau}}\Big(\int_{\tau}^{\sigma}%
e^{2V_{s}}dD_{s}\Big)^{q/2}+\mathbb{E}^{\mathcal{F}_{\tau}}\Big(\int_{\tau
}^{\sigma}e^{2V_{s}}\left\vert Z_{s}\right\vert ^{2}ds\Big)^{q/2}\medskip\\
\multicolumn{1}{r}{\displaystyle\leq C_{q,\lambda}\mathbb{E}^{\mathcal{F}%
_{\tau}}\left[  \sup\limits_{s\in\left[  \tau,\sigma\right]  }\left\vert
e^{V_{s}}Y_{s}\right\vert ^{q}+\Big(\int_{\tau}^{\sigma}e^{2V_{s}}%
dR_{s}\Big)^{q/2}+\Big(\int_{\tau}^{\sigma}e^{V_{s}}dN_{s}\Big)^{q}\right]  .}%
\end{array}
\]
Moreover, if $p>1$ and%
\[%
\begin{array}
[c]{l}%
\displaystyle dD_{t}+\left\langle Y_{t},dK_{t}\right\rangle \leq\left(
\mathbf{1}_{p\geq2}dR_{t}+|Y_{t}|dN_{t}+|Y_{t}|^{2}dV_{t}\right)
+\dfrac{n_{p}}{2}\lambda\left\vert Z_{t}\right\vert ^{2}dt,\medskip\\
\displaystyle\mathbb{E}\sup\nolimits_{s\in\left[  \tau,\sigma\right]
}e^{pV_{s}}\left\vert Y_{s}\right\vert ^{p}<\infty,
\end{array}
\]
then there exists a positive constant $C_{p,\lambda}$ such that $\mathbb{P}%
$--a.s.,%
\begin{equation}%
\begin{array}
[c]{l}%
\displaystyle\mathbb{E}^{\mathcal{F}_{\tau}}\Big(\sup\nolimits_{s\in\left[
\tau,\sigma\right]  }\left\vert e^{V_{s}}Y_{s}\right\vert ^{p}\Big)+\mathbb{E}%
^{\mathcal{F}_{\tau}}\Big(\int_{\tau}^{\sigma}e^{2V_{s}}dD_{s}\Big)^{p/2}%
+\mathbb{E}^{\mathcal{F}_{\tau}}\Big(\int_{\tau}^{\sigma}e^{2V_{s}}\left\vert
Z_{s}\right\vert ^{2}ds\Big)^{p/2}\medskip\\
\displaystyle\leq C_{p,\lambda}~\mathbb{E}^{\mathcal{F}_{\tau}}\left[
\left\vert e^{V_{\sigma}}Y_{\sigma}\right\vert ^{p}+\Big(\int_{\tau}^{\sigma
}e^{2V_{s}}\mathbf{1}_{p\geq2}dR_{s}\Big)^{p/2}+\Big(\int_{\tau}^{\sigma
}e^{V_{s}}dN_{s}\Big)^{p}\right]  .
\end{array}
\label{an3a}%
\end{equation}

\end{proposition}

Based mainly on this previous result one can prove :

\begin{proposition}
[{see \cite[Corollary 6.81]{pa-ra/14}}]\label{Appendix_result 2}Let $\left(
Y,Z\right)  \in S_{m}^{0}\times\Lambda_{m\times k}^{0}$ satisfying%
\[
Y_{t}=Y_{T}+\int_{t}^{T}dK_{s}-\int_{t}^{T}Z_{s}dB_{s},\;0\leq t\leq
T,\quad\mathbb{P}-a.s.,
\]
where $K\in S_{m}^{0}$ and $K_{\cdot}\left(  \omega\right)  \in BV_{loc}%
\left(  \mathbb{R}_{+};\mathbb{R}^{m}\right)  ,\;\mathbb{P}$--a.s.

Let $\tau$ and $\sigma$ be two stopping times such that $0\leq\tau\leq
\sigma<\infty$. Assume that there exists two increasing and continuous
p.m.s.p. $D,N$ with $N_{0}=0$ and a bounded variation p.m.s.p. $V$ with
$V_{0}=0$ such that for $\lambda<1,$%
\[%
\begin{array}
[c]{l}%
dD_{t}+\left\langle Y_{t},dK_{t}\right\rangle \leq|Y_{t}|dN_{t}+|Y_{t}%
|^{2}dV_{t}~,\medskip\\
\mathbb{E}\sup\nolimits_{s\in\left[  \tau,\sigma\right]  }\left\vert e^{V_{s}%
}Y_{s}\right\vert <\infty.
\end{array}
\]
Then%
\[
e^{V_{\tau}}\left\vert Y_{\tau}\right\vert \leq\mathbb{E}^{\mathcal{F}_{\tau}%
}e^{V_{\sigma}}\left\vert Y_{\sigma}\right\vert +\mathbb{E}^{\mathcal{F}%
_{\tau}}\int_{\tau}^{\sigma}e^{V_{s}}dN_{s}%
\]
and for all $0<a<1$%
\[%
\begin{array}
[c]{l}%
\displaystyle\sup\nolimits_{s\in\left[  \tau,\sigma\right]  }\left[
\mathbb{E}\left(  e^{V_{s}}\left\vert Y_{s}\right\vert \right)  \right]
^{a}+\mathbb{E}\Big(\sup\nolimits_{s\in\left[  \tau,\sigma\right]  }\left\vert
e^{V_{s}}Y_{s}\right\vert ^{a}\Big)+\mathbb{E}\Big(\int_{\tau}^{\sigma
}e^{2V_{s}}\left\vert Z_{s}\right\vert ^{2}ds\Big)^{a/2}+\mathbb{E}%
\Big(\int_{\tau}^{\sigma}e^{2V_{s}}dD_{s}\Big)^{a/2}\smallskip\\
\displaystyle\leq C_{a}\Big(\mathbb{E}\left(  e^{V_{\sigma}}\left\vert
Y_{\sigma}\right\vert \right)  \Big)^{a}+C_{a}\Big(\mathbb{E}\int_{\tau
}^{\sigma}e^{V_{s}}dN_{s}\Big)^{a}%
\end{array}
\]

\end{proposition}

\subsection{An It\^{o}'s formula and some backward stochastic inequalities}

\begin{proposition}
\label{p1-ito} (see \cite[Section 2.3.1]{pa-ra/14}) Let $p\in\mathbb{R}$,
$\rho\geq0$ and $\delta\geq0$ if $p\geq2$ and $\delta>0$ if $p<2.$ Let $Y\in
S_{d}^{0}$ be a local \textit{semimartingale of the form}%
\begin{equation}%
\begin{array}
[c]{l}%
\displaystyle Y_{t}=Y_{0}-\int_{0}^{t}F_{r}dr+\int_{0}^{t}R_{r}dB_{r},\quad
t\geq0\medskip\quad\text{or equivalently}\\
\displaystyle Y_{t}=Y_{T}+\int_{t}^{T}F_{r}dr-\int_{t}^{T}R_{r}dB_{r}%
,\quad\text{for all }0\leq t\leq T,
\end{array}
\label{ito1}%
\end{equation}
\textit{where }$R\in\Lambda_{m\times k}^{0}~,$ $F\in S_{m}^{0}$. Let
$\varphi=\varphi_{\rho,\delta}:\mathbb{R}^{d}\rightarrow\left]  0,\infty
\right[  $%
\[
\varphi\left(  x\right)  =\varphi_{\rho,\delta}\left(  x\right)  =\left(
\dfrac{\left\vert x\right\vert ^{2}}{1+\rho\left\vert x\right\vert ^{2}%
}+\delta\right)  ^{1/2}.
\]
By It\^{o}'s formula for $\varphi_{\rho,\delta}^{p}\left(  Y_{t}\right)  $,
$p\in\mathbb{R}$ we have for all $0\leq t\leq s\leq T:$%
\begin{equation}%
\begin{array}
[c]{l}%
\displaystyle\varphi_{\rho,\delta}^{p}\left(  Y_{t}\right)  +\dfrac{p}{2}%
{\displaystyle\int_{t}^{s}}
R_{r}^{\left(  p,\rho,\delta\right)  }dr+\dfrac{p}{2}\left[  L_{s}^{\left(
p,\rho,\delta\right)  }-L_{t}^{\left(  p,\rho,\delta\right)  }\right]
\medskip\\
\displaystyle=\varphi_{\rho,\delta}^{p}\left(  Y_{s}\right)  +\dfrac{p}{2}%
{\displaystyle\int_{t}^{s}}
Q_{r}^{\left(  p,\rho,\delta\right)  }dr+p%
{\displaystyle\int_{t}^{s}}
\big\langle U_{r}^{\left(  p,\rho,\delta\right)  },F_{r}\big\rangle-p%
{\displaystyle\int_{t}^{s}}
\big\langle U_{r}^{\left(  p,\rho,\delta\right)  },R_{r}dB_{r}%
\big\rangle,\quad a.s.
\end{array}
\label{ito2}%
\end{equation}
where$\medskip$

$U_{r}^{\left(  p,\rho,\delta\right)  }=\varphi_{\rho,\delta}^{p-2}\left(
Y_{r}\right)  \,\dfrac{1}{(1+\rho|Y_{r}|^{2})^{2}}\,Y_{r}~,\medskip$

$R_{r}^{\left(  p,\rho,\delta\right)  }=\varphi_{\rho,\delta}^{p-4}\left(
Y_{r}\right)  \,\dfrac{1}{(1+\rho|Y_{r}|^{2})^{3}}\,\bigg[\dfrac{p-1}%
{1+\rho\left\vert Y_{r}\right\vert ^{2}}\left\vert R_{r}^{\ast}Y_{r}%
\right\vert ^{2}+\left(  \left\vert R_{r}\right\vert ^{2}\left\vert
Y_{r}\right\vert ^{2}-\left\vert R_{r}^{\ast}Y_{r}\right\vert ^{2}\right)
\bigg]\medskip$

$L_{t}^{\left(  p,\rho,\delta\right)  }=\delta~%
{\displaystyle\int_{0}^{t}}
\varphi_{\rho,\delta}^{p-4}\left(  Y_{r}\right)  \,\dfrac{1}{(1+\rho
|Y_{r}|^{2})^{3}}\,\left[  \left\vert R_{r}\right\vert ^{2}+\rho\left(
\left\vert R_{r}\right\vert ^{2}\left\vert Y_{r}\right\vert ^{2}-\left\vert
R_{r}^{\ast}Y_{r}\right\vert ^{2}\right)  \right]  dr.\medskip$

and$\medskip$

$Q_{r}^{\left(  p,\rho,\delta\right)  }=\varphi_{\rho,\delta}^{p-2}\left(
Y_{r}\right)  \,\dfrac{3\rho}{(1+\rho|Y_{r}|^{2})^{3}}\,\left\vert R_{r}%
^{\ast}Y_{r}\right\vert ^{2},\medskip$

In the case $\rho=0$ we have%
\begin{equation}%
\begin{array}
[c]{l}%
\displaystyle\left(  \left\vert Y_{t}\right\vert ^{2}+\delta\right)
^{p/2}\medskip\\
\displaystyle\quad+\dfrac{p}{2}%
{\displaystyle\int_{t}^{s}}
\left(  \left\vert Y_{r}\right\vert ^{2}+\delta\right)  ^{\left(  p-4\right)
/2}\left[  \left(  p-1\right)  \left\vert R_{r}^{\ast}Y_{r}\right\vert
^{2}+\left(  \left\vert R_{r}\right\vert ^{2}\left\vert Y_{r}\right\vert
^{2}-\left\vert R_{r}^{\ast}Y_{r}\right\vert ^{2}\right)  +\delta\left\vert
R_{r}\right\vert ^{2}\right]  \medskip\\
\displaystyle=\left(  \left\vert Y_{s}\right\vert ^{2}+\delta\right)  ^{p/2}+p%
{\displaystyle\int_{t}^{s}}
\left(  \left\vert Y_{r}\right\vert ^{2}+\delta\right)  ^{\left(  p-2\right)
/2}\left\langle Y_{r},F_{r}\right\rangle -p%
{\displaystyle\int_{t}^{s}}
\left(  \left\vert Y_{r}\right\vert ^{2}+\delta\right)  ^{\left(  p-2\right)
/2}\left\langle Y_{r},R_{r}dB_{r}\right\rangle .
\end{array}
\label{ito3}%
\end{equation}

\end{proposition}

\begin{remark}
\label{r1-ito}If $p\geq1$ and $n_{p}=\left(  p-1\right)  \wedge1$ then
\begin{align*}
&  \left(  p-1\right)  \left\vert R_{r}^{\ast}Y_{r}\right\vert ^{2}+\left(
\left\vert R_{r}\right\vert ^{2}\left\vert Y_{r}\right\vert ^{2}-\left\vert
R_{r}^{\ast}Y_{r}\right\vert ^{2}\right)  +\delta\left\vert R_{r}\right\vert
^{2}\\
&  \geq n_{p}\left[  \left\vert R_{r}^{\ast}Y_{r}\right\vert ^{2}+\left(
\left\vert R_{r}\right\vert ^{2}\left\vert Y_{r}\right\vert ^{2}-\left\vert
R_{r}^{\ast}Y_{r}\right\vert ^{2}\right)  +\delta\left\vert R_{r}\right\vert
^{2}\right] \\
&  =n_{p}\left(  \left\vert Y_{r}\right\vert ^{2}+\delta\right)  \left\vert
R_{r}\right\vert ^{2}%
\end{align*}
and from (\ref{ito3}) we infer%
\begin{equation}%
\begin{array}
[c]{l}%
\displaystyle\left(  \left\vert Y_{t}\right\vert ^{2}+\delta\right)
^{p/2}+\dfrac{p}{2}n_{p}%
{\displaystyle\int_{t}^{s}}
\left(  \left\vert Y_{r}\right\vert ^{2}+\delta\right)  ^{\left(  p-2\right)
/2}\left\vert R_{r}\right\vert ^{2}\medskip\\
\displaystyle\leq\left(  \left\vert Y_{s}\right\vert ^{2}+\delta\right)
^{p/2}+p%
{\displaystyle\int_{t}^{s}}
\left(  \left\vert Y_{r}\right\vert ^{2}+\delta\right)  ^{\left(  p-2\right)
/2}\left\langle Y_{r},F_{r}\right\rangle \medskip\\
\displaystyle\quad-p%
{\displaystyle\int_{t}^{s}}
\left(  \left\vert Y_{r}\right\vert ^{2}+\delta\right)  ^{\left(  p-2\right)
/2}\left\langle Y_{r},R_{r}dB_{r}\right\rangle ,\quad\mathbb{P}\text{--a.s.,}%
\end{array}
\label{ito4}%
\end{equation}
for all $0\leq t\leq s\leq T.$
\end{remark}

\subsection{Smoothing approximations}

\begin{lemma}
\label{L1-approx}Let $\varepsilon>0.$ Let $Q:\Omega\times\mathbb{R}%
_{+}\rightarrow\mathbb{R}_{+},$ $Q_{0}=0,$ be a strictly increasing continuous
stochastic process, $\lim_{t\rightarrow\infty}Q_{t}=\infty$ and $G:\Omega
\times\mathbb{R}_{+}\rightarrow\mathbb{R}^{m}$ be a measurable stochastic
process such that $\sup_{t\in\mathbb{R}_{+}}\left\vert G_{t}\left(
\omega\right)  \right\vert <\infty$, $\mathbb{P}$-a.s.. Define%
\[
G_{t}^{\varepsilon}=%
{\displaystyle\int_{t\vee\varepsilon}^{\infty}}
\frac{1}{Q_{\varepsilon}}e^{-\frac{Q_{r}-Q_{t\vee\varepsilon}}{Q_{\varepsilon
}}}G_{r}dQ_{r}\,.
\]
Then $G^{\varepsilon}:\Omega\times\mathbb{R}_{+}\rightarrow\mathbb{R}^{m}$ are
continuous stochastic processes and $\mathbb{P}$-a.s.%
\begin{equation}%
\begin{array}
[c]{ll}%
\left(  a\right)  & \left\vert G_{t}^{\varepsilon}\left(  \omega\right)
\right\vert \leq\sup\nolimits_{r\geq0}\left\vert G_{r}\left(  \omega\right)
\right\vert ,\quad\text{for all }t\geq0;\medskip\\
\left(  b\right)  & \lim\nolimits_{\varepsilon\rightarrow0}G_{t}^{\varepsilon
}\left(  \omega\right)  =G_{t}~\left(  \omega\right)  ,\;\;\text{a.e.}%
\mathbb{\;}t\geq0;\medskip\\
\left(  c\right)  & \left\vert G_{t}^{\varepsilon}-G_{t}\right\vert \leq
\sup\nolimits_{r\geq0}\left\vert G_{r}\right\vert ~\exp\left(  2-\frac
{1}{\sqrt{Q_{\varepsilon}}}\right)  \medskip\\
& \quad+\sup\nolimits_{r\geq0}\left\{  \left\vert G_{r}-G_{t}\right\vert
:0\leq Q_{r}-Q_{t}\leq\sqrt{Q_{\varepsilon}}\right\}  ,\quad\text{for all
}t\geq0.
\end{array}
\label{a1}%
\end{equation}
If, moreover, $G$ is a continuous stochastic process then for all $T>0:$
\begin{equation}
\lim\limits_{\varepsilon\rightarrow0}\left(  \sup\limits_{s\in\left[
0,T\right]  }\left\vert G_{s}^{\varepsilon}\left(  \omega\right)
-G_{s}\left(  \omega\right)  \right\vert \right)  =0,\quad\mathbb{P}%
\text{-a.s..} \label{a2}%
\end{equation}

\end{lemma}

\begin{proof}
$\left(  b\right)  $ Let $n\in\mathbb{N}^{\ast}.$ We can assume $0<\varepsilon
<t.$%
\begin{align*}
\left\vert G_{t}^{\varepsilon}-G_{t}\right\vert  &  \leq\int_{t}^{\infty
}e^{-\frac{Q_{r}-Q_{t}}{Q_{\varepsilon}}}\dfrac{1}{Q_{\varepsilon}}%
|G_{r}-G_{t}|dQ_{r}\\
&  \leq\int_{0}^{\infty}e^{-s}\left\vert G_{Q^{-1}\left(  Q_{t+sQ_{\varepsilon
}}\right)  }-G_{Q^{-1}\left(  Q_{t}\right)  }\right\vert ds\\
&  \leq%
{\displaystyle\int_{0}^{n}}
\left\vert G_{Q^{-1}\left(  Q_{t+sQ_{\varepsilon}}\right)  }-G_{Q^{-1}\left(
Q_{t}\right)  }\right\vert ds+2\sup\limits_{r\geq0}\left\vert G_{r}\right\vert
\int_{n}^{\infty}e^{-s}ds.
\end{align*}
Since%
\[
\lim_{\varepsilon\rightarrow0}%
{\displaystyle\int_{0}^{n}}
\left\vert G_{Q^{-1}\left(  Q_{t+sQ_{\varepsilon}}\right)  }-G_{Q^{-1}\left(
Q_{t}\right)  }\right\vert ds=0,\;\text{a.e. }t\in\left[  0,n\right]  ,
\]
we have for all $n\in\mathbb{N}^{\ast}$%
\[
\limsup_{\varepsilon\rightarrow0}\left\vert G_{t}^{\varepsilon}-G_{t}%
\right\vert \leq2e^{-n}\sup\limits_{r\geq0}\left\vert G_{r}\right\vert
,\;\text{a.e. }t\in\left(  0,T\right)  .
\]
which yields $\left(  b\right)  .\medskip$

$\left(  c\right)  $ Let $t_{\varepsilon}=Q^{-1}\left(  Q_{t}+\sqrt
{Q_{\varepsilon}}\right)  .$We have%
\begin{align*}
\left\vert G_{t}^{\varepsilon}-G_{t}\right\vert  &  \leq\,\int_{t\vee
\varepsilon}^{\infty}e^{-\frac{Q_{r}-Q_{t\vee\varepsilon}}{Q_{\varepsilon}}%
}\dfrac{1}{Q_{\varepsilon}}|G_{r}-G_{t}|dQ_{r}\\
&  \leq\sup_{r\in\left[  t\vee\varepsilon,t_{\varepsilon}\vee\varepsilon
\right]  }|G_{r}-G_{t}|%
{\displaystyle\int_{t\vee\varepsilon}^{t_{\varepsilon}\vee\varepsilon}}
\dfrac{1}{Q_{\varepsilon}}e^{-\frac{Q_{r}-Q_{t\vee\varepsilon}}{Q_{\varepsilon
}}}dQ_{r}+2\sup_{s\geq0}|G_{s}|~\int_{t_{\varepsilon}}^{\infty}\dfrac
{1}{Q_{\varepsilon}}e^{-\frac{Q_{r}-Q_{t\vee\varepsilon}}{Q_{\varepsilon}}%
}dQ_{r}\\
&  \leq\sup_{r\in\left[  t\vee\varepsilon,t_{\varepsilon}\vee\varepsilon
\right]  }|G_{r}-G_{t}|+2e^{-\frac{Q_{t_{\varepsilon}}-Q_{t\vee\varepsilon}%
}{Q_{\varepsilon}}}\sup_{s\geq0}|G_{s}|~.
\end{align*}
Since%
\[
\frac{Q_{t_{\varepsilon}}-Q_{t\vee\varepsilon}}{Q_{\varepsilon}}=\frac
{\sqrt{Q_{\varepsilon}}}{Q_{\varepsilon}}+\frac{Q_{t}-Q_{t\vee\varepsilon}%
}{Q_{\varepsilon}}\geq\frac{1}{\sqrt{Q_{\varepsilon}}}-1,
\]
we obtain (\ref{a1}-d).

Clearly, (\ref{a2}) follows from (\ref{a1}).\hfill
\end{proof}

\begin{remark}
\label{R1-approx}Let $\varepsilon>0.$ Let $Q:\Omega\times\mathbb{R}%
\rightarrow\mathbb{R},$ $Q_{0}=0,$ be a strictly increasing continuous
stochastic process, $\lim_{t\rightarrow\infty}Q_{t}=\infty,$ $\lim
_{t\rightarrow-\infty}Q_{t}=-\infty$ and $G:\Omega\times\mathbb{R}%
\rightarrow\mathbb{R}^{m}$ be bounded measurable stochastic processes then
similar boundedness and convergence results hold for $G^{i,\varepsilon}%
:\Omega\times\mathbb{R}\rightarrow\mathbb{R}^{m}$ defined by%
\begin{align*}
G_{t}^{1,\varepsilon}  &  =%
{\displaystyle\int_{t}^{\infty}}
G_{r}\frac{1}{Q_{\varepsilon}}e^{-\frac{Q_{r}-Q_{t}}{Q_{\varepsilon}}}%
dQ_{r}~,\quad t\in\mathbb{R},\\
G_{t}^{2,\varepsilon}  &  =\int_{-\infty}^{t}G_{r}\frac{1}{Q_{\varepsilon}%
}e^{-\frac{Q_{t}-Q_{r}}{Q_{\varepsilon}}}dQ_{r}~,\quad t\in\mathbb{R},\\
G_{t}^{3,\varepsilon}  &  =e^{-\frac{Q_{t}}{Q_{\varepsilon}}}G_{0}+%
{\displaystyle\int_{0}^{t}}
G_{r}\frac{1}{Q_{\varepsilon}}e^{-\frac{Q_{t}-Q_{r}}{Q_{\varepsilon}}}%
dQ_{r}~\\
&  =\int_{-\infty}^{t}\left[  1_{(-\infty,0)}\left(  r\right)  G_{0}%
+\mathbf{1}_{[0,\infty)}\left(  r\right)  G_{r}\right]  \frac{1}%
{Q_{\varepsilon}}e^{-\frac{Q_{t}-Q_{r}}{Q_{\varepsilon}}}dQ_{r},\quad
t\geq0,\\
G_{t}^{4,\varepsilon}  &  =\mathbf{1}_{[0,\varepsilon)}\left(  t\right)
G_{0}+\mathbf{1}_{[\varepsilon,\infty)}\left(  t\right)
{\displaystyle\int_{0}^{t}}
G_{r}\frac{1}{Q_{\varepsilon}}e^{-\frac{Q_{t}-Q_{r\vee\varepsilon}%
}{Q_{\varepsilon}}}dQ_{r}~,\quad t\geq0.
\end{align*}

\end{remark}

\begin{corollary}
\label{C1-approx}Let the assumptions of Lemma \ref{L1-approx} be satisfied and
$\varphi:\mathbb{R}^{m}\rightarrow(-\infty,+\infty]$ be a proper convex lower
semicontinuous function such that $\int_{0}^{\infty}\left\vert \varphi\left(
G_{u}\right)  \right\vert dQ_{u}<\infty.$ Then for all $0\leq\alpha\leq\beta:$%
\[
\lim_{\varepsilon\rightarrow0}\int_{\alpha}^{\beta}\varphi\left(
G_{r}^{\varepsilon}\right)  dQ_{r}=\int_{\alpha}^{\beta}\varphi\left(
G_{r}\right)  dQ_{r}.
\]
Moreover, if $\mathbb{E}\int_{0}^{\infty}\left\vert \varphi\left(
G_{u}\right)  \right\vert dQ_{u}<\infty.$ Then for any stopping times
$0\leq\sigma\leq\theta$%
\[
\lim_{\varepsilon\rightarrow0}\mathbb{E}\int_{\sigma}^{\theta}\varphi\left(
G_{r}^{\varepsilon}\right)  dQ_{r}=\mathbb{E}\int_{\sigma}^{\theta}%
\varphi\left(  G_{r}\right)  dQ_{r}.
\]

\end{corollary}

\begin{proof}
We have%
\begin{align*}
\int_{\alpha}^{\beta}\varphi\left(  G_{r}^{\varepsilon}\right)  dQ_{r}  &
\leq\int_{\alpha}^{\beta}\left(
{\displaystyle\int_{r\vee\varepsilon}^{\infty}}
\frac{1}{Q_{\varepsilon}}e^{-\frac{Q_{u}-Q_{r\vee\varepsilon}}{Q_{\varepsilon
}}}\varphi\left(  G_{u}\right)  dQ_{u}\right)  dQ_{r}\\
&  =\int_{0}^{\infty}\varphi\left(  G_{u}\right)  \left(  \int_{0}^{\infty
}\mathbf{1}_{\left[  \alpha,\beta\right]  }\left(  r\right)  \mathbf{1}%
_{[r\vee\varepsilon,\infty)}\left(  u\right)  \frac{1}{Q_{\varepsilon}%
}e^{-\frac{Q_{u}-Q_{r\vee\varepsilon}}{Q_{\varepsilon}}}dQ_{r}\right)
dQ_{u}\\
&  =\int_{0}^{\infty}\varphi\left(  G_{u}\right)  \mathbf{1}_{[\varepsilon
,\infty)}\left(  u\right)  \left(  \int_{0}^{u}\mathbf{1}_{\left[
\alpha,\beta\right]  }\left(  r\right)  \frac{1}{Q_{\varepsilon}}%
e^{-\frac{Q_{u}-Q_{r\vee\varepsilon}}{Q_{\varepsilon}}}dQ_{r}\right)  dQ_{u},
\end{align*}
since $\mathbf{1}_{[r\vee\varepsilon,\infty)}\left(  u\right)  =\mathbf{1}%
_{[0,u]}\left(  r\right)  $ $\mathbf{1}_{[\varepsilon,\infty)}\left(
u\right)  .$

We remark that%
\[%
\begin{array}
[c]{l}%
\displaystyle%
{\displaystyle\int_{0}^{u}}
\mathbf{1}_{\left[  \alpha,\beta\right]  }\left(  r\right)  \frac
{1}{Q_{\varepsilon}}e^{-\frac{Q_{u}-Q_{r}}{Q_{\varepsilon}}}dQ_{r}\leq%
{\displaystyle\int_{0}^{u}}
\mathbf{1}_{\left[  \alpha,\beta\right]  }\left(  r\right)  \frac
{1}{Q_{\varepsilon}}e^{-\frac{Q_{u}-Q_{r\vee\varepsilon}}{Q_{\varepsilon}}%
}dQ_{r}\medskip\\
\displaystyle=%
{\displaystyle\int_{0}^{u}}
\mathbf{1}_{\left[  \alpha,\beta\right]  }\left(  r\right)  \left[
\mathbf{1}_{[0,\varepsilon)}\left(  r\right)  \frac{1}{Q_{\varepsilon}%
}e^{-\frac{Q_{u}-Q_{\varepsilon}}{Q_{\varepsilon}}}+\mathbf{1}_{[\varepsilon
,\infty)}\left(  r\right)  \frac{1}{Q_{\varepsilon}}e^{-\frac{Q_{u}-Q_{r}%
}{Q_{\varepsilon}}}\right]  dQ_{r}\medskip\\
\displaystyle\leq\frac{Q_{u\wedge\varepsilon}}{Q_{\varepsilon}}e^{1-\frac
{Q_{u}}{Q_{\varepsilon}}}+%
{\displaystyle\int_{0}^{u}}
\mathbf{1}_{\left[  \alpha,\beta\right]  }\left(  r\right)  \frac
{1}{Q_{\varepsilon}}e^{-\frac{Q_{u}-Q_{r}}{Q_{\varepsilon}}}dQ_{r}%
\end{array}
\]
and by Remark \ref{R1-approx} (with the extension $Q_{r}=r$ for $r<0$)%
\[
\lim\limits_{\varepsilon\rightarrow0}%
{\displaystyle\int_{0}^{u}}
\mathbf{1}_{\left[  \alpha,\beta\right]  }\left(  r\right)  \frac
{1}{Q_{\varepsilon}}e^{-\frac{Q_{u}-Q_{r}}{Q_{\varepsilon}}}dQ_{r}%
=\lim\limits_{\varepsilon\rightarrow0}%
{\displaystyle\int_{-\infty}^{u}}
\mathbf{1}_{\left[  \alpha,\beta\right]  }\left(  r\right)  \frac
{1}{Q_{\varepsilon}}e^{-\frac{Q_{u}-Q_{r}}{Q_{\varepsilon}}}dQ_{r}%
=\mathbf{1}_{\left[  \alpha,\beta\right]  }\left(  u\right)  ,\quad\text{a.e.
}u\geq0.
\]
Hence%
\[
\lim\limits_{\varepsilon\rightarrow0}%
{\displaystyle\int_{0}^{u}}
\mathbf{1}_{\left[  \alpha,\beta\right]  }\left(  r\right)  \frac
{1}{Q_{\varepsilon}}e^{-\frac{Q_{u}-Q_{r\vee\varepsilon}}{Q_{\varepsilon}}%
}dQ_{r}=\mathbf{1}_{\left[  \alpha,\beta\right]  }\left(  u\right)
,\quad\text{a.e. }u\geq0.
\]
Moreover,%
\[
0\leq\int_{0}^{u}\mathbf{1}_{\left[  \alpha,\beta\right]  }\left(  r\right)
\frac{1}{Q_{\varepsilon}}e^{-\frac{Q_{u}-Q_{r\vee\varepsilon}}{Q_{\varepsilon
}}}dQ_{r}\leq e+1.
\]
Then, by Fatou's Lemma and by the Lebesgue dominated convergence theorem, we
infer%
\[
\int_{\alpha}^{\beta}\varphi\left(  G_{r}\right)  dQ_{r}\leq\liminf
_{\varepsilon\rightarrow0}\int_{\alpha}^{\beta}\varphi\left(  G_{r}%
^{\varepsilon}\right)  dQ_{r}\leq\limsup_{\varepsilon\rightarrow0}\int
_{\alpha}^{\beta}\varphi\left(  G_{r}^{\varepsilon}\right)  dQ_{r}%
=\int_{\alpha}^{\beta}\varphi\left(  G_{r}\right)  dQ_{r}~.
\]
The second assertion of this corollary follows in the same manner.\hfill
\end{proof}

\begin{proposition}
\label{p1-approx}Let $Q:\Omega\times\left[  0,T\right]  \rightarrow
\mathbb{R}_{+},$ $Q_{0}=0,$ be a strictly increasing continuous stochastic
process.\newline Let $\tau:\Omega\rightarrow\left[  0,\infty\right]  $\ be a
stopping time, $\eta:\Omega\rightarrow\mathbb{R}^{m}$\ is $\mathcal{F}_{\tau}%
$--measurable random variable such that $\mathbb{E}\left\vert \eta\right\vert
^{p}<\infty,$ $p>1,$\ and $\left(  \xi,\zeta\right)  \in S_{m}^{p}%
\times\Lambda_{m\times k}^{p}\left(  0,\mathbb{\infty}\right)  $\ is the
unique pair associated to $\eta$\ given by the martingale representation
formula (see \cite[\textit{Corollary 2.44}]{pa-ra/14})%
\[
\left\{
\begin{array}
[c]{l}%
\xi_{t}=\eta-\displaystyle{\int_{t}^{\infty}}\zeta_{s}dB_{s},\;t\geq
0,\;\text{\textit{a.s.,}}\smallskip\\
\xi_{t}=\mathbb{E}^{\mathcal{F}_{\tau}}\eta\quad\text{\textit{and}}\quad
\zeta_{t}=_{\left[  0,\tau\right]  }\left(  t\right)  \zeta_{t}%
\end{array}
\right.
\]
(or equivalently, $\xi_{t}=\eta-\displaystyle\int_{t\wedge\tau}^{\tau}%
\zeta_{s}dB_{s},\;t\geq0,\;$a.s.).\newline Let $U\in S_{m}^{p}$, $p>1$ be such
that
\[%
\begin{array}
[c]{ll}%
\left(  a\right)  & \mathbb{E}\sup\limits_{t\geq0}~\left\vert U_{t}\right\vert
^{p}<\infty,\medskip\\
\left(  b\right)  & \lim_{t\rightarrow\infty}\mathbb{E}~\left\vert U_{t}%
-\xi_{t}\right\vert ^{p}=0.
\end{array}
\]
Define%
\begin{equation}
M_{t}^{\varepsilon}=\mathbb{E}^{\mathcal{F}_{t}}%
{\displaystyle\int_{t\vee\varepsilon}^{\infty}}
\dfrac{1}{Q_{\varepsilon}}e^{-\frac{Q_{r}-Q_{t\vee\varepsilon}}{Q_{\varepsilon
}}}U_{r}dQ_{r},\quad t\geq0, \label{a5}%
\end{equation}
Then:

\noindent\textbf{I.}%
\begin{equation}%
\begin{array}
[c]{rl}%
\left(  j\right)  & \left\vert M_{t}^{\varepsilon}\right\vert \leq
\mathbb{E}^{\mathcal{F}_{t}}\sup\limits_{r\geq0}\left\vert U_{r}\right\vert
,\quad\text{a.s., for all }t\geq0,\medskip\\
\left(  jj\right)  & \mathbb{E}\sup\nolimits_{t\geq0}\left\vert M_{t}%
^{\varepsilon}\right\vert ^{p}\leq C_{p}\mathbb{E}\sup\limits_{r\geq
0}\left\vert U_{r}\right\vert ^{p}.
\end{array}
\label{a6}%
\end{equation}
Also for all $t\geq0$%
\begin{equation}
\left\vert M_{t}^{\varepsilon}-U_{t}\right\vert \leq\,\mathbb{E}%
^{\mathcal{F}_{t}}\left[  \exp\left(  2-\frac{1}{\sqrt{Q_{\varepsilon}}%
}\right)  \sup_{r\geq0}\left\vert U_{r}\right\vert +\sup_{r\geq0}\left\{
\left\vert U_{r}-U_{t}\right\vert :0\leq Q_{r}-Q_{t}\leq\sqrt{Q_{\varepsilon}%
}\right\}  \right]  \label{a6-1}%
\end{equation}
which yields%
\begin{equation}%
\begin{array}
[c]{rl}%
\left(  jjj\right)  & \lim\nolimits_{\varepsilon\rightarrow0}M_{t}%
^{\varepsilon}=U_{t}~,\quad\mathbb{P}-a.s.,\;\text{for all }t\geq0;\medskip\\
\left(  jv\right)  & \lim\nolimits_{\varepsilon\rightarrow0}\mathbb{E}%
\sup\nolimits_{t\in\left[  0,T\right]  }\left\vert M_{t}^{\varepsilon}%
-U_{t}\right\vert ^{p}=0,\quad\text{for all }T>0.
\end{array}
\label{a6-2}%
\end{equation}
\noindent\textbf{II. }$M^{\varepsilon}$ is the unique solution of the BSDE:
\begin{equation}
\left\{
\begin{array}
[c]{l}%
M_{t}^{\varepsilon}=M_{T}^{\varepsilon}+{%
{\displaystyle\int_{t}^{T}}
}1_{[\varepsilon,\infty)}\left(  r\right)  \dfrac{1}{Q_{\varepsilon}}\left(
U_{r}-M_{r}^{\varepsilon}\right)  dQ_{r}-{%
{\displaystyle\int_{t}^{T}}
}R_{r}^{\varepsilon}dB_{r}\,,\quad\text{for all }T>0,\;t\in\left[  0,T\right]
,\medskip\\
\lim\limits_{t\rightarrow\infty}\mathbb{E}~\left\vert M_{t}^{\varepsilon}%
-\xi_{t}\right\vert ^{p}=0.
\end{array}
\right.  \label{a4}%
\end{equation}
Moreover, we also have
\begin{equation}
\lim_{t\rightarrow\infty}~\mathbb{E}~\sup_{s\geq t}\left\vert U_{t}-\xi
_{t}\right\vert ^{p}=0\quad\Longrightarrow\quad\lim\limits_{t\rightarrow
\infty}\mathbb{E}\Big(\sup\limits_{s\geq t}\left\vert M_{s}^{\varepsilon}%
-\xi_{s}\right\vert ^{p}\Big)=0. \label{a8}%
\end{equation}
\noindent\textbf{III. }Let $\varphi:\mathbb{R}^{m}\rightarrow(-\infty
,+\infty]$ be a proper convex lower semicontinuous function such that
\[
\mathbb{E}%
{\displaystyle\int_{0}^{\infty}}
\left\vert \varphi\left(  U_{r}\right)  \right\vert dQ_{r}<\infty.
\]
Let $0\leq s\leq r\leq t$ and the stopping times $s^{\ast}=Q_{s}^{-1},$
$t^{\ast}=Q_{t}^{-1}$, $r^{\ast}=Q_{r}^{-1}$ , where $Q_{\cdot}^{-1}\left(
\omega\right)  $ is the inverse mapping of the function $r\longmapsto
Q_{r}\left(  \omega\right)  :[0,\infty)\rightarrow\lbrack0,\infty).$ Then%
\[
\lim_{\varepsilon\rightarrow0}\mathbb{E}\int_{s^{\ast}}^{t^{\ast}}%
\varphi\left(  M_{r}^{\varepsilon}\right)  dQ_{r}=\mathbb{E}\int_{s^{\ast}%
}^{t^{\ast}}\varphi\left(  U_{r}\right)  dQ_{r}.
\]
Moreover, if $g:\mathbb{R}^{m}\times\mathbb{R}^{n}\rightarrow\mathbb{R}_{+}$
is a continuous function, $D:\Omega\times\mathbb{R}_{+}\rightarrow
\mathbb{R}^{n}$ is a continuous stochastic process such that for all $R>0$
\[
\mathbb{E}%
{\displaystyle\int_{0}^{R}}
~\left\vert \varphi\left(  U_{r}\right)  \right\vert \sup_{\theta\in\left[
0,r\right]  }g\left(  U_{\theta},D_{\theta}\right)  dQ_{r}+\mathbb{E}%
{\displaystyle\int_{0}^{R}}
~\left\vert \varphi\left(  U_{r}\right)  \right\vert \sup_{\theta\in\left[
0,r\right]  }\sup_{0<\varepsilon\leq1}g\left(  M_{\theta}^{\varepsilon
},D_{\theta}\right)  dQ_{r}<\infty,
\]
then for all $0\leq T\leq\infty$%
\begin{equation}%
\begin{array}
[c]{cc}%
\left(  c_{1}\right)  & \mathbb{E}%
{\displaystyle\int_{T\wedge s^{\ast}}^{T\wedge t^{\ast}}}
g\left(  M_{r}^{\varepsilon},D_{r}\right)  \varphi\left(  M_{r}^{\varepsilon
}\right)  dQ_{r}\leq\mathbb{E}%
{\displaystyle\int_{T\wedge s^{\ast}}^{T\wedge t^{\ast}}}
g\left(  M_{r}^{\varepsilon},D_{r}\right)  \varphi\left(  U_{r}^{\varepsilon
}\right)  dQ_{r}~,\medskip\\
\left(  c_{2}\right)  & \lim\limits_{\varepsilon\rightarrow0}\mathbb{E}%
{\displaystyle\int_{T\wedge s^{\ast}}^{T\wedge t^{\ast}}}
g\left(  M_{r}^{\varepsilon},D_{r}\right)  \varphi\left(  M_{r}^{\varepsilon
}\right)  dQ_{r}=\mathbb{E}%
{\displaystyle\int_{T\wedge s^{\ast}}^{T\wedge t^{\ast}}}
g\left(  U_{r},D_{r}\right)  \varphi\left(  U_{r}\right)  dQ_{r}~,
\end{array}
\label{a9}%
\end{equation}
where
\[
U_{t}^{\varepsilon}=%
{\displaystyle\int_{t\vee\varepsilon}^{\infty}}
\frac{1}{Q_{\varepsilon}}e^{-\frac{Q_{r}-Q_{t\vee\varepsilon}}{Q_{\varepsilon
}}}U_{r}dQ_{r}~.
\]

\end{proposition}

\begin{proof}
Remark that%
\[
M_{t}^{\varepsilon}=\mathbb{E}^{\mathcal{F}_{t}}\left(  U_{t}^{\varepsilon
}\right)  ,\quad\text{for all }t\in\left[  0,T\right]  ,
\]
that yields (\ref{a6}-j). By Doob's inequality (see \cite[Theorem
1.60]{pa-ra/14}) from (\ref{a6}-j) we get the estimate (\ref{a6}-jj).

Clearly $\left\vert M_{t}^{\varepsilon}-U_{t}\right\vert \leq\,\mathbb{E}%
^{\mathcal{F}_{t}}\sup_{r\in\left[  0,T\right]  }\left\vert U_{t}%
^{\varepsilon}-U_{t}\right\vert $ and the conclusions (\ref{a6-1}) and
(\ref{a6-2}) hold by Lemma \ref{L1-approx} and Doob's inequality.

Let us to prove (\ref{a4}).

By representation theorem we have%
\begin{align*}%
{\displaystyle\int_{\varepsilon}^{\infty}}
\dfrac{1}{Q_{\varepsilon}}e^{-\frac{Q_{r}}{Q_{\varepsilon}}}U_{r}dQ_{r}  &
=\mathbb{E}^{\mathcal{F}_{t}}~%
{\displaystyle\int_{\varepsilon}^{\infty}}
\dfrac{1}{Q_{\varepsilon}}e^{-\frac{Q_{r}}{Q_{\varepsilon}}}U_{r}dQ_{r}+%
{\displaystyle\int_{t}^{\infty}}
\tilde{R}_{r}^{\varepsilon}dB_{r}\\
&  =e^{-\frac{Q_{t\vee\varepsilon}}{Q_{\varepsilon}}}M_{t}^{\varepsilon
}+\mathbb{E}^{\mathcal{F}_{t}}~%
{\displaystyle\int_{\varepsilon}^{t\vee\varepsilon}}
\dfrac{1}{Q_{\varepsilon}}e^{-\frac{Q_{r}}{Q_{\varepsilon}}}U_{r}dQ_{r}+%
{\displaystyle\int_{t}^{\infty}}
\tilde{R}_{r}^{\varepsilon}dB_{r}%
\end{align*}
that yields%
\begin{equation}
e^{-\frac{Q_{t\vee\varepsilon}}{Q_{\varepsilon}}}M_{t}^{\varepsilon}=%
{\displaystyle\int_{t}^{\infty}}
1_{[\varepsilon,\infty)}\left(  r\right)  \dfrac{1}{Q_{\varepsilon}}%
e^{-\frac{Q_{r}}{Q_{\varepsilon}}}U_{r}dQ_{r}-%
{\displaystyle\int_{t}^{\infty}}
\tilde{R}_{r}^{\varepsilon}dB_{r}. \label{a7-a}%
\end{equation}

Now by It\^{o}'s formula%
\[%
\begin{array}
[c]{l}%
\displaystyle M_{t}^{\varepsilon}=M_{T}^{\varepsilon}-\int_{t}^{T}d\left[
e^{\frac{Q_{r\vee\varepsilon}}{Q_{\varepsilon}}}\left(  e^{-\frac
{Q_{r\vee\varepsilon}}{Q_{\varepsilon}}}M_{r}^{\varepsilon}\right)  \right]
\medskip\\
\displaystyle=M_{T}^{\varepsilon}-\int_{t}^{T}1_{[\varepsilon,\infty)}\left(
r\right)  \dfrac{1}{Q_{\varepsilon}}e^{\frac{Q_{r\vee\varepsilon}%
}{Q_{\varepsilon}}}\left(  e^{-\frac{Q_{r\vee\varepsilon}}{Q_{\varepsilon}}%
}M_{r}^{\varepsilon}\right)  dQ_{r}-\int_{t}^{T}e^{\frac{Q_{r\vee\varepsilon}%
}{Q_{\varepsilon}}}d\left(  e^{-\frac{Q_{r\vee\varepsilon}}{Q_{\varepsilon}}%
}M_{r}^{\varepsilon}\right)  \medskip\\
\displaystyle=M_{T}^{\varepsilon}-\int_{t}^{T}1_{[\varepsilon,\infty)}\left(
r\right)  \dfrac{1}{Q_{\varepsilon}}M_{r}^{\varepsilon}dQ_{r}+\int_{t}%
^{T}e^{\frac{Q_{r\vee\varepsilon}}{Q_{\varepsilon}}}1_{[\varepsilon,\infty
)}\left(  r\right)  \dfrac{1}{Q_{\varepsilon}}e^{-\frac{Q_{r\vee\varepsilon}%
}{Q_{\varepsilon}}}U_{r}dQ_{r}-%
{\displaystyle\int_{t}^{T}}
e^{\frac{Q_{r\vee\varepsilon}}{Q_{\varepsilon}}}\tilde{R}_{r}^{\varepsilon
}dB_{r}\medskip\\
\displaystyle=M_{T}^{\varepsilon}+\int_{t}^{T}1_{[\varepsilon,\infty)}\left(
r\right)  \dfrac{1}{Q_{\varepsilon}}\left(  U_{r}-M_{r}^{\varepsilon}\right)
dQ_{r}-%
{\displaystyle\int_{t}^{T}}
R_{r}^{\varepsilon}dB_{r}\,,
\end{array}
\]
where $R_{r}^{\varepsilon}=e^{\frac{Q_{r\vee\varepsilon}}{Q_{\varepsilon}}%
}\tilde{R}_{r}^{\varepsilon}.$

The convergence result from (\ref{a4}) is obtained as follows:%
\begin{align*}
M_{t}^{\varepsilon}-\xi_{t}  &  =\mathbb{E}^{\mathcal{F}_{t}}%
{\displaystyle\int_{t\vee\varepsilon}^{\infty}}
\dfrac{1}{Q_{\varepsilon}}e^{-\frac{Q_{r}-Q_{t\vee\varepsilon}}{Q_{\varepsilon
}}}\left(  U_{r}-\xi_{r}\right)  dQ_{r}+\mathbb{E}^{\mathcal{F}_{t}}~%
{\displaystyle\int_{t\vee\varepsilon}^{\infty}}
\dfrac{1}{Q_{\varepsilon}}e^{-\frac{Q_{r}-Q_{t\vee\varepsilon}}{Q_{\varepsilon
}}}\left(  \xi_{r}-\xi_{t}\right)  dQ_{r}\\
&  =\mathbb{E}^{\mathcal{F}_{t}}%
{\displaystyle\int_{0}^{\infty}}
e^{-s}\left(  U_{Q^{-1}\left(  sQ_{\varepsilon}+Q_{t\vee\varepsilon}\right)
}-\xi_{Q^{-1}\left(  sQ_{\varepsilon}+Q_{t\vee\varepsilon}\right)  }\right)
ds+\mathbb{E}^{\mathcal{F}_{t}}\sup_{r\geq t}\left\vert \xi_{r}-\xi
_{t}\right\vert
\end{align*}
Here first by Jensen's inequality and then by Burkholder--Davis--Gundy
inequality (see \cite[Corollary 2.9]{pa-ra/14}) we have%
\[
\left(  \mathbb{E}^{\mathcal{F}_{t}}\sup_{r\geq t}\left\vert \xi_{r}-\xi
_{t}\right\vert \right)  ^{p}\leq\left(  \mathbb{E}^{\mathcal{F}_{t}}%
\sup_{r\geq t}\left\vert {\int_{t}^{r}}\zeta_{s}dB_{s}\right\vert \right)
^{p}\leq C_{p~}\mathbb{E}^{\mathcal{F}_{t}}\left(  \int_{t}^{\infty}\left\vert
\zeta_{s}\right\vert ^{2}ds\right)  ^{p/2}.
\]
Hence%
\[%
\begin{array}
[c]{l}%
\displaystyle\mathbb{E~}\left\vert M_{t}^{\varepsilon}-\xi_{t}\right\vert
^{p}\medskip\\
\displaystyle\leq2^{p-1}~\mathbb{E}\left(  \mathbb{E}^{\mathcal{F}_{t}}~%
{\displaystyle\int_{0}^{\infty}}
e^{-s}\left\vert U_{Q^{-1}\left(  sQ_{\varepsilon}+Q_{t\vee\varepsilon
}\right)  }-\xi_{Q^{-1}\left(  sQ_{\varepsilon}+Q_{t\vee\varepsilon}\right)
}\right\vert ds\right)  ^{p}+2^{p-1}~\mathbb{E~}\left(  \mathbb{E}%
^{\mathcal{F}_{t}}\sup_{r\geq t}\left\vert \xi_{r}-\xi_{t}\right\vert \right)
^{p}\medskip\\
\displaystyle\leq2^{p-1}%
{\displaystyle\int_{0}^{\infty}}
e^{-s}\mathbb{E}\left\vert U_{Q^{-1}\left(  sQ_{\varepsilon}+Q_{t\vee
\varepsilon}\right)  }-\xi_{Q^{-1}\left(  sQ_{\varepsilon}+Q_{t\vee
\varepsilon}\right)  }\right\vert ^{p}ds+C_{p}\mathbb{E~}\left(  \int
_{t}^{\infty}\left\vert \zeta_{s}\right\vert ^{2}ds\right)  ^{p/2}%
\end{array}
\]
and using the Lebesgue dominated convergence theorem we get%
\[
\lim_{t\rightarrow\infty}\mathbb{E~}\left\vert M_{t}^{\varepsilon}-\xi
_{t}\right\vert ^{p}=0.
\]
To prove (\ref{a8}) we have for $t\geq T>\varepsilon$ and $1<q<p:$
\[
\left\vert M_{t}^{\varepsilon}-\xi_{t}\right\vert ^{p}\leq2^{p-1}%
\mathbb{E}^{\mathcal{F}_{t}}\sup_{r\geq T}\left\vert U_{r}-\xi_{r}\right\vert
^{p}+2^{p-1}\left(  \mathbb{E}^{\mathcal{F}_{t}}\sup_{r\geq t}\left\vert
\xi_{r}-\xi_{t}\right\vert ^{q}\right)  ^{p/q}%
\]
and consequently by Doob's inequality,
\begin{align*}
\mathbb{E~}\sup_{t\geq T}\left\vert M_{t}^{\varepsilon}-\xi_{t}\right\vert
^{p}  &  \leq2^{p-1}~\mathbb{E~}\sup_{r\geq T}\left\vert U_{r}-\xi
_{r}\right\vert ^{p}+C_{p,q}\mathbb{E~}\sup_{t\geq T}~\left[  \mathbb{E}%
^{\mathcal{F}_{t}}~\left(  \int_{T}^{\infty}\left\vert \zeta_{s}\right\vert
^{2}ds\right)  ^{q/2}\right]  ^{p/q}\\
&  \leq C_{p}~\mathbb{E~}\sup_{r\geq T}\left\vert U_{r}-\xi_{r}\right\vert
^{p}+C_{p,q}^{\prime}~\mathbb{E~}\left(  \int_{T}^{\infty}\left\vert \zeta
_{s}\right\vert ^{2}ds\right)  ^{p/2}%
\end{align*}
that yields (\ref{a8}).

Finally%
\[%
\begin{array}
[c]{l}%
\displaystyle\mathbb{E}\int_{T\wedge s^{\ast}}^{T\wedge t^{\ast}}g\left(
M_{r}^{\varepsilon},D_{r}\right)  \varphi\left(  M_{r}^{\varepsilon}\right)
dQ_{r}=\mathbb{E}\int_{s^{\ast}}^{t^{\ast}}\mathbf{1}_{\left[  0,T\right]
}\left(  r\right)  g\left(  M_{r}^{\varepsilon},D_{r}\right)  \varphi\left(
\mathbb{E}^{\mathcal{F}_{r}}\left(  U_{r}^{\varepsilon}\right)  \right)
dQ_{r}\medskip\\
\displaystyle\leq\mathbb{E}\int_{s^{\ast}}^{t^{\ast}}\mathbb{E}^{\mathcal{F}%
_{r}}\left[  \mathbf{1}_{\left[  0,T\right]  }\left(  r\right)  g\left(
M_{r}^{\varepsilon},D_{r}\right)  \varphi\left(  U_{r}^{\varepsilon}\right)
\right]  dQ_{r}=\mathbb{E}\int_{s}^{t}\mathbb{E}^{\mathcal{F}_{r^{\ast}}%
}\left[  \mathbf{1}_{\left[  0,T\right]  }\left(  r^{\ast}\right)  g\left(
M_{r^{\ast}}^{\varepsilon},D_{r^{\ast}}\right)  \varphi\left(  U_{r^{\ast}%
}^{\varepsilon}\right)  \right]  dr\medskip\\
\displaystyle=\int_{s}^{t}\mathbb{E}\left[  \mathbf{1}_{\left[  0,T\right]
}\left(  r^{\ast}\right)  g\left(  M_{r^{\ast}}^{\varepsilon},D_{r^{\ast}%
}\right)  \varphi\left(  U_{r^{\ast}}^{\varepsilon}\right)  \right]
dr=\mathbb{E}\int_{s}^{t}\mathbf{1}_{\left[  0,T\right]  }\left(  r^{\ast
}\right)  g\left(  M_{r^{\ast}}^{\varepsilon},D_{r^{\ast}}\right)
\varphi\left(  U_{r^{\ast}}^{\varepsilon}\right)  dr\medskip\\
\displaystyle=\mathbb{E}\int_{s^{\ast}}^{t^{\ast}}\mathbf{1}_{\left[
0,T\right]  }\left(  r\right)  g\left(  M_{r}^{\varepsilon},D_{r}\right)
\varphi\left(  U_{r}^{\varepsilon}\right)  dQ_{r}=\mathbb{E}\int_{T\wedge
s^{\ast}}^{T\wedge t^{\ast}}g\left(  M_{r}^{\varepsilon},D_{r}\right)
\varphi\left(  U_{r}^{\varepsilon}\right)  dQ_{r}%
\end{array}
\]
and as in the proof of Corollary \ref{C1-approx} we have%
\[%
\begin{array}
[c]{l}%
\displaystyle\mathbb{E}\int_{T\wedge s^{\ast}}^{T\wedge t^{\ast}}g\left(
M_{r}^{\varepsilon},D_{r}\right)  \varphi\left(  U_{r}^{\varepsilon}\right)
dQ_{r}\medskip\\
\displaystyle=\mathbb{E}\int_{0}^{\infty}\varphi\left(  U_{\theta}\right)
\mathbf{1}_{[\varepsilon,\infty)}\left(  \theta\right)  \left(  \int
_{0}^{\theta}\mathbf{1}_{\left[  T\wedge s^{\ast},T\wedge t^{\ast}\right]
}\left(  r\right)  g\left(  M_{r}^{\varepsilon},D_{r}\right)  \frac
{1}{Q_{\varepsilon}}e^{-\frac{Q_{\theta}-Q_{r\vee\varepsilon}}{Q_{\varepsilon
}}}dQ_{r}\right)  dQ_{\theta}\medskip\\
\displaystyle\leq\mathbb{E}\int_{0}^{\infty}\varphi\left(  U_{\theta}\right)
\mathbf{1}_{[\varepsilon,\infty)}\left(  \theta\right)  \sup_{r\in\left[
0,\theta\right]  }\left\vert g\left(  M_{r}^{\varepsilon},D_{r}\right)
-g\left(  U_{r},D_{r}\right)  \right\vert dQ_{\theta}\medskip\\
\displaystyle+\mathbb{E}\int_{0}^{\infty}\varphi\left(  U_{\theta}\right)
\mathbf{1}_{[\varepsilon,\infty)}\left(  \theta\right)  \left(  \int
_{0}^{\theta}\mathbf{1}_{\left[  T\wedge s^{\ast},T\wedge t^{\ast}\right]
}\left(  r\right)  g\left(  U_{r},D_{r}\right)  \frac{1}{Q_{\varepsilon}%
}e^{-\frac{Q_{\theta}-Q_{r\vee\varepsilon}}{Q_{\varepsilon}}}dQ_{r}\right)
dQ_{\theta}~.
\end{array}
\]
Now by Fatou's Lemma we have%
\[%
\begin{array}
[c]{l}%
\displaystyle\mathbb{E}\int_{T\wedge s^{\ast}}^{T\wedge t^{\ast}}g\left(
U_{r},D_{r}\right)  \varphi\left(  U_{r}\right)  dQ_{r}\leq\liminf
_{\varepsilon\rightarrow0_{+}}\mathbb{E}\int_{T\wedge s^{\ast}}^{T\wedge
t^{\ast}}g\left(  M_{r}^{\varepsilon},D_{r}\right)  \varphi\left(
M_{r}^{\varepsilon}\right)  dQ_{r}\medskip\\
\displaystyle\leq\liminf_{\varepsilon\rightarrow0_{+}}\mathbb{E}\int_{T\wedge
s^{\ast}}^{T\wedge t^{\ast}}g\left(  M_{r}^{\varepsilon},D_{r}\right)
\varphi\left(  U_{r}^{\varepsilon}\right)  dQ_{r}\leq\limsup_{\varepsilon
\rightarrow0_{+}}\mathbb{E}\int_{T\wedge s^{\ast}}^{T\wedge t^{\ast}}g\left(
M_{r}^{\varepsilon},D_{r}\right)  \varphi\left(  U_{r}^{\varepsilon}\right)
dQ_{r}\medskip\\
\displaystyle\leq\limsup_{\varepsilon\rightarrow0_{+}}\mathbb{E}\int
_{0}^{\infty}\varphi\left(  U_{\theta}\right)  \mathbf{1}_{[\varepsilon
,\infty)}\left(  \theta\right)  \left(  \int_{0}^{\theta}\mathbf{1}_{\left[
T\wedge s^{\ast},T\wedge t^{\ast}\right]  }\left(  r\right)  g\left(
U_{r},D_{r}\right)  \frac{1}{Q_{\varepsilon}}e^{-\frac{Q_{\theta}%
-Q_{r\vee\varepsilon}}{Q_{\varepsilon}}}dQ_{r}\right)  dQ_{\theta}\medskip\\
\displaystyle\leq\mathbb{E}\int_{T\wedge s^{\ast}}^{T\wedge t^{\ast}}g\left(
U_{\theta},D_{\theta}\right)  \varphi\left(  U_{\theta}\right)  dQ_{\theta}%
\end{array}
\]
and the convergence result follows.\hfill
\end{proof}

\subsection{Mollifier approximation\label{Annex-MA}}

\noindent Let $F:\Omega\times\mathbb{R}_{+}\times\mathbb{R}^{m}\times
\mathbb{R}^{m\times k}\rightarrow\mathbb{R}^{m}$\ and $G:\Omega\times
\mathbb{R}_{+}\times\mathbb{R}^{m}\rightarrow\mathbb{R}^{m}$ be such that
assumptions $\left(  \mathrm{A}_{4}\right)  $ and $\left(  \mathrm{A}%
_{5}\right)  $ are satisfied.$\medskip$

\noindent Let $\rho\in C_{0}^{\infty}\left(  \mathbb{R}^{m};\mathbb{R}%
_{+}\right)  $ such that $\rho\left(  y\right)  =0$ if $\left\vert
y\right\vert \geq1$ and $\int_{\mathbb{R}^{d}}\rho\left(  y\right)
dy=1$.$\medskip$

\noindent Let $\kappa\mathbf{1}_{\overline{B\left(  0,1\right)  }}\left(
y\right)  \geq\left\vert \nabla_{y}\rho\left(  y\right)  \right\vert ,$ for
all $y\in\mathbb{R}^{m}.\medskip$

Define, for $0<\varepsilon\leq1,$%
\begin{equation}%
\begin{array}
[c]{l}%
F_{\varepsilon}\left(  t,y,z\right)  =%
{\displaystyle\int_{\overline{B\left(  0,1\right)  }}}
F\left(  t,y-\varepsilon u,\beta_{\varepsilon}\left(  z\right)  \right)
\mathbf{1}_{\left[  0,1\right]  }\left(  \varepsilon\left\vert F\left(
t,y-\varepsilon u\right)  ,0\right\vert \right)  \rho\left(  u\right)
du\medskip\\
=\dfrac{1}{\varepsilon^{m+1}}%
{\displaystyle\int_{\mathbb{R}^{m}}}
\varepsilon~F\left(  t,u,\beta_{\varepsilon}\left(  z\right)  \right)
\mathbf{1}_{\left[  0,1\right]  }\left(  \varepsilon\left\vert F\left(
t,u,0\right)  \right\vert \right)  \rho\left(  \dfrac{y-u}{\varepsilon
}\right)  du,
\end{array}
\label{sde9c}%
\end{equation}
where%
\[
\beta_{\varepsilon}\left(  z\right)  =\frac{z}{1\vee\left(  \varepsilon
\left\vert z\right\vert \right)  }=\Pr\nolimits_{\overline{B\left(
0,1/\varepsilon\right)  }}\left(  z\right)  .
\]
Clearly for all $y,u\in\mathbb{R}^{m}$, $\left\vert u\right\vert \leq1$ and
$z\in\mathbb{R}^{m\times k}$,%
\[
\left\vert F\left(  t,y-\varepsilon u,\beta_{\varepsilon}\left(  z\right)
\right)  \right\vert \leq\ell_{t}\left\vert z\right\vert +F_{\left\vert
y\right\vert +1}^{\#}\left(  t\right)
\]
and consequently%
\begin{equation}
\left\vert F_{\varepsilon}\left(  t,y,z\right)  \right\vert \leq\ell
_{t}\left\vert z\right\vert +F_{\left\vert y\right\vert +1}^{\#}\left(
t\right)  \quad\text{and}\quad\left\vert F_{\varepsilon}\left(  t,0,0\right)
\right\vert \leq F_{1}^{\#}\left(  t\right)  . \label{ma-0}%
\end{equation}
It is easy to prove that this mollifier approximation of $F$ satisfies the
following properties:%
\begin{equation}%
\begin{array}
[c]{ll}%
\left(  a\right)  & \left\vert F_{\varepsilon}\left(  t,y,z\right)
\right\vert \leq\ell_{t}\beta_{\varepsilon}\left(  z\right)  +\dfrac
{1}{\varepsilon}\leq\dfrac{1}{\varepsilon}\left(  1+\ell_{t}\right)
,\medskip\\
\left(  b\right)  & \left\vert F_{\varepsilon}\left(  t,y,z\right)
-F_{\varepsilon}\left(  t,y,\hat{z}\right)  \right\vert \leq\ell_{t}\left\vert
z-\hat{z}\right\vert \medskip\\
\left(  c\right)  & \left\vert F_{\varepsilon}\left(  t,y,z\right)
-F_{\varepsilon}\left(  t,\hat{y},z\right)  \right\vert \leq\dfrac{\kappa
}{\varepsilon}\left\vert y-\hat{y}\right\vert \left[  \ell_{t}\left\vert
\beta_{\varepsilon}\left(  z\right)  \right\vert +\dfrac{1}{\varepsilon
}\right]  \leq\dfrac{\kappa\left(  1+\ell_{t}\right)  }{\varepsilon^{2}%
}\left\vert y-\hat{y}\right\vert .
\end{array}
\label{ma-1}%
\end{equation}
Also we have for all $y,\hat{y}\in\mathbb{R}^{m}$, $\left\vert \hat
{y}\right\vert \leq\rho:$%
\begin{equation}%
\begin{array}
[c]{l}%
\left\langle y-\hat{y},F_{\varepsilon}\left(  t,y,z\right)  \right\rangle
\leq\;\mu_{t}^{+}\left\vert y-\hat{y}\right\vert ^{2}+\left\vert y-\hat
{y}\right\vert \left[  F_{\rho+1}^{\#}\left(  t\right)  +\ell_{t}\left\vert
z\right\vert \right]  \medskip\\
\leq\left\vert y-\hat{y}\right\vert F_{\rho+1}^{\#}\left(  t\right)  +\left(
\mu_{t}+\dfrac{1}{2n_{p}\lambda}\ell_{t}^{2}\mathbf{1}_{z\neq0}\right)
^{+}\left\vert y-\hat{y}\right\vert ^{2}+\dfrac{n_{p}\lambda}{2}\left\vert
z\right\vert ^{2},\quad\text{for all }\lambda>0.
\end{array}
\label{ma-2}%
\end{equation}
where $p>1,$ $n_{p}=\left(  p-1\right)  \wedge1$.

Indeed, by taking%
\[
\alpha_{\varepsilon}\left(  t,y\right)  =%
{\displaystyle\int_{\overline{B\left(  0,1\right)  }}}
\mathbf{1}_{\left[  0,1\right]  }\left(  \varepsilon\left\vert F\left(
t,y-\varepsilon u\right)  ,0\right\vert \right)  \rho\left(  u\right)  du,
\]
we have $0\leq\alpha_{\varepsilon}\left(  t,y\right)  \leq1$ and%
\begin{align*}
&  \left\langle y-\hat{y},F_{\varepsilon}\left(  t,y,z\right)  \right\rangle
\\
&  =%
{\displaystyle\int_{\overline{B\left(  0,1\right)  }}}
\left\langle y-\hat{y},F\left(  t,y-\varepsilon u,\beta_{\varepsilon}\left(
z\right)  \right)  -F\left(  t,\hat{y}-\varepsilon u,\beta_{\varepsilon
}\left(  z\right)  \right)  \right\rangle \mathbf{1}_{\left[  0,1\right]
}\left(  \varepsilon\left\vert F\left(  t,y-\varepsilon u\right)
,0\right\vert \right)  \rho\left(  u\right)  du\medskip\\
&  \quad+%
{\displaystyle\int_{\overline{B\left(  0,1\right)  }}}
\left\langle y-\hat{y},F\left(  t,\hat{y}-\varepsilon u,\beta_{\varepsilon
}\left(  z\right)  \right)  -F\left(  t,\hat{y}-\varepsilon u,0\right)
\right\rangle \mathbf{1}_{\left[  0,1\right]  }\left(  \varepsilon\left\vert
F\left(  t,y-\varepsilon u\right)  ,0\right\vert \right)  \rho\left(
u\right)  du\\
&  \quad+%
{\displaystyle\int_{\overline{B\left(  0,1\right)  }}}
\left\langle y-\hat{y},F\left(  t,\hat{y}-\varepsilon u,0\right)
\right\rangle \mathbf{1}_{\left[  0,1\right]  }\left(  \varepsilon\left\vert
F\left(  t,y-\varepsilon u\right)  ,0\right\vert \right)  \rho\left(
u\right)  du\\
&  \leq\left[  \mu_{t}\left\vert y-\hat{y}\right\vert ^{2}+\left\vert
y-\hat{y}\right\vert \ell_{t}\left\vert \beta_{\varepsilon}\left(  z\right)
\right\vert \right]  \alpha_{\varepsilon}\left(  t,y\right)  +\left\vert
y-\hat{y}\right\vert F_{\rho+1}^{\#}\left(  t\right)  .
\end{align*}
Moreover, for all $y,\hat{y}\in\mathbb{R}^{m}$, $\left\vert y\right\vert
\leq\rho$, $\left\vert \hat{y}\right\vert \leq\rho:$%
\begin{equation}%
\begin{array}
[c]{ll}%
\left(  a\right)  & \left\langle y-\hat{y},F_{\varepsilon}\left(
t,y,z\right)  -F_{\varepsilon}\left(  t,\hat{y},z\right)  \right\rangle
\leq\mu_{t}^{+}\left\vert y-\hat{y}\right\vert ^{2}\medskip\\
& \quad+\left\vert y-\hat{y}\right\vert \left[  F_{\rho+1}^{\#}\left(
t\right)  +\ell_{t}\left\vert z\right\vert \right]  \mathbf{1}_{[\frac
{1}{\varepsilon},\infty)}(F_{\rho+1}^{\#}\left(  t\right)  )\medskip\\
\left(  b\right)  & \left\langle y-\hat{y},F_{\varepsilon}\left(
t,y,z\right)  -F_{\varepsilon}\left(  t,\hat{y},\hat{z}\right)  \right\rangle
\leq\left\vert y-\hat{y}\right\vert \left[  F_{\rho+1}^{\#}\left(  t\right)
+\ell_{t}~\left\vert \hat{z}\right\vert \right]  \mathbf{1}_{[\frac
{1}{\varepsilon},\infty)}(F_{\rho+1}^{\#}\left(  t\right)  )\medskip\\
& \quad+\left(  \mu_{t}+\dfrac{1}{2n_{p}\lambda}\ell_{t}^{2}\mathbf{1}%
_{z\neq\hat{z}}\right)  ^{+}\left\vert y-\hat{y}\right\vert ^{2}+\dfrac
{n_{p}\lambda}{2}\left\vert z-\hat{z}\right\vert ^{2},\quad\text{for all
}\lambda>0.\medskip\\
\left(  c\right)  & \left\langle y-\hat{y},F_{\varepsilon}\left(
t,y,z\right)  -F_{\delta}\left(  t,\hat{y},\hat{z}\right)  \right\rangle
\leq\left\vert \varepsilon-\delta\right\vert \left[  \mu_{t}^{+}\left\vert
\varepsilon-\delta\right\vert +2F_{\rho+1}^{\#}\left(  t\right)  +2\ell
_{t}\left\vert z\right\vert \right]  \medskip\\
& \quad+\left\vert y-\hat{y}\right\vert \left[  2\left\vert \mu_{t}\right\vert
\left\vert \varepsilon-\delta\right\vert +\ell_{t}\left\vert \hat
{z}\right\vert \mathbf{1}_{[\frac{1}{\varepsilon}\wedge\frac{1}{\delta}%
,\infty)}\left(  \left\vert \hat{z}\right\vert \right)  \mathbf{1}%
_{\varepsilon\neq\delta}\right]  \medskip\\
& \quad\quad\quad\quad\quad\quad+(F_{\rho+1}^{\#}\left(  t\right)  +\ell
_{t}\left\vert \hat{z}\right\vert )\mathbf{1}_{[\frac{1}{\varepsilon}%
\wedge\frac{1}{\delta},\infty)}(F_{\rho+1}^{\#}\left(  t\right)
)\Big]\medskip\\
& \quad+\left(  \mu_{t}+\dfrac{1}{2n_{p}\lambda}\ell_{t}^{2}\mathbf{1}%
_{z\neq\hat{z}}\right)  ^{+}\left\vert y-\hat{y}\right\vert ^{2}+\dfrac
{n_{p}\lambda}{2}\left\vert z-\hat{z}\right\vert ^{2}.
\end{array}
\label{ma-3}%
\end{equation}
It is sufficient to prove (\ref{ma-3}-c); inequalities (\ref{ma-3}-a,b) are
obtained by particularization of (\ref{ma-3}-c). We have
\begin{align*}
&  \left\langle y-\hat{y},F_{\varepsilon}\left(  t,y,z\right)  -F_{\delta
}\left(  t,\hat{y},\hat{z}\right)  \right\rangle \\
&  \leq%
{\displaystyle\int_{\overline{B\left(  0,1\right)  }}}
\left\langle y-\varepsilon u-\left(  \hat{y}-\delta u\right)  +\left(
\varepsilon-\delta\right)  u,F\left(  t,y-\varepsilon u,\beta_{\varepsilon
}\left(  z\right)  \right)  -F\left(  t,\hat{y}-\delta u,\beta_{\varepsilon
}\left(  z\right)  \right)  \right\rangle \cdot\\
&  \quad\quad\quad\quad\quad\quad\quad\quad\quad\quad\quad\quad\quad\quad
\quad\quad\quad\quad\quad\quad\quad\quad\quad\quad\quad\quad\cdot
\mathbf{1}_{\left[  0,1\right]  }\left(  \varepsilon\left\vert F\left(
t,y-\varepsilon u\right)  ,0\right\vert \right)  \rho\left(  u\right)  du\\
&  \quad+%
{\displaystyle\int_{\overline{B\left(  0,1\right)  }}}
\left\langle y-\hat{y},F\left(  t,\hat{y}-\delta u,\beta_{\varepsilon}\left(
z\right)  \right)  -F\left(  t,\hat{y}-\delta u,\beta_{\delta}\left(  \hat
{z}\right)  \right)  \right\rangle \mathbf{1}_{\left[  0,1\right]  }\left(
\varepsilon\left\vert F\left(  t,y-\varepsilon u\right)  ,0\right\vert
\right)  \rho\left(  u\right)  du\\
&  \quad+%
{\displaystyle\int_{\overline{B\left(  0,1\right)  }}}
\left\langle y-\hat{y},F\left(  t,\hat{y}-\delta u,\beta_{\delta}\left(
\hat{z}\right)  \right)  \right\rangle \cdot\\
&  \quad\quad\quad\quad\quad\quad\quad\quad\quad\quad\quad\quad\quad
~\cdot\left[  \mathbf{1}_{\left[  0,1\right]  }\left(  \varepsilon\left\vert
F\left(  t,y-\varepsilon u\right)  ,0\right\vert \right)  -\mathbf{1}_{\left[
0,1\right]  }\left(  \delta\left\vert F\left(  t,\hat{y}-\delta u\right)
,0\right\vert \right)  \right]  \rho\left(  u\right)  du\\
&  \leq\mu_{t}\left\vert y-\varepsilon u-\left(  \hat{y}-\delta u\right)
\right\vert ^{2}\alpha_{\varepsilon}\left(  t,y\right)  +2\left\vert
\varepsilon-\delta\right\vert \left[  F_{\rho+1}^{\#}\left(  t\right)
+\ell_{t}\left\vert \beta_{\varepsilon}\left(  z\right)  \right\vert \right]
\\
&  \quad+\left\vert y-\hat{y}\right\vert \mathcal{\ell}_{t}\left\vert
\beta_{\varepsilon}\left(  z\right)  -\beta_{\delta}\left(  \hat{z}\right)
\right\vert \alpha_{\varepsilon}\left(  t,y\right) \\
&  \quad+\left\vert y-\hat{y}\right\vert \left[  F_{\rho+1}^{\#}\left(
t\right)  +\ell_{t}\left\vert \beta_{\delta}\left(  \hat{z}\right)
\right\vert \right]  \mathbf{1}_{[\frac{1}{\varepsilon}\wedge\frac{1}{\delta
},\infty)}(F_{\rho+1}^{\#}\left(  t\right)  ).
\end{align*}
But%
\begin{equation}%
\begin{array}
[c]{ll}%
\left(  a\right)  & \mu_{t}\left\vert y-\varepsilon u-\left(  \hat{y}-\delta
u\right)  \right\vert ^{2}\alpha_{\varepsilon}\left(  t,y\right)  \leq\mu
_{t}\left\vert y-\hat{y}\right\vert ^{2}\alpha_{\varepsilon}\left(
t,y\right)  +2\left\vert \mu_{t}\right\vert \left\vert y-\hat{y}\right\vert
\left\vert \varepsilon-\delta\right\vert +\mu_{t}^{+}\left\vert \varepsilon
-\delta\right\vert ^{2}\medskip\\
\left(  b\right)  & \left\vert \beta_{\varepsilon}\left(  z\right)
\right\vert \leq\left\vert z\right\vert \wedge\dfrac{1}{\varepsilon}%
\leq\left\vert z\right\vert \medskip\\
\left(  c\right)  & \left\vert \beta_{\varepsilon}\left(  z\right)
-\beta_{\delta}\left(  \hat{z}\right)  \right\vert \leq\left\vert
\beta_{\varepsilon}\left(  z\right)  -\beta_{\varepsilon}\left(  \hat
{z}\right)  \right\vert +\left\vert \beta_{\varepsilon}\left(  \hat{z}\right)
-\beta_{\delta}\left(  \hat{z}\right)  \right\vert \leq\left\vert z-\hat
{z}\right\vert +\left\vert \hat{z}\right\vert \mathbf{1}_{[\frac
{1}{\varepsilon}\wedge\frac{1}{\delta},\infty)}\left(  \left\vert \hat
{z}\right\vert \right)  \mathbf{1}_{\varepsilon\neq\delta}\,.
\end{array}
\label{ma-4}%
\end{equation}
Hence%
\[%
\begin{array}
[c]{l}%
\displaystyle\left\langle y-\hat{y},F_{\varepsilon}\left(  t,y,z\right)
-F_{\delta}\left(  t,\hat{y},\hat{z}\right)  \right\rangle \leq\left(  \mu
_{t}+\frac{1}{2n_{p}\lambda}\ell_{t}^{2}\mathbf{1}_{z\neq\hat{z}}\right)
^{+}\left\vert y-\hat{y}\right\vert ^{2}+\frac{n_{p}\lambda}{2}\left\vert
z-\hat{z}\right\vert ^{2}\medskip\\
\displaystyle\quad+\left\vert y-\hat{y}\right\vert \left[  2\left\vert \mu
_{t}\right\vert \left\vert \varepsilon-\delta\right\vert +\ell_{t}\left\vert
\hat{z}\right\vert \mathbf{1}_{[\frac{1}{\varepsilon}\wedge\frac{1}{\delta
},\infty)}\left(  \left\vert \hat{z}\right\vert \right)  \mathbf{1}%
_{\varepsilon\neq\delta}+(F_{\rho+1}^{\#}\left(  t\right)  +\ell_{t}\left\vert
\hat{z}\right\vert )\mathbf{1}_{[\frac{1}{\varepsilon}\wedge\frac{1}{\delta
},\infty)}(F_{\rho+1}^{\#}\left(  t\right)  )\right]  \medskip\\
\displaystyle\quad+\mu_{t}^{+}\left\vert \varepsilon-\delta\right\vert
^{2}+2\left\vert \varepsilon-\delta\right\vert \left[  F_{\rho+1}^{\#}\left(
t\right)  +\ell_{t}\left\vert z\right\vert \right]  .
\end{array}
\]

\begin{remark}
The function $G$ will be approximate in the same manner. For $0<\varepsilon
\leq1:$%
\begin{equation}%
\begin{array}
[c]{l}%
G_{\varepsilon}\left(  t,y\right)  =%
{\displaystyle\int_{\overline{B\left(  0,1\right)  }}}
G\left(  t,y-\varepsilon u\right)  \mathbf{1}_{\left[  0,1\right]  }\left(
\varepsilon\left\vert G\left(  t,y-\varepsilon u\right)  \right\vert \right)
\rho\left(  u\right)  du.
\end{array}
\label{ma-5}%
\end{equation}
Similar properties (\ref{ma-1}) and (\ref{ma-3})are satisfied with $z=\hat
{z}=0$ and $\ell=0.$
\end{remark}

\addcontentsline{toc}{section}{References}

\end{document}